\documentclass[a4paper, reqno]{amsart}
\usepackage{amssymb,graphicx, enumerate, enumitem, color,hyperref}
\usepackage{fullpage}
\usepackage{afterpage}
\usepackage{varwidth}

\newtheorem{theorem}{Theorem}
\newtheorem{claim}[theorem]{Claim}

\newtheorem{lemma}[theorem]{Lemma}

\theoremstyle{remark}
\newtheorem{obs}[theorem]{Observation}

\newcommand{\calR}{\mathcal{R}}
\newcommand{\LL}{\mathcal{L}}

\newcommand{\MM}{\operatorname{MM}}

\newcommand{\set}[1]{\{#1\}}
\newcommand{\norm}[1]{{|#1|}}
\newcommand{\sigtree}{\Psi}

\newcommand{\q}[1]{
\def\result{??}
\def\questionnode{#1}
\def\leftorright{sign:left-or-right}
\ifx \questionnode \leftorright
\def\result{1}
\fi
\def\asmallerthanx{sign:a<x}
\ifx \questionnode \asmallerthanx
\def\result{2}
\fi
\def\signagoesbelowu{sign:a-goes-below-u}
\ifx \questionnode \signagoesbelowu
\def\result{4}
\fi
\def\signHcoloring{sign:H-coloring}
\ifx \questionnode \signHcoloring
\def\result{5}
\fi
\def\signnoxyedge{sign:no-xy-edge}
\ifx \questionnode \signnoxyedge
\def\result{6}
\fi
\def\signcolorsspecialgraph{sign:colors_special_graph}
\ifx \questionnode \signcolorsspecialgraph
\def\result{8}
\fi
\def\signcolorsK{sign:colors_K}
\ifx \questionnode \signcolorsK
\def\result{9}
\fi
\def\signHcoloringsecond{sign:H-coloring-second}
\ifx \questionnode \signHcoloringsecond
\def\result{11}
\fi
\def\signcomparabilitystatusofzanda{sign:comparability-status}
\ifx \questionnode \signcomparabilitystatusofzanda
\def\result{12}
\fi
\def\signcoloringtwocycles{sign:coloring-2-cycles}
\ifx \questionnode \signcoloringtwocycles
\def\result{13}
\fi
\def\signcolorsKsecond{sign:colors_K_second}
\ifx \questionnode \signcolorsKsecond
\def\result{14}
\fi
\alpha_{\result}
}

\DeclareMathOperator\Inc{Inc}
\DeclareMathOperator\cover{cover}
\DeclareMathOperator\tw{tw}
\DeclareMathOperator\p{p}
\DeclareMathOperator\mm{m}
\DeclareMathOperator\nn{n}
\DeclareMathOperator\inc{\|}

\newcommand{\side}[1]{\mm(#1)}
\newcommand{\up}[1]{\nn(#1)}
\newcommand{\down}[1]{\p(#1)}

\newcommand{\yes}{\textrm{yes}}
\newcommand{\no}{\textrm{no}}

\let\le\leqslant

\let\leq\leqslant
\let\geq\geqslant
\let\nleq\nleqslant
\let\ngeq\ngeqslant
\let\subset\subseteq
\let\subsetneq\varsubsetneq
\let\epsilon\varepsilon

\renewenvironment{enumerate}{\begin{enumorig}[label=\textup{(\roman*)}, noitemsep, topsep=2pt plus 2pt, labelindent=.2em, leftmargin=*, widest=iii]}{\end{enumorig}}

\renewenvironment{itemize}{\begin{itemorig}[label=\textbullet, noitemsep, topsep=2pt plus 2pt, labelindent=.5em, labelsep=.5em, leftmargin=*]}{\end{itemorig}}

\begin{document}
\title{On the dimension of posets with cover graphs of treewidth $2$}

\author[G.~Joret]{Gwena\"{e}l Joret}

\address[G.~Joret]{Computer Science Department \\
  Universit\'e Libre de Bruxelles\\
  Brussels\\
  Belgium}

\email{gjoret@ulb.ac.be}

\author[P.~Micek]{Piotr Micek}

\address[P.~Micek]{Theoretical Computer Science Department\\
  Faculty of Mathematics and Computer Science\\
  Jagiellonian University\\
  Krak\'ow\\
  Poland}

\email{micek@tcs.uj.edu.pl}

\author[W.~T.~Trotter]{William T. Trotter}

\author[R.~Wang]{Ruidong Wang}

\address[W.~T.~Trotter, R.~Wang]{School of Mathematics\\
  Georgia Institute of Technology\\
  Atlanta, Georgia 30332\\
  U.S.A.}

\email{trotter@math.gatech.edu, rwang49@math.gatech.edu}

\author[V.~Wiechert]{Veit Wiechert}

\address[V.~Wiechert]{Institut f\"ur Mathematik\\
  Technische Universit\"at Berlin\\
  Berlin \\
  Germany}

\email{wiechert@math.tu-berlin.de}

\thanks{G.\ Joret was supported by a
DECRA Fellowship from the Australian Research Council.}

\thanks{P.\ Micek is supported by the Mobility Plus program from The Polish Ministry of Science and higher Education.}

\thanks{V.\ Wiechert is supported by the Deutsche Forschungsgemeinschaft within the research training group `Methods for Discrete Structures' (GRK 1408).}

\date{\today}

\subjclass[2010]{06A07, 05C35}

\keywords{Poset, dimension, treewidth}

\begin{abstract}
In 1977, Trotter and Moore proved that a poset has dimension at most $3$ whenever its cover graph is a forest, or equivalently, has treewidth at most $1$. On the other hand, a well-known construction of Kelly shows that there are posets of arbitrarily large dimension whose cover graphs have treewidth $3$. In this paper we focus on the boundary case of treewidth $2$. It was recently shown that the dimension is bounded if the cover graph is outerplanar (Felsner, Trotter, and Wiechert) or if it has pathwidth $2$ (Bir\'o, Keller, and Young). This can be interpreted as evidence that the dimension should be bounded more generally when the cover graph has treewidth $2$. We show that it is indeed the case: Every such poset has dimension at most $1276$.
\end{abstract}

\maketitle

\section{Introduction}

The purpose of this paper is to show the following:

\begin{theorem}\label{thm:main}
Every poset whose cover graph has treewidth at most $2$ has dimension at most $1276$.
\end{theorem}

Let us provide some context for our theorem.
Already in 1977, Trotter and Moore~\cite{TM77} showed that if the cover graph of a poset $P$ is a forest
then $\dim(P)\leq 3$ and this is best possible, where  $\dim(P)$ denotes the dimension of $P$.
Recalling that forests are exactly the graphs of treewidth at most $1$, it is natural to
ask how big can the dimension be for larger treewidths.
Motivated by this question, we proceed with a brief survey of relevant results
about the dimension of posets and properties of their cover graphs.

One such result, due to  Felsner, Trotter and Wiechert~\cite{FTW13},
states that if the cover graph of a poset $P$ is outerplanar then $\dim(P)\leq 4$. Again, the bound is best possible.
Note that outerplanar graphs have treewidth at most $2$.
Note also that one cannot hope for a similar bound on the dimension of posets with a planar cover graph.
Indeed, already in 1981 Kelly~\cite{Kel81} presented a family of posets $\set{Q_n}_{n\geq 2}$ with planar cover graphs and $\dim(Q_n)=n$ (see Figure~\ref{fig:Kelly}).
One interesting feature of Kelly's construction for our purposes
is that the cover graphs also have treewidth at most $3$ (with equality for $n \geq 5$),
as is easily verified.
In fact, they even have pathwidth at most $3$ (with equality for $n \geq 4$).

\begin{figure}
\begin{center}
\includegraphics[scale=.6]{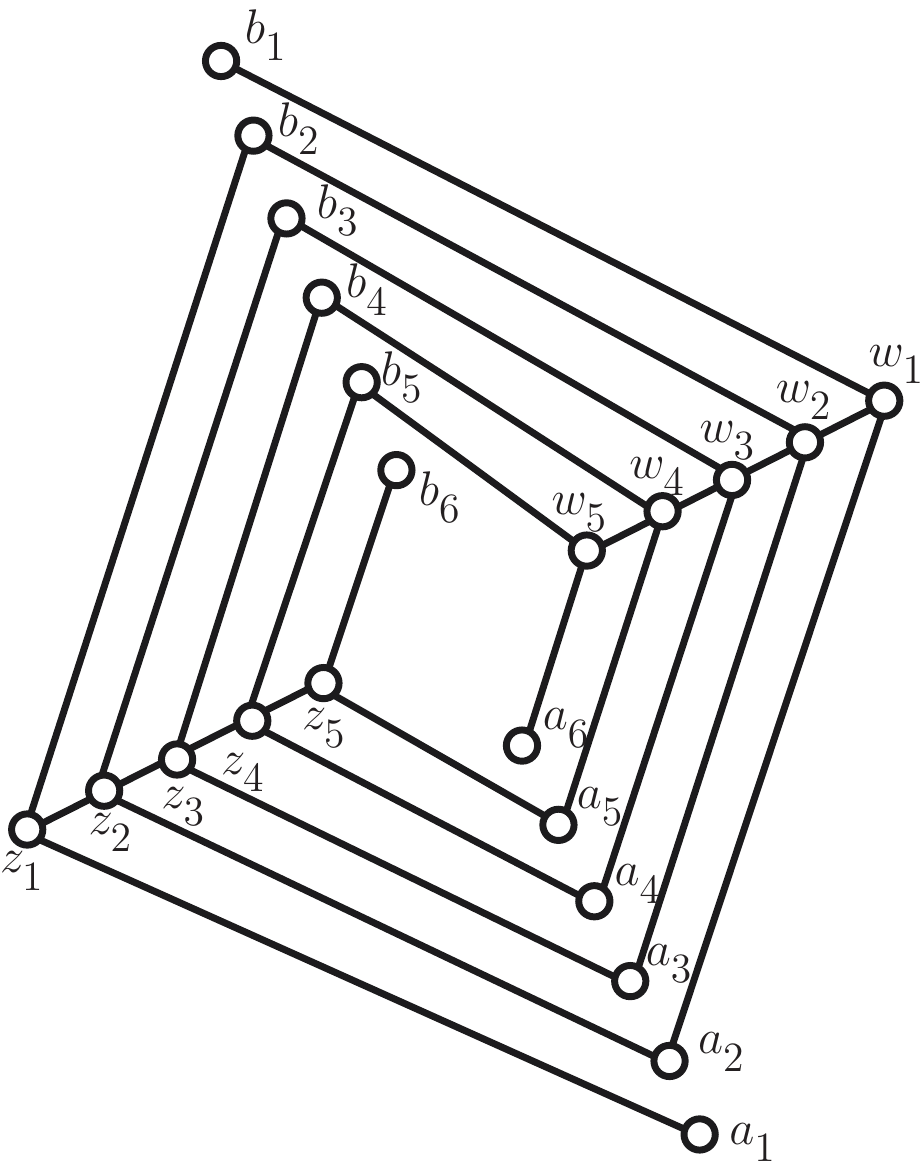}
\end{center}
\caption{\label{fig:Kelly}Kelly's construction of a poset $Q_n$ with a planar cover graph containing the standard example $S_n$
as a subposet, for $n=6$.
(Let us recall that the {\em standard example $S_n$} is the poset on $2n$ elements consisting
of $n$ minimal elements $a_1, \dots, a_n$ and $n$ maximal elements $b_1, \dots, b_n$ which is such that $a_i < b_j$ in $S_n$ if and only if
$i \neq j$.)
The subposet induced by the $a_i$'s and $b_i$'s form $S_{6}$, which has dimension $6$.
The general definition of $Q_n$ for any $n \geq 2$ is easily inferred from the figure.
Since the standard example $S_n$ has dimension $n$, this shows that posets with planar cover graphs have unbounded dimension.
}
\end{figure}

Very recently, Bir\'o, Keller and Young~\cite{Biro} showed that if the cover graph of a poset $P$ has pathwidth at most $2$,
then its dimension is bounded: it is at most $17$.
Furthermore, they proved that
the treewidth of the cover graph of any poset containing the standard example $S_n$ with $n \geq 5$ is at least $3$,
thus showing in particular that Kelly's construction cannot be modified to have treewidth $2$.

To summarize, while the dimension of posets with cover graphs of treewidth $3$ is unbounded, no such
property is known to hold for the case of treewidth $2$, and we cannot hope to obtain it by constructing
posets containing large standard examples.
Moreover, as mentioned above, the dimension is bounded
for two important classes of graphs of treewidth at most $2$, outerplanar graphs and graphs of pathwidth at most $2$.
All this can be interpreted as strong evidence that the dimension should be bounded more generally when the cover graph
has treewidth at most $2$, which is exactly what we prove in this paper.

We note that the bound on the dimension we obtain is large ($1276$), and is most likely far from the truth. Furthermore,
while we strove to make our arguments as simple as possible---and as a result did not try to optimize the bound---the proofs are lengthy and technical.
We believe that there is still room for improvements, and it could very well be that
a different approach would give a better bound and/or more insight into these problems.

We conclude this introduction by briefly mentioning a related line of research.
Recently, new bounds for the dimension were found for certain posets \emph{of bounded height}.
Streib and Trotter~\cite{ST14} proved that for every positive integer $h$, there is a constant $c$ such that
if a poset $P$ has height at most $h$ and its cover graph is planar, then $\dim(P)\leq c$.
Joret, Micek, Milans, Trotter, Walczak, and Wang~\cite{JMMTWW} showed that for every positive integers $h$ and $t$, there is a constant $c$ so that if $P$ has height at most $h$ and the treewidth of its cover graph is at most $t$, then $\dim(P)\leq c$.

These two results are closely related. In particular, one can deduce the result for planar cover graphs from the result for bounded treewidth cover graphs using a `trick'
introduced in~\cite{ST14} that reduces the problem to the special case where there is a special minimal element $a_0$ in the poset
that is smaller than all the maximal elements.
This implies that the diameter of the cover graph is
bounded from above by a  function of the height of the poset, and it is well-known
that planar graphs with bounded diameter have bounded treewidth (see for instance~\cite{Epp00}).
This trick of having a  special minimal element $a_0$ below all maximal elements turned out to be very useful in the context of this paper as well (though for different reasons),
see Observation~\ref{obs:a0} in Section~\ref{sec:preliminaries}.

Finally, we mention that several new results on bounding the dimension of certain posets in terms of their height have recently been obtained~\cite{minors, MW, sparsity},
the interested reader is referred to~\cite{sparsity} for a detailed overview of that area.

The paper is organized as follows.
In Section~\ref{sec:preliminaries} we give the necessary definitions and present a number of reductions,
culminating in a more technical version of our theorem, Theorem~\ref{thm:technical}.
Then, in Section~\ref{sec:proof}, we prove the result.

\section{Definitions and Preliminaries}\label{sec:preliminaries}

Let $P=(X, \leq)$ be a finite poset.
The \emph{cover graph} of $P$, denoted $\cover(P)$, is the graph on the elements of $P$ where two distinct elements $x$, $y$ are adjacent if and only if they are in a cover relation in $P$; that is, either $x < y$ or $x >y$ in $P$, and this relation cannot be deduced from transitivity.
Informally, the cover graph of $P$ can be thought of as its order diagram seen as an undirected graph.
The \emph{dimension} of $P$, denoted $\dim(P)$, is the least positive integer $d$ for which there are $d$ linear extensions $L_1,\ldots,L_d$ of $P$ so that $x\leq y$ in $P$ if and only if $x\leq y$ in $L_i$ for each $i \in \{1,\ldots,d\}$.
We mention that an introduction to the theory of posets and their dimension
can be found in the monograph~\cite{Tro-book} and in the survey article~\cite{Tro-handbook}.

When $x$ and $y$ are distinct elements in $P$, we write $x\inc y$ to denote that $x$
and $y$ are incomparable.
Also, we let $\Inc(P)=\{(x,y) \mid x, y\in X \textrm{ and } x\inc y \textrm{ in } P\}$
denote the set of ordered pairs of incomparable elements in $P$.
We denote by $\min(P)$ the set of minimal elements in $P$ and by $\max(P)$ the set of maximal elements in $P$.
The \emph{downset} of a set $S \subseteq X$ of elements is defined as $D(S)=\set{x\in X\mid \exists s\in S \text{ such that } x\leq s \text{ in }P}$, and similarly we define the \emph{upset} of $S$ to be $U(S)=\set{x\in X\mid \exists s\in S\text{ such that }s\leq x\text{ in }P}$.

A set $I \subseteq \Inc(P)$ of incomparable pairs is \emph{reversible} if there is a linear extension
$L$ of $P$ with $x>y$ in $L$ for every $(x,y)\in I$. It is easily seen that if $P$ is not a chain, then
$\dim(P)$ is the least positive
integer $d$ for which there exists a partition of $\Inc(P)$ into $d$ reversible sets.

A subset $\{(x_i,y_i)\}_{i=1}^k$ of  $\Inc(P)$ with $k \geq 2$ is said to be
an \emph{alternating cycle} if $x_i\le y_{i+1}$ in $P$ for each $i\in \{1, 2, \dots, k\}$, where indices are taken cyclically
(thus $x_k\le y_1$ in $P$ is required).
For example, in the poset $Q_6$ of Figure~\ref{fig:Kelly} the pairs $(a_i,b_i)$, $(a_j,b_j)$ form an alternating cycle of length $2$ for all $i,j\in\{1,\ldots,6\}$ such that $i\neq j$.
An alternating cycle  $\{(x_i,y_i)\}_{i=1}^k$ is
\emph{strict} if, for each $i, j \in \{1,2, \dots, k\}$, we have $x_i\le y_{j}$ in $P$ if and only if $j=i+1$ (cyclically).
Note that in that case
$x_1, x_2, \dots, x_k$ are all distinct, and $y_1, y_2, \dots, y_k$ are all distinct.
Notice also that every non-strict alternating cycle can be made strict by discarding some of its incomparable pairs.

Observe that if $I=\{(x_i,y_i)\}_{i=1}^k$ is an alternating cycle in $\Inc(P)$ then
$I$  cannot be reversed by a linear extension $L$ of $P$. Indeed, otherwise we would
have $y_i < x_i \leq y_{i+1}$ in $L$ for each $i \in \{1,2, \dots, k\}$, which cannot hold cyclically.
Hence, alternating cycles are not reversible. It is easily checked---and this was originally observed
by Trotter and Moore~\cite{TM77}---that every non-reversible subset $I \subseteq \Inc(P)$ contains
an alternating cycle, and thus a strict alternating cycle:
\begin{obs}\label{obs:reversible}
A set $I$ of incomparable pairs of a poset $P$ is reversible if and only if $I$ contains no strict alternating cycle.
\end{obs}

An incomparable pair $(x,y)$ of a poset $P$
is said to be a \emph{min-max pair} if $x$ is minimal in $P$ and $y$ is maximal in $P$.
The set of all min-max pairs in $P$ is denoted by $\MM(P)$.
Define $\dim^*(P)$ as the least positive integer $t$ such that $\MM(P)$ can be partitioned into $t$ reversible subsets
if $\MM(P)\neq \emptyset$, and as being equal to $1$ otherwise.
For our purposes, when bounding the dimension we will be able to focus on reversing only those
incomparable pairs that are min-max pairs.
This is the content of Observation~\ref{obs:min-max-pairs_general} below.
In order to state this observation formally we first need to recall some standard definitions from graph theory.

By `graph' we will always mean an undirected finite simple graph in this paper.
The \emph{treewidth} of a graph $G=(V,E)$ is the least positive integer $t$ such that there exist a tree $T$
and non-empty subtrees $T_x$ of $T$ for each $x\in V$ such that
\begin{enumerate}
\item $V(T_x)\cap V(T_y)\neq\emptyset$ for each edge $xy\in E$, and
\item $|\set{x\in V\mid u\in V(T_x)}|\leq t+1$ for each node $u$ of the tree $T$.
\end{enumerate}
The \emph{pathwidth} of $G$ is defined as treewidth, except that the tree $T$ is required to be a path.
A graph $H$ is a \emph{minor} of a graph $G$ if $H$ can be obtained from a subgraph of $G$ by contracting edges.
(We note that since we only consider simple graphs, loops and parallel edges resulting from edge contractions are deleted.)
Recall that the class of graphs of treewidth at most $k$ ($k \geq 0$) is closed under taking minors, thus
$\tw(H) \leq \tw(G)$ for every graph $G$ and minor $H$ of $G$.

Given a class $\mathcal{F}$ of graphs, we let $\widehat{\mathcal{F}}$ denote the class of graphs that can be obtained
from a graph $G \in \mathcal{F}$ by adding independently for each vertex $v$ of $G$ zero,
one, or two new pendant vertices adjacent to $v$.
We will use the easy observation that $\widehat{\mathcal{F}} = \mathcal{F}$
when $\mathcal{F}$ is the class of graphs of treewidth at most $k$ (provided $k \geq 1$).
We note that $\widehat{\mathcal{F}} = \mathcal{F}$ holds for other classes $\mathcal{F}$ of interest, such as planar graphs.

The next elementary observation is due to Streib and Trotter \cite{ST14}, who were interested in the case of planar cover graphs. We provide a proof for the sake of completeness.

\begin{obs}\label{obs:min-max-pairs_general}
Let $\mathcal{F}$ be a class of graphs.
If $P$ is a poset with $\cover(P) \in \mathcal{F}$
then there exists a poset $Q$ such that
\begin{enumerate}
\item $\cover(Q) \in \widehat{\mathcal{F}}$, and
\item $\dim(P)\leq\dim^*(Q)$.
\end{enumerate}
\end{obs}

\begin{proof}

If $P$ is a chain then we set $Q=P$ and the statement can be easily verified.

Otherwise, let $Q$ be the poset constructed from $P$ as follows:
For each non-minimal element $x$ of $P$, add a new element $x'$ below $x$ (and its upset) such that $x'<x$ is the only cover relation involving $x'$ in $Q$.
Also, for each non-maximal element $y$ of $P$, add a new element $y''$ above $y$ (and its downset) such that $y<y''$ is the only cover relation involving $y''$ in $Q$.
Now, the cover graph of $Q$ is the same as the cover graph of $P$
except that we attached up to two new pendant vertices to each vertex.

For convenience, we also define an element $x'$ for each minimal element $x$ of $P$, simply
by setting $x'=x$. Similarly, we let $y''=y$, for each maximal element $y$ of $P$.

Observe that if a set $\LL$ of linear extensions of $Q$ reverses
all min-max pairs of $Q$ then it must reverse all incomparable pairs
of $P$.
Indeed, for each pair $(x,y)\in \Inc(P)$ consider the min-max pair
$(x',y'')$ in $Q$.
 There is some linear extension $L\in \LL$ reversing $(x',y'')$.
 Given that $x' \leq x$ and $y \leq y''$ in $Q$, it follows that $y
\leq y''<x' \leq x$ in $L$.
Hence, restricting the linear orders in $\LL$ to the elements of $P$
we deduce that $L$ reverses all pairs in $\Inc(P)$ so $\dim(P) \leq
|\LL|$ (as $P$ is not a chain).
Therefore, $\dim(P)\leq\dim^*(Q)$.
\end{proof}

As a corollary, for treewidth we obtain:
\begin{obs}\label{obs:min-max-pairs}
For every poset $P$ there exists a poset $Q$ such that
\begin{enumerate}
\item $\tw(\cover(P))=\tw(\cover(Q))$, and
\item\label{item:dimP=dim*Q} $\dim(P)\leq\dim^*(Q)$.
\end{enumerate}
\end{obs}
\begin{proof}
This follows from Observation~\ref{obs:min-max-pairs_general} if $\tw(\cover(P)) \geq 1$.
If, on the other hand, $\tw(\cover(P)) =0$, then $P$ is an antichain and we can simply take $Q=P$.
\end{proof}

In the next observation we consider posets with disconnected cover graphs.
As expected, we define the \emph{components} of a poset $P$ as the subposets of $P$ induced by the
components of its cover graph.

\begin{obs}\label{obs:disconnected-cover-graph}
If $P$ is a poset with $k \geq 2$ components $C_1,\ldots,C_k$ then either
\begin{enumerate}
 \item $P$ is a disjoint union of chains and we have $\dim(P)=\dim^*(P)=2$, or
 \item $\dim(P)=\max\set{\dim(C_i)\mid i=1,\ldots,k}$ and\\ $\dim^*(P)=\max\set{\dim^*(C_i)\mid i=1,\ldots,k}$.
\end{enumerate}
\end{obs}
\begin{proof}
If for each $i\in \{1,\ldots,k\}$ the subposet $C_i$ of $P$ is a chain then it is easy to see that $\dim(P)=\dim^*(P)=2$.
Thus we may assume that this is not the case, that is, $\dim(C_i)\geq2$
for some  $i\in\set{1,\ldots,k}$.

For each $i\in\set{1,\ldots,k}$
let $\calR_i$ be a family of $\dim(C_i)$ linear extensions of $C_i$ witnessing the dimension of $C_i$.
We construct a family $\calR$ of linear extensions of $P$ in the following way.
First, let $\calR := \emptyset$. Then,
as long as there is a set $\calR_i$ which is not empty,
\begin{enumerate}
 \item choose a linear extension $L_i\in \calR_i$ for each $i\in \{1,\ldots,k\}$ such that $\calR_i$ is not empty;
 \item choose any linear extension $L_i$ of $C_i$ for each $i\in \{1,\ldots,k\}$ such that $\calR_i$ is empty;
 \item add to $\calR$ the linear extension $L$ of $P$ defined by $L:=L_1 < L_2 < \ldots < L_k$, and
 \item remove $L_i$ from $\calR_i$ for each $i\in \{1,\ldots,k\}$.
\end{enumerate}
Clearly, $|\calR|=\max\set{\dim(C_i)\mid i\in \{1,\ldots,k\}}$.
Now consider one arbitrarily chosen linear extension $L\in\calR$; say we had
$L=L_1<\ldots<L_k$ when it was defined above,
and replace $L$ by $L':=L_k<\ldots<L_1$ in $\calR$.
It is easy to verify that the resulting family $\calR$ reverses all incomparable pairs in $P$.
In particular, all incomparable pairs of $P$ with elements from distinct components are reversed by $L'$ and any other linear extension in $\calR$ (note there is at least one more as $\dim(C_i)\geq2$ for some $i$). This shows that $\dim(P)=\max\set{\dim(C_i)\mid i\in \{1,\ldots,k\}}$.

The proof for $\dim^*(P)$ goes along the same lines and is thus omitted.
\end{proof}

To prove the next observation we partition the minimal and maximal elements of a poset by `unfolding' the poset from an arbitrary minimal element, and contract some part of the poset into a single element.
This proof idea is due to Streib and Trotter~\cite{ST14}, and is very useful for our purposes.
In~\cite{ST14} it was used in the context of planar cover graphs but it works equally well
for any minor-closed class of graphs.

\begin{obs}\label{obs:a0}
For every poset $P$ there exists a poset $Q$ such that
\begin{enumerate}
\item $\cover(Q)$ is a minor of $\cover(P)$ (and thus in particular $\tw(\cover(Q))\leq\tw(\cover(P))$);
\item\label{item:a0} there is an element $q_0\in\min(Q)$ with $q_0<q$ in $Q$ for all $q\in\max(Q)$, and
\item $\dim^*(P)\leq 2\dim^*(Q)$.
\end{enumerate}
\end{obs}

\begin{proof}
First of all, we note that it is enough to prove the statement in the case where
$\cover(P)$ is connected. Indeed, if $\cover(P)$ is disconnected then by Observation \ref{obs:disconnected-cover-graph} either $P$ is a disjoint union of chains and $\dim(P)=\dim^*(P)=2$, in which case the observation is trivial, or $\dim^*(P)=\max\set{\dim^*(C)\mid C \text{  component of $P$}}$ and we can simply consider a component $C$ of $P$ with $\dim^*(P)=\dim^*(C)$.

From now on we suppose that $\cover(P)$ is connected.
We are going to build a small set of linear extensions of $P$ reversing all min-max pairs of $P$.
Partition the minimal and maximal elements of $P$ as follows.
Choose an arbitrary element $a_0\in \min(P)$, let $A_0=\set{a_0}$, and
for $i = 1,2, 3, \dots$ let
\begin{align*}
 B_i&=\set{b\in\max(P)-\bigcup_{1 \leq j<i} B_j \mid \text{there exists } a\in A_{i-1} \text{ with } a<b \text{ in } P},\\
 A_i&=\set{a\in\min(P)-\bigcup_{0 \leq j<i} A_j \mid \text{there exists } b\in B_{i} \text{ with } a<b \text{ in } P}.
\end{align*}
Let $k$ be the least index such that $A_k$ is empty.
See Figure~\ref{fig:contraction1} for an illustration.
The fact that each minimal and maximal element of $P$ is included in one of the sets defined above follows from the connectivity of $\cover(P)$.
If $k=1$ then $a_0$ is below all maximal elements of $P$ and hence $P$ itself satisfies conditions (i)-(iii).
So we may assume $k\geq 2$ from now on.

\begin{figure}[t]
 \centering
 \includegraphics[scale=0.5]{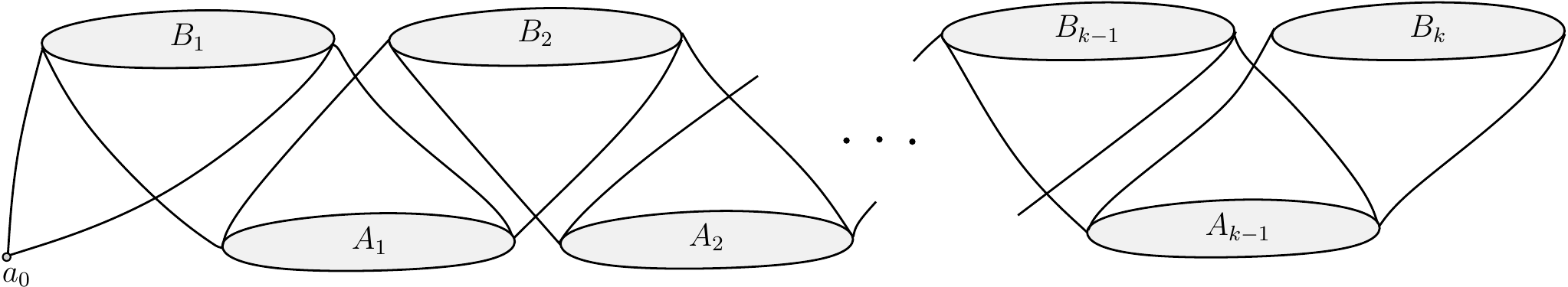}
 \caption{Schematic drawing of $P$ and the sets $A_0,A_1,\ldots,A_{k-1}$ and $B_1,\ldots,B_k$.}
 \label{fig:contraction1}
\end{figure}

Let $Q_i^{i+1}$ be the poset on the set of elements $X_i^{i+1}=A_i\cup B_{i+1}\cup \big(U(A_i)\cap D(B_{i+1})\big)$ with order relation inherited from $P$.
Figure~\ref{fig:contraction2} illustrates this definition.
Let
\[
 t=\max\set{\dim^*(Q_i^{i+1})\mid i=0,\ldots,k-1}.
\]
For each $i\in\set{0,\ldots,k-1}$ consider $t$ linear extensions $L_1^i,\ldots,L_t^i$ of $Q_i^{i+1}$ that reverse all pairs from the set $\MM(Q_i^{i+1})$.
Combining these we define $t$ linear extensions of $P$.
For $j\in\set{1,\ldots,t}$ let $L_j$ be a linear extension of $P$ that contains the linear order
\[
 L_j^{k-1}<\cdots<L_j^1<L_j^0.
\]
Then, $L_1,\ldots,L_t$ reverse all pairs $(a,b)\in\MM(P)$ with $a\in A_i$ and $b\in B_j$ where $j\geq i+1$.
In a similar way we are able to reverse the pairs where $j\leq i$.

Let $Q_i^i$ be the poset on the set of elements $X_i^i=A_i\cup B_i\cup \big(U(A_i)\cap D(B_i) \big)$ being ordered as in $P$.
We set $t'=\max\set{\dim^*(Q_i^i)\mid i\in\set{1,\ldots,k-1}}$ and for each $i\in\set{1,\ldots,k-1}$ we fix $t'$ linear extensions $L_1^i,\ldots,L_{t'}^i$ of $Q_i^i$ reversing all pairs from $\MM(Q_i^i)$.
Again, we combine these to obtain linear extensions of $P$.
For $j\in\set{1,\ldots,t'}$ let $L_j'$ be a linear extension of $P$ that contains the linear order
\[
 L_j^1<L_j^2<\cdots<L_j^{k-1}.
\]
Clearly, $L_1'\ldots,L_{t'}'$ reverse all pairs $(a,b)\in\MM(P)$ with $a\in A_i$ and $b\in B_j$ where $j\leq i$.
It follows that $L_1,\ldots,L_t,L_1'\ldots,L_{t'}'$ reverse the set $\MM(P)$ and hence $\dim^*(P)\leq t+t'$.

Now suppose first $t> t'$, so in particular $t>1$.
Then let $\ell\in\set{0,\ldots,k-1}$ such that $t=\dim^*(Q_\ell^{\ell+1})$.
Note that we must have $\ell\geq 1$ since $\dim^*(Q_0^1)=1<t$.
We define $Q$ to be the poset that is obtained from $Q_\ell^{\ell+1}$ by adding an extra element $q$ which is such that $q>x$ for all $x\in X_\ell^{\ell+1}\cap D(B_\ell)$, and incomparable to all other elements of $Q_\ell^{\ell+1}$
(here we need $\ell\geq 1$ so that $B_\ell$ exists).
In particular, $q>a$ for all $a\in A_\ell$.
Observe that the cover graph of $Q$ is an induced subgraph of $\cover(P)$ with an extra vertex $q$ linked to some of the other vertices.
Here, $q$ can be seen as the result of the contraction of the connected set $\bigcup_{1\leq j\leq \ell} D(B_j) - X_\ell^{\ell+1}$ plus the deletion of some
of the edges incident to the contracted vertex (see Figure~\ref{fig:contraction2} with $i=\ell$, dashed edges indicate deletions).
The deletion step is necessary, as after the contraction it might be that some edges incident to $q$ do not correspond to cover relations anymore.
It follows that $\cover(Q)$ is a minor of $\cover(P)$.
Furthermore, it holds that
\[
\dim^*(P)\leq 2t= 2\dim^*(Q_\ell^{\ell+1})\leq 2\dim^*(Q).
\]
Therefore, the dual of $Q$ satisfies conditions (i)-(iii).

The case $t'\geq t$ goes along similar lines as in the first case (with the slight difference that we do not need to exclude the subcase $t'=1$).
We leave the details to the reader.
\end{proof}

\begin{figure}[t]
 \centering
 \includegraphics[scale=0.75]{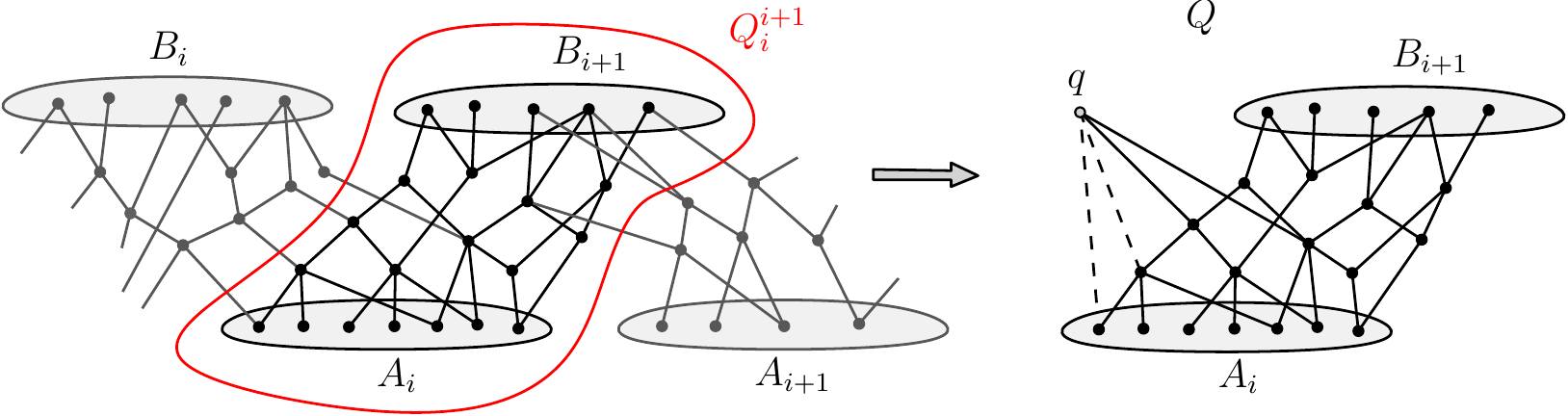}
 \caption{Definition of $Q_i^{i+1}$ and construction of $Q$ with its cover graph.}
 \label{fig:contraction2}
\end{figure}

Applying these observations we move from Theorem~\ref{thm:main} to a more technical statement.

\begin{theorem}\label{thm:technical}
Let $P$ be a poset with
\begin{enumerate}
\item a cover graph of treewidth at most $2$, and
\item a minimal element $a_0\in\min(P)$ such that $a_0 < b$ for all $b\in\max(P)$.
\end{enumerate}
Then the set $\MM(P)$ can be partitioned into $638$ reversible sets.
\end{theorem}

In order to deduce Theorem~\ref{thm:main} from Theorem~\ref{thm:technical} consider any poset $P$ with cover graph of treewidth at most $2$.
By Observation~\ref{obs:min-max-pairs} there is a poset $Q$ with $\tw(\cover(Q))\leq2$ and $\dim(P)\leq\dim^*(Q)$.
Now by Observation~\ref{obs:a0} and applying Theorem \ref{thm:technical} there is a poset $R$ with $\tw(\cover(R))\leq2$, a minimal element $a_0\in\min(R)$ such that $a_0<b$ for all $b\in\max(R)$, and
\[
\dim(P)\leq \dim^*(Q)\leq 2\dim^*(R)\leq 2\cdot638 = 1276,
\]
as desired.

From now on we focus on the proof of Theorem~\ref{thm:technical}.
Let $P=(X,\leq)$ be a poset fulfilling the conditions of Theorem~\ref{thm:technical}.
Consider a tree decomposition of width at most $2$ of $\cover(P)$, consisting of a tree $T$
and subtrees $T_x$ for each $x\in X$.
We may assume that the width of the decomposition is exactly $2$, since otherwise
$\dim(P) \leq 3$ by the result of Trotter and Moore~\cite{TM77}, and the theorem follows trivially.

For each node $u$ of $T$ let $B(u)$ denote its \emph{bag}, namely, the set $\set{x\in X\mid u\in V(T_x)}$.
Since the tree decomposition has width $2$, every bag has size at most $3$, and at least one bag has size exactly $3$.
Modifying the tree decomposition if necessary, we may suppose that every bag has size $3$.
Indeed, say $uv$ is an edge of $T$ with $|B(u)|=3$ and $|B(v)|\leq 2$. Then choose
arbitrarily $3 - |B(v)|$ elements from $B(u) \setminus B(v)$ and add them to $B(v)$.
Repeating this process as many times as necessary, we eventually  ensure that every bag has size $3$.
Note that the subtrees $T_x$ ($x\in X$) of the tree decomposition are uniquely determined by
the bags, and vice versa; thus, it is enough to specify how $T$ and the bags are modified.
The above modification repeatedly adds leaves to some of the subtrees $T_x$ ($x\in X$), which clearly
keeps the fact that $T$ and the subtrees $T_x$ ($x\in X$) form a tree decomposition of $\cover(P)$.

Recall that, by the assumptions of Theorem \ref{thm:technical}, the poset $P$ has a minimal element $a_0$ with $a_0<b$ for all $b\in\max(P)$. This implies that the cover graph of $P$ is connected.
Using this, we may suppose without loss of generality that $|B(u) \cap B(v)| \geq 1$
for each edge $uv$ of $T$. For if this does not hold, then the bags of one of the two components of $T -uv$
are all empty (as is easily checked),
and thus the nodes of that component can be removed from $T$ without affecting the tree decomposition.

In fact, we may even assume that $|B(u) \cap B(v)|=2$  holds for every edge $uv$ of $T$.
To see this, consider the following iterative modification of the tree decomposition:
Suppose that  $uv$ is an edge of $T$ such that $t:= |B(u) \cap B(v)| \neq 2$.
If $t=3$ then simply identify $u$ and $v$, and contract the edge $uv$ in $T$.
If $t =1$ then subdivide the edge $uv$ in $T$ with a new node $w$, and
let the bag $B(w)$ of $w$ be the set $(B(u) \cap B(v)) \cup \{x,y\}$, where $x$ and $y$ are
arbitrarily chosen elements in $B(u) \setminus B(v)$ and $B(v) \setminus B(u)$, respectively.
These modifications are valid,
in the sense that the bags still define a tree decomposition of $\cover(P)$ of width $2$, and in order to ensure the desired property
it suffices to apply them iteratively until there is no problematic edge left.

To summarize, in the tree decomposition we have $|B(u)| = 3$ for every node $u$ of $T$,
and $|B(u) \cap B(v)|=2$  for every edge $uv$ of $T$. We will need to further refine our tree decomposition so as
to ensure a few extra properties. These changes will be explained one by one below.
Let us mention that we will keep the fact that $|B(u) \cap B(v)|=2$  for every edge $uv$ of $T$, and
that $|B(u)| = 3$ for every {\em internal} node $u$ of $T$. However, we will add
new leaves to $T$ having bags of size $2$ only.

Choose an arbitrary node $r' \in V(T)$ with $a_0\in B(r')$. Add a new node $r$ to $T$ and make it adjacent to $r'$.
The bag $B(r)$ of $r$ is defined as the union of $a_0$ and one arbitrarily chosen element from $B(r') - \{a_0\}$.
(Observe that the size of $B(r)$ is only $2$; on the other hand, we do have $|B(r) \cap B(r')|=2$.)
We call $r$ the \emph{root} of $T$, and thus see $T$ as being rooted at $r$.
(For a technical reason we need the root to be a leaf of $T$, which explains why we set it up this way.)
Every non-root node $u$ in $T$ has a \emph{parent} $\p(u)$ in $T$, namely, the neighbor
of $u$ on the path from $u$ to $r$ in $T$.
Now we have an order relation on the nodes of $T$, namely $u\leq v$ in $T$ if $u$ is on the path from $r$ to $v$ in $T$.
The following observation will be useful later.

\begin{obs}\label{obs:comp-nodes}
If $v_1,\ldots,v_n$ is a sequence of nodes of $T$ such that consecutive nodes are comparable in $T$
(that is $v_i\leq v_{i+1}$ or $v_{i+1}\leq v_i$ in $T$ for each $i \in \set{1,\ldots,n-1}$), then there is an index $j\in\set{1,\ldots,n}$ such that $v_j\leq v_i$ in $T$ for each $i \in \set{1,\ldots,n}$.
\end{obs}
\begin{proof}
 We prove this by induction on $n$.
For $n=1$ it is immediate.
So suppose that $n>1$.
Then we can apply the induction hypothesis on the sequence $v_1,\ldots,v_{n-1}$ and get $j\in\set{1,\ldots,n-1}$ such that $v_j\leq v_i$ for each $i\in \set{1,\ldots,n-1}$.
As $v_{n-1}$ and $v_n$ are comparable in $T$, we have $v_{n-1}\leq v_n$ or $v_n\leq v_{n-1}$ in $T$.
In the first case we conclude $v_j\leq v_{n-1}\leq v_n$ in $T$ and we are done.
In the second case we have $\set{v_j,v_n}\leq v_{n-1}$ in $T$, which makes $v_j$ and $v_n$ comparable in $T$.
But clearly, from this it follows that $v_j\leq v_i$ in $T$ for each $i\in \set{1,\ldots,n}$ or $v_n\leq v_i$ in $T$ for each $i\in \set{1,\ldots,n}$.
\end{proof}

Fix a planar drawing of the tree $T$ with the root $r$ at the bottom.
Suppose that $v$ and $v'$ are two nodes of $T$ that are incomparable in $T$.
Take the maximum node $u$ (with respect to the order in $T$) such that $u\leq v$ and $u\leq v'$ in $T$.
We denote this node by $v\land v'$.
Observe that $u$ has degree at least $2$ in $T$, and hence is distinct from the root $r$.
(Ensuring this is the reason why we made sure that the root $r$ is a leaf.)
Consider the edge $p$ from $u$ to $\p(u)$, the edge $e$ from $u$ towards $v$ and the edge $e'$ from $u$ towards $v'$.
All these edges are distinct. If the clockwise order around $u$ in the drawing is $p,e,e'$ for these three edges,
then we say that $v$ is \emph{to the left} of $v'$ in $T$, otherwise
the clockwise order around $u$ is $p,e',e$ and we say that $v$ is \emph{to the right} of $v'$ in $T$.
Observe that the relations  ``is left of in $T$'' and ``is right of in $T$'' both induce a linear order
on any set of nodes which are pairwise incomparable in $T$.

\begin{obs}\label{obs:left-right}
 Let $v$ and $v'$ be incomparable nodes in $T$ with $v$ left of $v'$ in $T$, and let $u:=v\land v'$.
If $w$ and $w'$ are the neighbors of $u$ on the paths towards $v$ and $v'$ in $T$, respectively, then for each node $c$ in $T$ we have that
\end{obs}
\begin{enumerate}
 \item  $v$ is left of $c$ in $T$ if $w'\leq c$ in $T$\label{item:a-is-left}, and
 \item $c$ is left of $v'$ in $T$ if $w\leq c$ in $T$\label{item:b-is-right}.
\end{enumerate}
\begin{proof}
 If $w'\leq c$ in $T$, then we also have $u =v \land c$, and the first edge on the path from $u$ to $c$ in $T$ is the same as that of the path from $u$ to $v'$ in $T$. Since $v$ is left of $v'$ in $T$, it follows that $v$ is left of $c$ as well.
The proof for the second item is analogous.
\end{proof}

Next we modify once more the tree decomposition.
For each element $a\in \min(X)$ such that $(a,b)\in\MM(P)$ for some $b\in \max(X)$, choose arbitrarily a node
$w_a$ of $T$ such that  $a\in B(w_a)$.
Similarly, for each element $b\in \max(X)$ such that $(a,b)\in\MM(P)$ for some $a\in \min(X)$,
choose arbitrarily a node $w_b$ of $T$ such that  $b\in B(w_b)$.
(Note that the same node of $T$ could possibly be chosen more than once.)
Now that all these choices are made,
for each minimal element $a$ of $P$ considered above,
add a new leaf $a^T$ to $T$ adjacent to $w_a$ with bag
$B(a^T):=\set{a,x}$, where $x$ is an arbitrarily chosen element from $B(w_a) \setminus \{a\}$.
Similarly, for each maximal element $b$ of $P$ considered above,
add a new leaf $b^T$ to $T$ adjacent to $w_b$ with bag
$B(b^T):=\set{b,x}$, where $x$ is an arbitrarily chosen element from $B(w_b) \setminus \{b\}$.

This concludes our modifications of the tree decomposition. Notice that we made sure that
$|B(u)|=3$ for every internal node $u$ of $T$, and that $|B(u) \cap B(v)|=2$  for every edge $uv$ of $T$.
Observe also that for every pair $(a, b) \in \MM(P)$, the two nodes $a^T$ and $b^T$ are incomparable in $T$,
and thus one is to the left of the other in $T$.
Figure~\ref{fig:tree-decomposition} provides an illustration.
(We also note that while the tree $T$ has been modified since stating Observations~\ref{obs:comp-nodes}
and~\ref{obs:left-right}, they obviously still apply to the new tree $T$.)

\begin{figure}[t]
 \centering
 \includegraphics[scale=1]{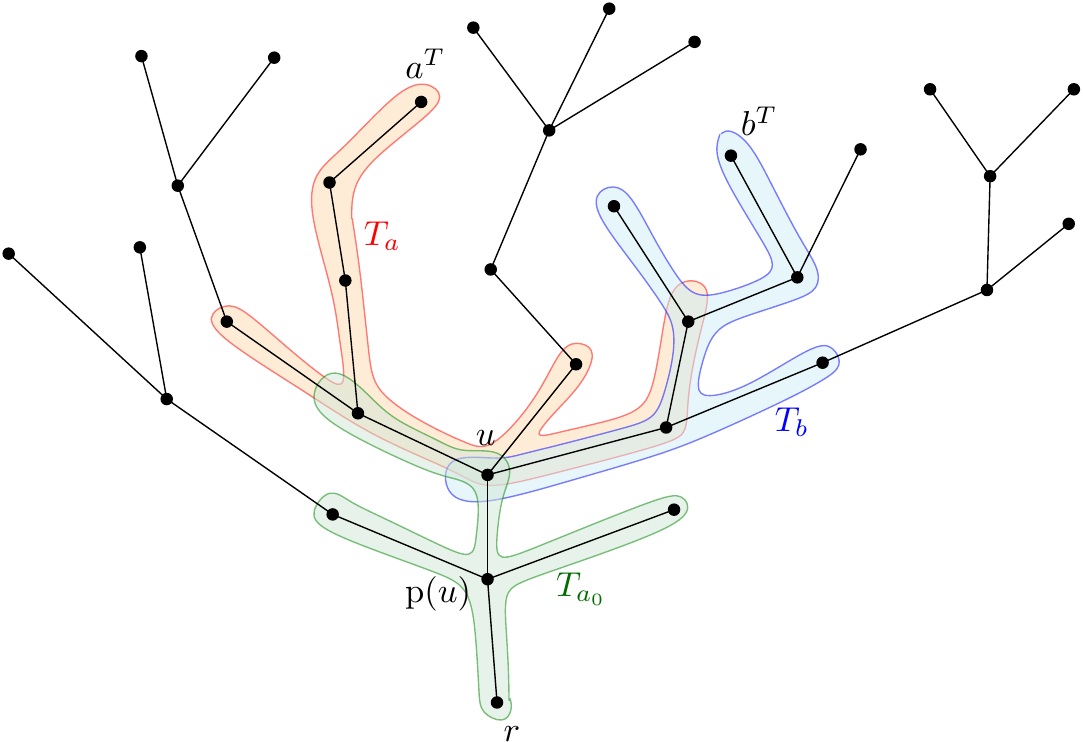}
 \caption{\label{fig:tree-decomposition}We have the following properties in this example:
 $(a, b) \in \MM(P)$ with
 $a^T$ left of $b^T$ in $T$, and $u = a^T \land b^T$ (and hence $u<a^T$ and $u<b^T$ in $T$). We also have $B(u)=\set{a_0,a,b}$ here.}
\end{figure}

Let $G$ be the intersection graph of the subtrees $T_x$ ($x\in X$) of $T$.
Thus two distinct elements $x, y \in X$ are adjacent in $G$ if and only if $V(T_x) \cap V(T_y) \neq \emptyset$.
The graph $G$ is chordal and the maximum clique size in $G$ is $3$.
Hence the vertices of $G$ can be (properly) colored with three colors.
We fix a $3$-coloring $\phi$ of $X$ which is such that $x,y\in X$ receive distinct colors whenever $V(T_x)\cap V(T_y)\neq\emptyset$.
In particular, if $x$ and $y$ are two distinct elements of $P$ such that
$x,y\in B(u)$ for some $u\in V(T)$ then $x$ and $y$ receive different colors.

We end this section with a fundamental observation which is going to be used repeatedly in a number of forthcoming arguments.
We say that a relation $x\leq y$ in $P$ \emph{hits} a set $Z\subset X$ if there exists $z\in Z$ with $x\leq z \leq y$ in $P$.

\begin{obs}\label{obs:hitting-vertex-or-edge}
Let $x\leq y$ in $P$ and let $u,v \in V(T)$ be such that $x\in B(u)$, $y\in B(v)$.
\begin{enumerate}
\item\label{item:vertex} If $w\in V(T)$ lies on the path from $u$ to $v$ in $T$ then $x\leq y$ hits $B(w)$.
\item\label{item:edge} If $e=w_1w_2 \in E(T)$ lies on the path from $u$ to $v$ in $T$ then $x\leq y$ hits $B(w_1)\cap B(w_2)$.
\item\label{item:2vertices}
If $w_{1}, \dots, w_{t}\in V(T)$ are $t$ nodes on the path from $u$ to $v$ in $T$ appearing in this order,
then there exist $z_{i} \in B(w_{i})$ for each $i\in \set{1, \dots, t}$
such that $x \leq z_{1} \leq \cdots \leq z_t \leq y$ in $P$.
\end{enumerate}
\end{obs}

\begin{proof}
Suppose that $w$ lies on a path from $u$ to $v$ in $T$.
Since $x\leq y$ in $P$ there is a path $x=z_0,z_1,\ldots,z_k=y$ in $G$ such that
$z_{i} < z_{i+1}$ is a cover relation in $P$ for each $i \in \{0, 1, \dots, k-1\}$.
This means that $\bigcup_{0 \leq i \leq k} T_{z_i}$ is a (connected) subtree of $T$ containing $u$ and $v$.
Thus, $\bigcup_{0 \leq i \leq k} T_{z_i}$ contains $w$ and therefore there exists $i$ with $z_i\in B(w)$.
The proof of~\ref{item:edge} is analogous.

We prove~\ref{item:2vertices} by induction on $t$. For $t=1$ this corresponds to~\ref{item:vertex}, so let us
assume $t > 1$ and consider the inductive case. By induction
there exist $z_{i} \in B(w_{i})$ for each $i\in \set{1, \dots, t-1}$
such that $x \leq z_{1} \leq \cdots \leq z_{t-1} \leq y$ in $P$.
Applying~\ref{item:vertex} with relation $z_{t-1} \leq y$ and the $w_{t-1}$--$v$ path,
we obtain that $z_{t-1} \leq z_t \leq y$ in $P$ for some $z_{t} \in B(w_{t})$.
Combining, we obtain $x \leq z_{1} \leq \cdots \leq z_t \leq y$ in $P$, as desired.
\end{proof}

\section{The Proof}\label{sec:proof}
\label{sec:partitioning}

We aim to partition $\MM(P)$ into a constant number of sets, each of which is reversible.
This will be realized with the help of a \emph{signature tree}, which is depicted on Figure~\ref{fig:signature_tree}.
This plane tree $\sigtree$, rooted at node $\nu_{1}$, assigns to each pair $(a,b)\in\MM(P)$ a corresponding leaf of $\sigtree$ according to properties of the pair $(a,b)$.

\begin{figure}
\centering
\includegraphics[width=0.4\textwidth]{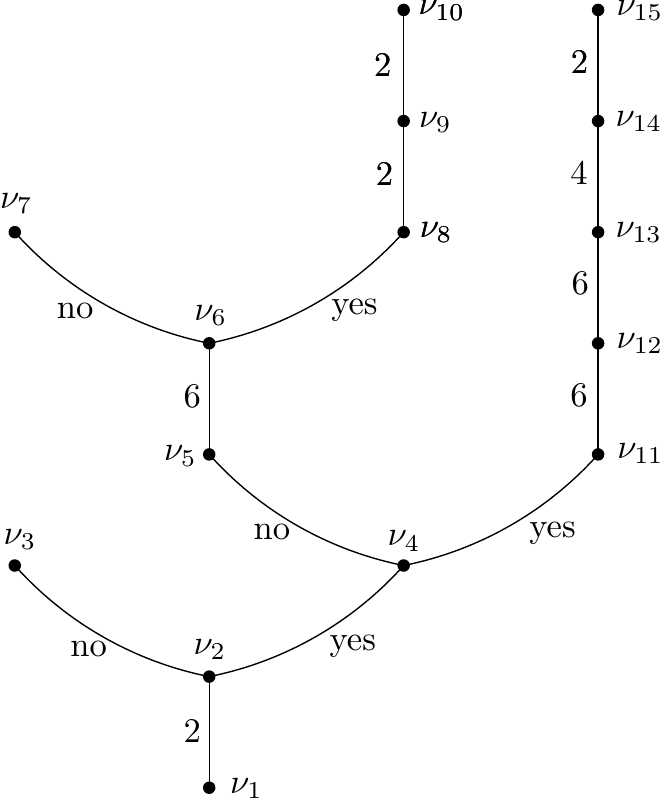}
\caption{\label{fig:signature_tree}The signature tree $\Psi$. Each of the three branching nodes $\nu_2$,
$\nu_4$, and $\nu_6$  corresponds to a yes/no question, and the edges towards their children are labeled
according to the possible answers. Edges starting from a non-branching node $\nu_i$ to a node $\nu_j$
are labeled with the size of $\mathbf{\Sigma}(\nu_i,\nu_j)$.}
\end{figure}

The nodes $\nu_1,\ldots,\nu_{15}$ of $\sigtree$ are enumerated by depth-first and left-to-right search.
Each node $\nu_i$ which is distinct from the root and not a leaf has a corresponding function of the form $\alpha_i:\MM(P,\nu_i)\to  \mathbf{\Sigma}_i$, where $\MM(P,\nu_i)\subset\MM(P)$ and $\mathbf{\Sigma}_i$ is a finite set, whose size does not depend on $P$.
We put $\MM(P,\nu_1)=\MM(P)$ and the other domains will be defined one by one in this section.
To give an example, let us look forward to upcoming subsections where we define $\alpha_1$ and $\alpha_2$ as follows.
\begin{itemize}
\item $\q{sign:left-or-right}(a,b) \in \mathbf{\Sigma}_1=\{\textrm{left}, \textrm{right}\}$ encodes whether $a^T$ is to the left or to the right of $b^T$ in $T$;
\item $\q{sign:a<x}(a,b) \in \mathbf{\Sigma}_2 =\{\textrm{yes}, \textrm{no}\}$ is the answer to the question ``Is there an element $q\in B(a^{T} \land b^{T})$ with $a\leq q$ in $P$?''.
\end{itemize}

Furthermore, for each internal node $\nu_i$ with children $\nu_{i_1},\ldots,\nu_{i_l}$ in $\sigtree$, the edges $\nu_i\nu_{i_1}, \dots,
\nu_i\nu_{i_l}$ of $\sigtree$ are respectively labeled by subsets $\mathbf{\Sigma}(\nu_{i},\nu_{i_1}),\ldots,\mathbf{\Sigma}(\nu_{i},\nu_{i_l})$ of $\mathbf{\Sigma}_i$ such that $\mathbf{\Sigma}_i=\mathbf{\Sigma}(\nu_{i},\nu_{i_1})\sqcup \cdots\sqcup \mathbf{\Sigma}(\nu_{i},\nu_{i_l})$, that is,
so that the sets $\mathbf{\Sigma}(\nu_{i},\nu_{i_j})$ form a partition of $\mathbf{\Sigma}_i$.
For example,
\begin{align*}
&\mathbf{\Sigma}(\nu_1,\nu_2)=\set{\textrm{left}, \textrm{right}}=\mathbf{\Sigma}_1; \\
&\mathbf{\Sigma}(\nu_2,\nu_3)=\set{\textrm{no}};  \\
&\mathbf{\Sigma}(\nu_2,\nu_4)=\set{\textrm{yes}}.
\end{align*}
Observe that each internal node $\nu_i$ of $\Psi$ has either one or two children; in particular, if $\nu_i$ has only one child then the corresponding edge is labeled
with the full set $\mathbf{\Sigma}_i$.

The reader may wonder why we do not refine the tree $\Psi$ and have an edge out of $\nu_i$ for every possible value in $\mathbf{\Sigma}_i$.
This is because sometimes several values in $\mathbf{\Sigma}_i$ will correspond to analogous cases in our proofs
which can be treated all at once.  To give a concrete example,
consider $\mathbf{\Sigma}(\nu_1,\nu_2)=\set{\textrm{left},\textrm{right}}$: When proving that a set $S$ of min-max pairs is reversible, the case that $a^T$ is left of $b^T$ for every $(a,b)\in S$ is analogous to the case that $a^T$ is right of $b^T$ for every $(a,b)\in S$, as one
is obtained from the other by exchanging the notion of left and right in $T$ (that is, by replacing the plane tree $T$ by its mirror image).
Hence it will be enough to only consider, say, the case where $a^T$ is to the left of $b^T$ for every $(a,b)\in S$.

Now for an internal node $\nu_i$ of $\sigtree$ distinct from the root ($i\neq1$), let $\nu_1=\nu_{i_1},\ldots,\nu_{i_l}=\nu_i$ be the path from the root $\nu_1$ to $\nu_i$ in $\sigtree$.
Define the \emph{signature} of $\nu_i$ as the set
\begin{equation}
\mathbf{\Sigma}(\nu_i)=\mathbf{\Sigma}(\nu_{i_1},\nu_{i_2})\times\ldots\times\mathbf{\Sigma}(\nu_{i_{l-1}},\nu_{i_l})\label{eq:signature-def}
\end{equation}
and let
\begin{align*}
\MM(P,\nu_i)&=\set{(a,b)\in\MM(P) \mid (\alpha_{i_1}(a,b),\ldots,\alpha_{i_{l-1}}(a,b))\in\mathbf{\Sigma}(\nu_i)};\\
\MM(P,\nu_i,\Sigma)&=\set{(a,b)\in\MM(P) \mid (\alpha_{i_1}(a,b),\ldots,\alpha_{i_{l-1}}(a,b))=\Sigma}\quad\textrm{for $\Sigma\in\mathbf{\Sigma}(\nu_i)$}.
\end{align*}

Observe that by this definition, for each internal node $\nu_i$ of $\Psi$ with children $\nu_{i_1},\ldots,\nu_{i_l}$ we get the partition

\begin{align*}
 \MM(P,\nu_i)&=\bigcup_{1\leq j\leq l} \MM(P,\nu_{i_j}).
\end{align*}

Therefore, by  construction the sets $\MM(P,\nu_i)$ with $\nu_i$ a leaf of $\sigtree$
(so for $\nu_3,\nu_7,\nu_{10},\nu_{15}$) form a partition of $\MM(P)$.
With a further refinement it follows that
\[
 \MM(P)=\bigcup_{\nu_i\text{ leaf of }\Psi} \quad \bigcup_{\Sigma\in\mathbf{\Sigma}(\nu_i)} \MM(P,\nu_i,\Sigma)
\]
and the proof below boils down to showing that $\MM(P,\nu_i,\Sigma)$ is reversible for each leaf $\nu_i$ of $\sigtree$ and each $\Sigma\in\mathbf{\Sigma}(\nu_i)$.

Once this is established we get an upper bound on $\dim^*(P)$ just by counting the number of sets in our partition of $\MM(P)$, namely
\begin{align*}
\dim^*(P)\leq \displaystyle{\sum_{\textrm{$\nu_i$ leaf of $\sigtree$}}} \norm{\mathbf{\Sigma}(\nu_i)}
&= \norm{\mathbf{\Sigma}(\nu_3)}+\norm{\mathbf{\Sigma}(\nu_7)}+\norm{\mathbf{\Sigma}(\nu_{10})}+\norm{\mathbf{\Sigma}(\nu_{15})}\\
&= 2 + 2\cdot6 + 2\cdot6\cdot2\cdot2 + 2\cdot6\cdot6\cdot4\cdot2 = 638.
\end{align*}

Note that the calculation for the particular summands follows from \eqref{eq:signature-def} and the edge labelings in Figure~\ref{fig:signature_tree}.
Our proof will follow a depth-first, left-to-right search of the signature tree $\sigtree$, defining
the functions $\alpha_i$ one by one in that order, and showing that for each $\Sigma \in \mathbf{\Sigma}(\nu_i)$ the set $\MM(P, \nu_i, \Sigma)$  is reversible when encountering a leaf $\nu_i$.
Hence, the tree $\sigtree$ also serves as a road map of the proof.
Moreover, for the reader's convenience we included a table collecting all functions $\alpha_i$ and their meanings, see Table~\ref{tab:functions}.
It is not necessary to read this table now, but it might be helpful while going through the main proof.

\begin{table}
\begin{center}
\begin{tabular}{p{1.4cm}p{1.5cm}p{9cm}}
\hline
  Function & Section & Meaning\\ \hline\hline
  $\phi$ &
  \ref{sec:preliminaries} &
  Proper $3$-coloring of the intersection graph of the subtrees $T_x (x\in X)$ of $T$.
  That is, $\phi(x)\neq \phi(y)$ whenever $V(T_x)\cap V(T_y)\neq \emptyset$.\\
  \hline
  $\alpha_1$ &
  \ref{sec:node-1-and-2} &
  Assigns `left' or `right' to pairs $(a,b)$ depending on whether $a^T$ lies to the left or to the right of $b^T$ in $T$.\\
  \hline
  $\alpha_2$ &
  \ref{sec:node-1-and-2} &
  Records the answer to the question: ``Is there an element $q\in B(u_{ab})$ such that $a\leq q$ in $P$?"\\
  \hline
  $\alpha_4$ &
  \ref{sec:node-4-and-5} &
  Records the answer to the question: ``Is there an element $q\in B(u_{ab})\cap B(p_{ab})$ such that $a\leq q$ in $P$?"\\
  \hline
  $\alpha_5$ &
  \ref{sec:node-4-and-5} &
$\alpha_5(a,b)=(\phi(x_{ab}), \phi(y_{ab}), \phi(z_{ab}))$, where $B(u_{ab})=\{x_{ab}, y_{ab}, z_{ab}\}$ and $x_{ab}, y_{ab}, z_{ab}$ satisfy:
\begin{itemize}
\item $a\leq x_{ab}\not\leq b$ in $P$;
\item $B(u_{ab})\cap B(p_{ab})=\{y_{ab},z_{ab}\}$;
\item $a\not\leq y_{ab} \leq b \textrm{ in } P$;
\item $a_0 \leq y_{ab} \textrm{ in } P$;
\item $a\not\leq z_{ab} \textrm{ in } P$, and
\item $y_{ab}\in B(u_{ab})\cap B(w_{ab})$.
\end{itemize} \\
  \hline
  $\alpha_6$ &
  \ref{sec:node-6} &
  Records the answer to the question ``Is $B(u_{ab})\cap B(w_{ab})=\{x_{ab},y_{ab}\}$?".\\
  \hline
  $\alpha_8$ &
  \ref{sec:node-8} &
  Given $(a,b)\in \textrm{MM}(P,\nu_8)$, let $\Sigma\in \mathbf{\Sigma}_8$ be such that $(a,b)\in\textrm{MM}(P,\nu_8,\Sigma)$. Then $\alpha_8(a,b)=\psi_{8,\Sigma}(a,b)$, where $\psi_{8,\Sigma}$ is a $2$-coloring of the graph $S_\Sigma$ of special $2$-cycles.\\
  \hline
  $\alpha_9$ &
  \ref{sec:node-9} &
  Given $(a,b)\in \textrm{MM}(P,\nu_9,\Sigma)$ for some $\Sigma\in\mathbf{\Sigma}(\nu_9)$, function $\alpha_9$ records the color $\psi_{9,\Sigma}(a,b)$, where $\psi_{9,\Sigma}$ is a coloring of $K_\Sigma$.\\
  \hline
  $\alpha_{11}$ &
  \ref{sec:node-11} &
  $\alpha_{11}(a,b)=(\phi(x_{ab}), \phi(y_{ab}), \phi(z_{ab}))$, where $B(u_{ab})=\{x_{ab}, y_{ab}, z_{ab}\}$ and $x_{ab}, y_{ab}, z_{ab}$ satisfy:
\begin{itemize}
\item $B(u_{ab})\cap B(\p_{ab})=\{x_{ab},y_{ab}\}$;
\item $a\leq x_{ab}\not\leq y_{ab}$ in $P$, and
\item $a\not\leq y_{ab}\leq b$ in $P$.
\end{itemize} \\
  \hline
  $\alpha_{12}$ &
  \ref{sec:node-12} &
  Given $(a,b)\in\textrm{MM}(P,\nu_{12})$, $\alpha_{12}(a,b)$ records the answers to the questions ``Is $a\leq z_{ab}$ in $P$", ``Is $z_{ab}\leq b$ in $P$?", and ``Is $a_0\leq x_{ab}$ in $P$?".\\
  \hline
  $\alpha_{13}$ &
  \ref{sec:node-13} &
  Given $(a,b)\in \textrm{MM}(P,\nu_{13},\Sigma)$, function $\alpha_{13}$ records the color of the pair $(a,b)$ in the $4$-coloring $\psi_{13,\Sigma}$ of the graph $J_{\Sigma}$.
The purpose of $\alpha_{13}$ is to get rid of $2$-cycles in $\textrm{MM}(P,\nu_{13},\Sigma)$. \\
  \hline
  $\alpha_{14}$ &
  \ref{sec:node-14}   &
  Given $(a,b)\in\textrm{MM}(P,\nu_{14},\Sigma)$, function $\alpha_{14}$ records the color of the pair $(a,b)$ in the $2$-coloring  $\psi_{14,\Sigma}$ of the graph $\hat K_{\Sigma}$.
The purpose of $\alpha_{14}$ is to get rid of strict alternating cycles of length at least $3$ in  $\textrm{MM}(P,\nu_{14},\Sigma)$.
\end{tabular}
\end{center}
\caption{Table of functions and their meanings}
\label{tab:functions}
\end{table}

Now that the necessary definitions are introduced and the preliminary
observations are made, we are about to
consider the nodes of the signature tree one by one, stating and proving many
technical statements along the way. At this point the reader
might legitimately wonder why it all works, that is, what are the basic ideas underlying our
approach. While we are unable to offer a general intuition---indeed, this is why
we believe that better insights into these posets remain
to be obtained---we can at least explain a couple of the strategies
we repeatedly apply in our proofs.

A first strategy builds on the fact
that when choosing three times an element in a $2$-element set, some element
is bound to be chosen at least twice: As a toy example, suppose that
$\set{(a_i,b_i)}_{i=1}^k$ is a strict alternating cycle with $k \geq 3$
in some subset $I\subseteq \MM(P)$ which we are trying to prove is reversible.
Suppose further that we somehow previously established that the $a_{i}^{T}$--$b_{i+1}^{T}$
path in $T$ includes a specific edge $uv$ of $T$ for at least three distinct
indices $i \in \set{1, \dots, k}$; which indices is not important,
so let us say this happens for indices $1,2,3$.
Then by Observation~\ref{obs:hitting-vertex-or-edge} the relation $a_{i} \leq b_{i+1}$ hits $B(u)\cap B(v)$
for each $i=1,2,3$. Given that $|B(u)\cap B(v)|=2$, this implies that
some element $x\in B(u)\cap B(v)$ is hit by two of these relations, that is, we have
$a_{i} \leq x \leq b_{i+1}$ and $a_{j} \leq x \leq b_{j+1}$ in $P$ for some $i, j \in \set{1,2,3}$
with $i < j$. However, this implies $a_{i} \leq x \leq b_{j+1}$ in $P$, contradicting the fact that
the alternating cycle is strict. (Here we use that $k \geq 3$.)
Therefore, the alternating cycle $\set{(a_i,b_i)}_{i=1}^k$ could not have
existed in the first place. More generally, when analyzing certain situations we claim cannot occur,
we will typically easily find two relations $c_1 \leq d_1$ and
$c_2 \leq d_2$ in $P$ both hitting $B(u)\cap B(v)$ for some edge $uv$ of $T$, and
which are incompatible, in the sense that they cannot hit the same element. The work then
goes into pinning down a third relation $c_3 \leq d_3$ in $P$ which is incompatible with the first two, and
yet hits $B(u)\cap B(v)$. (The fact that $a_0 \leq b$ in $P$ for every $b \in \max(P)$ will often be helpful here.)

A second strategy is to see certain strict alternating cycles as inducing a graph on
$\MM(P)$, and then study and exploit properties of said graph. This is natural for
strict alternating cycles of length $2$: Any such cycle $(a_{1}, b_{1}), (a_{2}, b_{2})$
can be seen as inducing an edge between vertex $(a_{1}, b_{1})$ and vertex $(a_{2}, b_{2})$.
If we somehow can show that the resulting graph has bounded chromatic number, then
we can consider a corresponding coloring of the pairs, and we will know that within a color class
there are no strict alternating cycles of length $2$ left. Thus, by doing so we `killed' all such cycles by partitioning the pairs in a constant number of sets.
Such a strategy is used twice in the proof, when considering nodes $\nu_{8}$ and $\nu_{13}$ of the signature
tree $\Psi$. We also use a variant of it tailored to handle certain
strict alternating cycles of length at least $3$ and involving a {\em directed} graph
on $\MM(P)$, when considering nodes $\nu_{9}$ and $\nu_{14}$ of $\Psi$. \\

We now turn to the proof.
From now on we will use the following notations for a given pair $(a,b)\in\MM(P)$:
We let $u_{ab}:=a^T \land b^T$, $p_{ab}:=\p(u_{ab})$, and denote by $v_{ab}$ and $w_{ab}$ the neighbors of $u_{ab}$ in $T$ towards $a^T$ and $b^T$, respectively.
Figure~\ref{fig:a-left-of-b} illustrates the newly defined nodes.

\begin{figure}[t]
 \centering
 \includegraphics{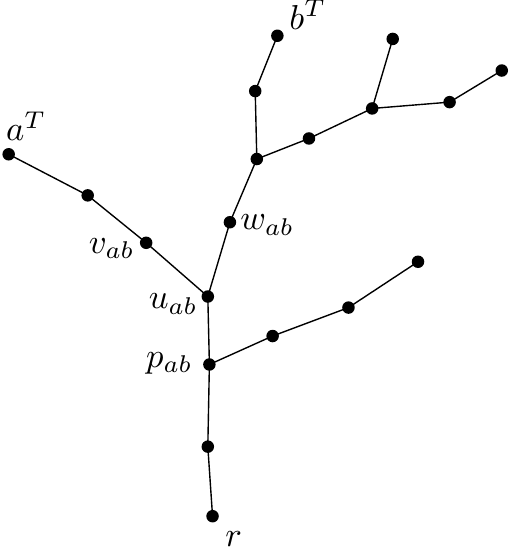}
 \caption{\label{fig:a-left-of-b}Pair $(a,b)\in\MM(P)$ with $\alpha_1(a,b)=\textrm{left}$ and the corresponding nodes $p_{ab}$, $u_{ab}$, $v_{ab}$, and $w_{ab}$ in $T$.}
\end{figure}

\subsection{Nodes \texorpdfstring{$\nu_1$ and $\nu_2$ and their respective functions $\alpha_1$ and $\alpha_2$}.}
\label{sec:node-1-and-2}
We start with the definition of $\q{sign:left-or-right}$, which belongs to the node $\nu_1$ of the signature tree $\sigtree$.

\begin{center}
\medskip
\fbox{\begin{varwidth}{0.9\textwidth}
For each $(a,b)\in \MM(P,\nu_1)=\MM(P)$, $\q{sign:left-or-right}(a,b) \in\mathbf{\Sigma}_1= \{\textrm{left}, \textrm{right}\}$ encodes whether $a^T$ is to the left or to the right of $b^T$ in $T$.
\end{varwidth}}
\medskip
\end{center}
Let us proceed with node $\nu_2$ and its function $\alpha_2$.
\begin{center}
\medskip
\fbox{\begin{varwidth}{0.9\textwidth}
For each $(a,b)\in \MM(P,\nu_2)=\MM(P)$, we let $\q{sign:a<x}(a,b) \in \mathbf{\Sigma}_2=\{\textrm{yes}, \textrm{no}\}$ be the answer to the question
\begin{center}
 ``Is there an element $q\in B(u_{ab})$ with $a\leq q$ in $P$?''.
\end{center}
\end{varwidth}}
\medskip
\end{center}
Note that it might be the case that all elements of $B(u_{ab})$ are incomparable with $a$ (so the answer is ``no"), while there is always an element $q\in B(u_{ab})$ such that $q\leq b$ in $P$ as the comparability $a_0\leq b$ hits the bag $B(u_{ab})$.
In the next section we treat all the pairs $(a,b)\in\MM(P)$ with $\q{sign:a<x}(a,b)=\textrm{no}$ and it turns out that they can easily be reversed.

\subsection{First leaf of \texorpdfstring{$\sigtree$: the node $\nu_3$}.}
Incomparable pairs of $\MM(P,\nu_3)$ received one of the two possible signatures in $\mathbf{\Sigma}(\nu_3)=\set{(\textrm{left},\no),(\textrm{right},\no)}$.
We start by showing that pairs with these signatures are reversible.

\begin{claim}\label{claim:a-inc}
$\MM(P, \nu_3, \Sigma)$ is reversible
for each $\Sigma \in \mathbf{\Sigma}(\nu_3)$.
\end{claim}
\begin{proof}
Let $\Sigma \in \mathbf{\Sigma}(\nu_3)=\set{(\textrm{left},\no),(\textrm{right},\no)}$.
We will assume that $\Sigma=(\textrm{left},\no)$, thus $\q{sign:left-or-right}(a,b)=\textrm{left}$ for pairs $(a, b) \in \MM(P, \nu_3, \Sigma)$.
In the other case it suffices to exchange the notion of left and right in the following argument.
(We note that we will start with that assumption in all subsequent proofs, for the same reason.)

Arguing by contradiction,
suppose that there is a strict alternating cycle $\set{(a_i,b_i)}_{i=1}^k$ in $\MM(P, \nu_3, \Sigma)$.
Thus $a_i \leq b_{i+1}$ in $P$ for all $i$ (cyclically).
Let $b_j^T$ be leftmost in $T$ among all the $b_i^T$'s ($i \in \set{1, \dots, k}$).
The node $a_j^T$ is to the left of $b_j^T$ (as $(a_j,b_j)\in\MM(P,\nu_3,\Sigma)$, so $\q{sign:left-or-right}(a_j,b_j)=\textrm{left}$), and thus to the left of all the $b_i^T$'s.
Hence, the path from $a_j^T$ to $b_{j+1}^T$ in $T$ goes through the node $u_{a_jb_j}$.
By Observation~\ref{obs:hitting-vertex-or-edge}, the relation $a_j\leq b_{j+1}$ in $P$ hits $B(u_{a_jb_j})$, contradicting
$\q{sign:a<x}(a_j,b_j)=\textrm{no}$ (recall that $(a_j,b_j)\in\MM(P,\nu_3,\Sigma)$, and thus $\alpha_2(a_j,b_j)=\no$).
\end{proof}
As a consequence of Claim~\ref{claim:a-inc}, all the pairs $(a,b)\in\MM(P)$ being considered in the following satisfy $\q{sign:a<x}(a,b)=\textrm{yes}$ (see how this fact can be read from the signature tree $\sigtree$).

\subsection{Nodes \texorpdfstring{$\nu_4$ and $\nu_5$ and their respective functions $\alpha_4$ and $\alpha_5$}.}
\label{sec:node-4-and-5}
For all pairs $(a,b)$ in $\MM(P,\nu_4)$ it holds that there is $q\in B(u_{ab})$ such that $a\leq q$ in $P$, but it is not clear whether there is such an element being also contained in $B(p_{ab})$.
It is the purpose of function $\q{sign:a-goes-below-u}$ to distinguish the two possible cases at this point.
\begin{center}
\medskip
\fbox{\begin{varwidth}{0.9\textwidth}
Given $(a,b)\in\MM(P,\nu_4)$, let $\q{sign:a-goes-below-u}(a,b)\in\set{\yes,\no}$ be the answer to the following question about $(a,b)$:
$$\textrm{``Is there an element $q\in B(u_{ab})\cap B(p_{ab})$ with $a\leq q$ in $P$?''.}$$
\end{varwidth}}
\medskip
\end{center}
Before defining the function $\q{sign:H-coloring}$ we first show some useful properties of pairs in $\MM(P, \nu_5, \Sigma)$ for $\Sigma \in {\bf \Sigma}(\nu_5)$.  Note that these pairs $(a, b)$ satisfy $\q{sign:a<x}(a, b)=\yes$
and $\q{sign:a-goes-below-u}(a,b)=\no$.

\begin{claim}\label{claim:u-ordering}
Let  $\Sigma \in {\bf \Sigma}(\nu_5)$ and suppose that $\set{(a_i,b_i)}_{i=1}^k$ is an alternating cycle in $\MM(P, \nu_5, \Sigma)$.
Let $u_i$ denote $u_{a_ib_i}$, for each $i \in \{1, 2, \dots, k\}$.
Then
\begin{enumerate}
\item $u_i$ and $u_{i+1}$ are comparable in $T$ for each $i\in \{1,2, \dots, k\}$, and
\item there is an index $j \in \{1,2, \dots, k\}$ such that $u_j \leq u_i$ in $T$ for each $i \in \{1,2, \dots, k\}$.
\end{enumerate}
\end{claim}
\begin{proof}
Let $\Sigma\in\mathbf{\Sigma}(\nu_5)=\set{(\textrm{left},\yes,\no),(\textrm{right},\yes,\no)}$.
Again we may assume $\Sigma=(\textrm{left},\yes,\no)$ as the other case is symmetrical.
Thus $\q{sign:left-or-right}(a_i, b_i)=\textrm{left}$ for each $i \in \set{1,\ldots,k}$.

We denote $u_{a_ib_i},w_{a_ib_i},p_{a_ib_i}$ by $u_i, w_i, p_i$ respectively, for each $i\in \{1,2, \dots, k\}$.

To prove the first item observe that since $\q{sign:a-goes-below-u}(a_i, b_i)=\textrm{no}$ for all pairs $(a_i,b_i)$,
and since $a_i\leq b_{i+1}$ in $P$, we have $u_i < b_{i+1}^T$ in $T$.
Indeed, otherwise the path from $a_i^T$ to $b_{i+1}^T$ in $T$ would go through $u_i$ and $p_i$, and hence $a_i\leq b_{i+1}$ would hit $B(u_i)\cap B(p_i)$, contradicting $\q{sign:a-goes-below-u}(a_i,b_i)=\no$.
Clearly $u_{i+1} < b_{i+1}^T$ in $T$.
Therefore, $\set{u_i,u_{i+1}}< b_{i+1}^T$ in $T$, which makes $u_i$ and $u_{i+1}$ comparable in $T$.

The second item follows immediately from the first item and Observation~\ref{obs:comp-nodes}.
\end{proof}

Thanks to Claim~\ref{claim:u-ordering} we know that for every $\Sigma \in {\bf \Sigma}(\nu_5)$,
each alternating cycle in $\MM(P, \nu_5, \Sigma)$ can be written
as $\set{(a_i,b_i)}_{i=1}^k$ in such a way that $u_{a_1b_1} \leq u_{a_ib_i}$ in $T$ for $i \in \{1, \dots, k\}$.
We may further assume that the pair $(a_1,b_1)$ is chosen in such a way that
\begin{enumerate}
 \item if $\q{sign:left-or-right}(a_1,b_1)=\textrm{left}$ then $b_1^T$ is to the right of $b_i^T$ in $T$ for each $i \in \{2, \dots, k\}$ satisfying $u_{a_1b_1} = u_{a_ib_i}$.
 \item if $\q{sign:left-or-right}(a_1,b_1)=\textrm{right}$ then $b_1^T$ is to the left of $b_i^T$ in $T$ for each $i \in \{2, \dots, k\}$ satisfying $u_{a_1b_1} = u_{a_ib_i}$.
\end{enumerate}
Note that the pair $(a_1,b_1)$ is uniquely defined; we call it the \emph{root} of the alternating cycle.

Now for each $\Sigma\in\mathbf{\Sigma}(\nu_5)$ and $(a,b)\in\MM(P,\nu_5,\Sigma)$ we take a closer look at elements in $B(u_{ab})$.
The bag $B(u_{ab})$ consists of three distinct elements; let us denote them  $x_{ab}$, $y_{ab}$, $z_{ab}$.
Given that $\q{sign:a<x}(a,b)=\yes$ and $\q{sign:a-goes-below-u}(a,b)=\no$, we may assume without loss of generality
\begin{align}
 &a \leq x_{ab} \not\leq b \textrm{ in } P;\label{eq:x-element} \\
 &B(u_{ab})\cap B(p_{ab}) = \set{y_{ab},z_{ab}}.
\end{align}
Recall that the $u_{ab}w_{ab}$ edge lies on the path from $r$ to $b^T$ in $T$.
This implies that the relation $a_0 \leq b$ hits $B(u_{ab})\cap B(w_{ab})$.
Clearly, it cannot hit $x_{ab}$, and thus $a_0 \leq b$ hits at least one of $y_{ab}, z_{ab}$.
Let us suppose without loss of generality that this is the case for $y_{ab}$.
It follows that
\begin{align}
 &a\not\leq y_{ab} \leq b \textrm{ in } P; \\
 &a_0 \leq y_{ab} \textrm{ in } P; \\
 &a\not\leq z_{ab} \textrm{ in } P; \\
 &y_{ab}\in B(u_{ab})\cap B(w_{ab}).\label{eq:y_ab-in-bags}
\end{align}
With these notations, we define $\q{sign:H-coloring}$ in the following way:
\begin{center}
\medskip
\fbox{\begin{varwidth}{0.9\textwidth}
For each $\Sigma \in {\bf \Sigma}(\nu_{5})$ and pair $(a,b) \in \MM(P, \nu_{5}, \Sigma)$, we let
\[
\q{sign:H-coloring}(a,b) := (\phi(x_{ab}), \phi(y_{ab}), \phi(z_{ab})).
\]
\end{varwidth}}
\medskip
\end{center}
(Recall that $\phi(w)$ is the color of the element $w\in X$ in the $3$-coloring $\phi$ of the intersection graph defined by the subtrees $T_x$ ($x\in X$), and that
$x_{ab}$, $y_{ab}$, $z_{ab}$ have distinct colors.)
Hence there are $6$ possible answers for $\q{sign:H-coloring}(a,b)$.
In the following when considering nodes $\nu_i$ of $\sigtree$ that are descendants of $\nu_5$,
all we will need is that min-max pairs $(a,b)\in\MM(P,\nu_i,\Sigma)$
have the same value $\q{sign:H-coloring}(a,b)$ but the value itself will not be important.
This is why $\sigtree$ does not branch at $\nu_5$.

\subsection{Node \texorpdfstring{$\nu_6$ and its function $\alpha_6$}.}
\label{sec:node-6}

Before defining the next function $\q{sign:no-xy-edge}$, let us show some useful properties of strict alternating cycles in
$\MM(P, \nu_6, \Sigma)$ for $\Sigma \in {\bf \Sigma}(\nu_6)$.
These properties will be used not only when considering the second leaf $\nu_7$ of $\sigtree$
but also later on when considering the third leaf $\nu_{10}$.

Now, recall that pairs $(a,b)\in\MM(P,\nu_6)$ satisfy
\begin{itemize}
 \item $\q{sign:a<x}(a,b)=\textrm{yes}$, and hence there is $q\in B(u_{ab})$ with $a\leq q$ in $P$,
 \item $\q{sign:a-goes-below-u}(a,b)=\textrm{no}$, and hence there is no $q\in B(u_{ab})\cap B(p_{ab})$ such that $a\leq q$ in $P$,
 \item the elements of $B(u_{ab})$ can be labeled with $x_{ab}$, $y_{ab}$, $z_{ab}$ such that \eqref{eq:x-element}-\eqref{eq:y_ab-in-bags} hold.
\end{itemize}
We need these properties and the mentioned labeling for the following claim.
\begin{claim}\label{claim:uk}
Let  $\Sigma \in {\bf \Sigma}(\nu_6)$ and
suppose that $\set{(a_i,b_i)}_{i=1}^k$ is a strict alternating cycle in $\MM(P, \nu_6, \Sigma)$
with root $(a_1, b_1)$.
Let $u_i, w_i$ denote $u_{a_ib_i},w_{a_ib_i}$ respectively, for each $i \in \{1, 2, \dots, k\}$.
Then  $u_1 < w_1 \leq u_k < b_1^T$ in $T$.
\end{claim}
\begin{proof}
We denote $p_{a_ib_i}, x_{a_ib_i}, y_{a_ib_i}, z_{a_ib_i}$ by $p_i, x_i, y_i, z_i$ respectively, for each $i\in \{1,2, \dots, k\}$.
We assume that $\alpha_1(a,b)=\textrm{left}$ for each $(a,b)\in\MM(P,\nu_6,\Sigma)$.
In particular, $\q{sign:left-or-right}(a_i, b_i)=\textrm{left}$ for each $i\in \{1,2, \dots, k\}$.

The path from $a_k^T$ to $b_1^T$ in $T$ cannot go through the edge $u_kp_k$, since otherwise by Observation~\ref{obs:hitting-vertex-or-edge}
the relation $a_k\leq b_1$ would hit $B(u_k)\cap B(p_k)$ which contradicts the fact that $\q{sign:a-goes-below-u}(a_k,b_k)=\no$.
This implies $u_k < b_1^T$ in $T$.

Next we prove that $u_1\neq u_k$.
Suppose to the contrary that $u_1=u_k$.
Then $u_k$ lies on the paths from $a_k^T$ to $r$ and from $b_1^T$ to $r$ in $T$, implying that $a_k^T\land b_1^T \geq u_k$ in $T$.

If $a_k^T\land b_1^T = u_k$ then the path from $a_k^T$ to $b_1^T$ in $T$ goes through $u_k$.
Thus, the relation $a_k\leq b_1$ hits $B(u_k)=\{x_k,y_k,z_k\}$ and hence $a_k\leq x_k\leq b_1$ in $P$
(as $a_k\not\leq y_k$ and $a_k\not\leq z_k$ in $P$).
Since $B(u_k)=B(u_1)$ and $\q{sign:H-coloring}(a_1,b_1)=\q{sign:H-coloring}(a_k,b_k)$ implying $\phi(x_k)=\phi(x_1)$, we get $x_k=x_1$.
Now $a_1\leq x_1=x_k\leq b_1$ in $P$ gives a contradiction.

If $a_k^T\land b_1^T > u_k$ in $T$ then it follows that $v_k\leq b_1^T$ in $T$.
By Observation~\ref{obs:left-right}~\ref{item:b-is-right} we conclude that $b_1^T$ is left of $b_k^T$ in $T$.
Since $u_k=u_1$, this contradicts the fact
that $(a_1, b_1)$ is the root of  $\set{(a_i,b_i)}_{i=1}^k$.

Therefore, $u_1\neq u_k$ as claimed, and $u_1 < u_k < b_1^T$ in $T$. Given the definition of $w_1$ and the fact that $u_k < b_1^T$ in $T$,
we deduce $u_1 < w_1 \leq u_k < b_1^T$ in $T$.
\end{proof}

\begin{claim}\label{claim:u2-in-component}
Let  $\Sigma \in {\bf \Sigma}(\nu_6)$ and
suppose that $\set{(a_i,b_i)}_{i=1}^k$ is a strict alternating cycle in $\MM(P, \nu_6, \Sigma)$
with root $(a_1, b_1)$.
Let $u_i, w_i$ denote $u_{a_ib_i},w_{a_ib_i}$ respectively, for each $i \in \{1, 2, \dots, k\}$.
Then $u_1<w_1\leq u_i$ in $T$ for each $i\in \{2, \dots, k\}$.
\end{claim}

\begin{proof}
We denote $v_{a_ib_i}, p_{a_ib_i}, x_{a_ib_i}, y_{a_ib_i}, z_{a_ib_i}$ by $v_i, p_i, x_i, y_i, z_i$ respectively, for each $i\in \{1,2, \dots, k\}$.
We assume that $\alpha_1(a,b)=\textrm{left}$ for each $(a,b)\in\MM(P,\nu_6,\Sigma)$.
In particular, $\q{sign:left-or-right}(a_i, b_i)=\textrm{left}$ for each $i\in \{1,2, \dots, k\}$.

By Claim~\ref{claim:uk} we have $w_1 \leq u_k$ in $T$. Arguing by contradiction,
suppose that  $w_1 \not\leq u_i$ for some $i\in \{2, \dots, k-1\}$, and let $i$ be the largest such index.
Thus, $w_1 \leq u_{i+1}$ in $T$. Note also that in this case we must have $k \geq 3$.

Since $u_i$ and $u_{i+1}$ are comparable in $T$ (by Claim \ref{claim:u-ordering})
and $u_1$ is minimal in $T$ among all the $u_i$'s, we obtain $u_1=u_i< w_1 \leq u_{i+1}$ in $T$.

Observe that $u_i\leq a_i^T\land b_{i+1}^T$ in $T$, as $u_i<\set{a^T_i,b^T_{i+1}}$ in $T$. If $u_i= a_i^T\land b_{i+1}^T$ then the path from $a_i^T$ to $b_{i+1}^T$ in $T$ goes through $u_i$ .
Thus, the relation $a_i \leq b_{i+1}$ hits $B(u_i)=\{x_i,y_i,z_i\}$, and it follows that $a_i \leq x_i \leq b_{i+1}$ in $P$ (as $a_i\nleq y_i$ and $a_i\nleq z_i$).
Since $B(u_i)=B(u_1)$ and $\phi(x_i)=\phi(x_1)$ (because $\q{sign:H-coloring}(a_1,b_1)=\q{sign:H-coloring}(a_i,b_i)$),
we deduce that $x_i=x_1$.
But then $a_1 \leq x_1 = x_i \leq b_{i+1}$ in $P$, which is a contradiction. (Recall that $i \geq 2$.)

If $u_1=u_i<a_i^T \land b_{i+1}^T$ in $T$ then we must have $w_1\leq a_i^T\land b_{i+1}^T$ in $T$,
since $a_i^T\land b_{i+1}^T$ has to be an internal node of the path from $u_i$ to $b_{i+1}^T$ in $T$
and since $w_1$ is the neighbor of $u_i$ on that path (as $u_i<w_1\leq u_{i+1} < b_{i+1}^T$ in $T$).
In particular, this implies $u_i<w_1 < a_i^T$ in $T$, and it follows that $v_i=w_1$.
As $a_i^T$ is left of $b_i^T$ in $T$, and since $v_i=w_1 < b_1^T$ in $T$, by Observation~\ref{obs:left-right}~\ref{item:b-is-right} we obtain that $b_1^T$ is left of $b_i^T$, which contradicts the fact that $(a_1, b_1)$ is the root of  $\set{(a_i,b_i)}_{i=1}^k$.
This completes the proof.
\end{proof}
Now let us define the function $\q{sign:no-xy-edge}$.
\begin{center}
\medskip
\fbox{\begin{varwidth}{0.9\textwidth}
We set $\q{sign:no-xy-edge}(a,b)$ to be the answer to the following question:
$$\textrm{``Is $B(u_{ab})\cap B(w_{ab})=\set{x_{ab},y_{ab}}$?''.}$$
\end{varwidth}}
\medskip
\end{center}
Note that $y_{ab}$ always belongs to the intersection, and hence $\q{sign:no-xy-edge}$ tells us whether $x_{ab}$ or $z_{ab}$ is the other element in $B(u_{ab})\cap B(w_{ab})$.
If the answer to this question is ``no'', then our signature tree leads us to the second leaf of $\Psi$, leaf $\nu_7$.

\subsection{Second leaf of \texorpdfstring{$\sigtree$: the node $\nu_7$}.}
\label{sec:second_leaf}
In this section we show that incomparable pairs in $\MM(P,\nu_7,\Sigma)$ are reversible for each $\Sigma\in \mathbf{\Sigma}(\nu_7)$.
Recall that $\MM(P,\nu_7)$ is a subset of $\MM(P,\nu_6)$, allowing us to use the observations and claims from the previous section.
Moreover, for every pair $(a,b)\in\MM(P,\nu_7)$ we have that $\q{sign:no-xy-edge}(a,b)=\textrm{no}$, implying that $B(u_{ab})\cap B(w_{ab})=\{y_{ab},z_{ab}\}$.
\begin{claim}\label{claim:no-xy-edge}
$\MM(P, \nu_7, \Sigma)$ is reversible
for each $\Sigma \in \mathbf{\Sigma}(\nu_7)$.
\end{claim}
\begin{proof}
Let $\Sigma \in \mathbf{\Sigma}(\nu_7)$.
We assume that $\alpha_1(a,b)=\textrm{left}$ for each $(a,b)\in\MM(P,\nu_7,\Sigma)$.
In particular, $\q{sign:left-or-right}(a_i, b_i)=\textrm{left}$ for each $i\in \{1,2, \dots, k\}$.

Arguing by contradiction, suppose that there is a strict alternating cycle $\set{(a_i,b_i)}_{i=1}^k$ in $\MM(P, \nu_7, \Sigma)$ with root $(a_1,b_1)$.
We have $u_1 < w_1 \leq u_2$ in $T$
by Claim~\ref{claim:u2-in-component}, and in particular  $a_1^T \land b_2^T = u_1$.
Thus, the path from $a_1^T$ to $b_2^T$ in $T$ includes the edge $u_1w_1$.
Hence, the relation $a_1 \leq b_2$ hits $B(u_1)\cap B(w_1)\subset \set{x_1,y_1,z_1}$.
Since $a_1 \leq x_1$ and $a_1\inc\{y_1,z_1\}$ in $P$ we obtain $x_1\in B(u_1)\cap B(w_1)$.
Recalling that we also have $y_{1}\in B(u_{1})\cap B(w_{1})$, it follows that
$B(u_1)\cap B(w_1)=\{x_1, y_1\}$,
contradicting $\q{sign:no-xy-edge}(a_1,b_1)=\textrm{no}$.
\end{proof}

\subsection{Node \texorpdfstring{$\nu_8$ and its function $\alpha_8$}.}
\label{sec:node-8}
Before we study strict alternating cycles in $\MM(P,\nu_8)$, let us recall some useful properties of incomparable pairs in $\MM(P,\nu_8)$.
Each pair $(a,b)\in \MM(P,\nu_8)$ satisfies
\begin{itemize}
 \item $\q{sign:a-goes-below-u}(a,b)=\textrm{no}$, and hence there is no $q\in B(u_{ab})\cap B(p_{ab})$ such that $a\leq q$ in $P$,
 \item the elements of $B(u_{ab})$ can be labeled with $x_{ab}$, $y_{ab}$, $z_{ab}$ such that \eqref{eq:x-element}-\eqref{eq:y_ab-in-bags} hold,
 \item $\q{sign:no-xy-edge}(a,b)=\textrm{yes}$, and hence $B(u_{ab})\cap B(w_{ab})=\{x_{ab},y_{ab}\}$.
 \end{itemize}
We proceed with an observation about $\MM(P, \nu_8, \Sigma)$ for fixed $\Sigma \in {\bf \Sigma}(\nu_8)$.

\begin{claim}\label{claim:path-simple}
Let  $\Sigma \in {\bf \Sigma}(\nu_8)$ and
suppose that $\set{(a_i,b_i)}_{i=1}^k$ is a strict alternating cycle in $\MM(P, \nu_8, \Sigma)$
with root $(a_1, b_1)$.
Let $u_i$ denote $u_{a_ib_i}$ for each $i \in \{1, 2, \dots, k\}$.
Then $u_1 < w_1 \leq u_2 < b_1^T$ in $T$.
\end{claim}

\begin{proof}
We denote $w_{a_ib_i},p_{a_ib_i}, x_{a_ib_i}, y_{a_ib_i}, z_{a_ib_i}$ by $w_i, p_i, x_i, y_i, z_i$ respectively, for each $i\in \{1,2, \dots, k\}$.
We assume that $\alpha_1(a,b)=\textrm{left}$ for each $(a,b)\in\MM(P,\nu_8,\Sigma)$.
In particular, $\q{sign:left-or-right}(a_i, b_i)=\textrm{left}$ for each $i\in \{1,2, \dots, k\}$.

If $k=2$ then the claim follows from Claim~\ref{claim:uk}.
So we assume $k\geq 3$ from now on.
By Claim~\ref{claim:u2-in-component} we already know $u_1 < w_1 \leq u_2$ in $T$, and thus
it remains to show $u_2 < b_1^T$ in $T$.

\begin{figure}
 \centering
 \includegraphics{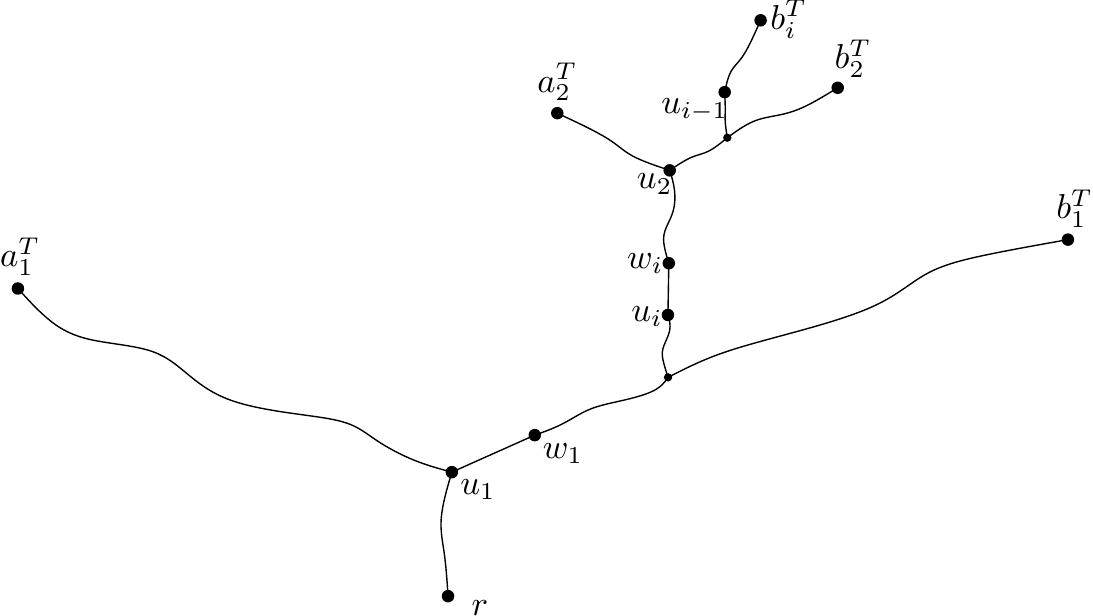}
 \caption{\label{fig:claim17} A possible situation in the proof of Claim~\ref{claim:path-simple}.}
\end{figure}

Arguing by contradiction suppose that $u_2\not<b_1^T$ in $T$.
Let $i \in \{3, \dots, k\}$ be smallest such that $u_i\ngeq u_2$ in $T$.
There is such an index since $u_k<b_1^T$ in $T$ by Claim \ref{claim:uk}, and thus $u_k\ngeq u_2$ in $T$.
See Figure~\ref{fig:claim17} for an illustration of this and upcoming arguments.

By our choice of $i$, we have $u_1<u_2\leq u_{i-1}$ in $T$.
Note that $a_{i-1} \leq b_i$ in $P$,
which combined with $\q{sign:a-goes-below-u}(a_{i-1},b_{i-1})=\textrm{no}$ yields $u_{i-1} < b_i^T$ in $T$.
Hence $u_2\leq u_{i-1}<b_i^T$ in $T$.

Since $u_2<b_i^T$ and $u_i<b_i^T$ in $T$, the two nodes $u_2$ and $u_i$ are comparable in $T$, and thus
$u_i < u_2$ since $u_i\ngeq u_2$ in $T$. Combining this with $u_2<b_i^T$ in $T$ we further deduce that
$w_i\leq u_2$ in $T$.
On the other hand, $w_1 \leq u_i$ in $T$ by Claim~\ref{claim:u2-in-component}.
To summarize, we have $u_1< w_1 \leq u_i<w_i\leq u_2< b_i^T$ in $T$.

Now consider the path from $a_1^T$ to $b_2^T$ in $T$.
Since $w_1 < u_2 \leq b_2^T$ and  $w_1 \nleq a_1^T$ in $T$, this path goes through $u_1$, and thus includes the edge $u_iw_i$.
Hence the relation $a_1\leq b_2$ hits the set $B(u_i)\cap B(w_i)$, the latter being equal to $\set{x_i,y_i}$
since $\q{sign:no-xy-edge}(a_i,b_i)=\textrm{yes}$.
Therefore, $a_1 \leq x_i \leq b_2$ or $a_1\leq y_i \leq b_2$ in $P$.
But this implies $a_i \leq x_i \leq b_2$ or $a_1 \leq y_i \leq b_i$ in $P$, a contradiction in both cases to the properties of a strict alternating cycle (recall that $i\in\set{3,\ldots,k}$).
\end{proof}

Given $\Sigma \in {\bf \Sigma}(\nu_8)$, we
say that pairs $(a,b), (a',b')\in \MM(P, \nu_8, \Sigma)$ form a \emph{special $2$-cycle} if, exchanging
$(a,b)$ and $(a',b')$ if necessary, we have
\begin{enumerate}
 \item $a \leq b'$ and $a' \leq b$ in $P$, and
 \item $u_{ab} < w_{ab} \leq u_{a'b'} < w_{a'b'} < b^T$ in $T$.
\end{enumerate}
The first requirement is simply that $(a,b), (a',b')$ is a (strict) alternating cycle.
(Note that every alternating cycle of length $2$ is strict.)
As a consequence of this and Claim~\ref{claim:path-simple}, we know that the first three inequalities of the second requirement are satisfied.
So the question here is whether also $w_{a'b'}<b^T$ holds in $T$.
An implication of the second requirement is
that the paths from $a^T$ to $b'^T$ and  from $a'^T$ to $b^T$ in $T$ both go through the edge $u_{a'b'}w_{a'b'}$ of $T$.
Note also that the pair $(a,b)$ is the root this special $2$-cycle.

Let $S_{\Sigma}$ be the graph with vertex set $\MM(P, \nu_8, \Sigma)$
where distinct pairs $(a,b),(a',b') \in \MM(P, \nu_8, \Sigma)$ are adjacent if and only if
they form a special $2$-cycle.

\begin{claim}\label{claim:special-2-cycle}
The graph $S_{\Sigma}$ is bipartite for each $\Sigma \in {\bf \Sigma}(\nu_8)$.
\end{claim}
\begin{proof}
Arguing by contradiction,
suppose that there is an odd cycle $C=\set{(a_i,b_i)}_{i=1}^k$ in $S_{\Sigma}$
for some $\Sigma \in {\bf \Sigma}(\nu_8)$.
We may assume that $C$ is induced.
Let $u_i :=u_{a_ib_i}$, $w_i :=w_{a_ib_i}$, $x_i :=x_{a_ib_i}$ and $y_i :=y_{a_ib_i}$ for each $i \in \{1, \dots, k\}$.

First we consider the case $k=3$.
Since $u_1,u_2,u_3$ are pairwise comparable in $T$,
we may assume that $u_1<u_3<u_2$ in $T$ (recall that consecutive $u_i$'s are distinct by property (ii) of special $2$-cycles).
By the definition of special $2$-cycles,
we then obtain $u_2<w_2<\set{b_1^T,b_2^T,b_3^T}$ in $T$.
Thus the paths from $a_1^T$ to $b_2^T$, from $a_2^T$ to $b_3^T$, and from $a_3^T$ to $b_1^T$ in $T$ all go through the edge $u_2w_2$.
This implies that two relations must hit the same element and therefore
there is $i\in\set{1,2,3}$ and $q\in B(u_2)\cap B(w_2)$ such that
$a_i \leq q \leq b_{i+1}$ and $a_{i+1} \leq q \leq b_{i+2}$ in $P$ (indices are taken cyclically).
However, this gives $a_{i+1}\leq b_{i+1}$ in $P$, a contradiction.

Next consider the case $k \geq 5$.
We will show that $C$ has a chord, contradicting the fact that $C$ is induced.
(We remark that the parity of $k$ will not be used here, only that $k \geq 5$.)

We may suppose that $u_2$ is maximal in $T$ among all the $u_i$'s.
We may also assume without loss of generality $u_1\leq u_3 < u_2$ in $T$.
(Recall that by property (ii) of special $2$-cycles $u_i$ and $u_{i+1}$ are comparable in $T$ and distinct for each $i\in \set{1, \dots, k}$.)
Let $i,j$ be such that $\{i,j\}=\{3,4\}$ and $u_j < u_i$ in $T$. (Note that
$u_{j} \neq u_{i}$ since $(a_j,b_j)$, $(a_i,b_i)$ form a special $2$-cycle.)
We claim that
\begin{equation}\label{eq:w-inequality}
w_j < w_i \leq w_2 \textrm{ and } w_1 \leq w_i
\end{equation}
in $T$.

The inequality $w_j<w_i$ follows from the fact that $(a_j,b_j)$ and $(a_i,b_i)$ form a special $2$-cycle,
and thus in particular $u_j < w_j \leq u_i < w_i$ in $T$.
For the inequality $w_i\leq w_2$ we do a case distinction.
If $i=3$ then $(a_2,b_2)$ and $(a_i,b_i)$ form a special $2$-cycle with $u_i < w_i \leq u_2 < w_2$ in $T$.
If $i=4$ then $(a_2,b_2)$, $(a_3,b_3)$ as well as $(a_3,b_3)$, $(a_i,b_i)$ form special $2$-cycles.
In particular, $\set{w_2,w_i} < b_3^T$ in $T$.
This makes $w_2$ and $w_i$ comparable in $T$ and by the choice of $u_2$ we have $w_i \leq w_2$, as desired.
Besides this, we also have $w_1 < w_2$ in $T$ (as $(a_1,b_1),(a_2,b_2)$ form a special $2$-cycle),
which makes $w_1$ and $w_i$ comparable in $T$.
Since we have $u_1\leq u_i$ by our choice of $i$ and by the assumption that $u_1\leq u_3$ in $T$, it follows that $w_1\leq w_i$ in $T$ and \eqref{eq:w-inequality} is proven.

Now, we are going to argue that the $a_1^T$--$b_2^T$ path, the $a_0$--$b_1^T$ path, the
$a_j^T$--$b_i^T$ path, and the $a_i^T$--$b_j^T$ path all go through the edge $u_iw_i$ in $T$.

For the $a_1^T$--$b_2^T$ path note that $u_1 < w_1 \leq w_i \leq w_2 < b_2^T$ in $T$.
This implies that this path has to go first
from $a_1^T$ down to $u_1$ and then pursue with the $u_{1}$--$b_{2}^{T}$ path, which includes $w_{i}$
by the previous inequalities, and thus also its parent $u_{i}$ (since $u_{1} < w_{i}$ in $T$).
Hence, it includes the edge $u_iw_i$ of $T$.

For the $a_0$--$b_1^T$ path it suffices to observe that $r < u_i < w_i \leq w_2 < b_1^T$ in $T$.
Similarly, for the $a_j^T$--$b_i^T$ path, notice that $u_j < w_j \leq u_i < w_i < b_i^T$ in $T$.
Finally, for the $a_i^T$--$b_j^T$ path, observe that $u_i < w_i < b_j^T$ in $T$.

Using Observation~\ref{obs:hitting-vertex-or-edge}, it follows that the relations $a_1\leq b_2$, $a_0 \leq b_1$, $a_j\leq b_i$ and $a_i\leq b_j$ in $P$ all hit $B(u_i)\cap B(w_i)=\set{x_i,y_i}$. Clearly,
\[
a_j \leq y_i \leq b_i \quad\quad \textrm{and} \quad\quad a_i \leq x_i \leq b_j
\]
in $P$.

Now, in $P$ we either have $a_0\leq x_i\leq b_1$ and $a_1\leq y_i\leq b_2$, or $a_0\leq y_i\leq b_1$ and $a_1\leq x_i\leq b_2$.
This implies $a_i\leq b_1$ and $a_1\leq b_i$, or $a_j\leq b_1$ and $a_1\leq b_j$.

In the first case, $(a_1, b_1), (a_i, b_i)$ is an alternating cycle of length $2$ (and thus is strict).
Recall that $w_2<b_1^T$ in $T$, which together with \eqref{eq:w-inequality} yields $w_1\leq w_i < b_1^T$ in $T$.
Furthermore, applying Claim~\ref{claim:u2-in-component} on $(a_1, b_1), (a_i, b_i)$ we obtain
$u_1 \neq u_i$, implying $w_1 \neq w_i$, and hence
$w_1<w_i$ in $T$. It follows that $u_1<w_1\leq u_i<w_i<b_1^T$ in $T$, which shows that $(a_1,b_1), (a_i,b_i)$
is a special $2$-cycle. This gives us a chord of the cycle $C$, a contradiction.

In the second case, $(a_1, b_1), (a_j, b_j)$ is a (strict) alternating cycle.
Moreover, $w_1 \leq w_i$ and $w_j \leq w_i$ in $T$ (see \eqref{eq:w-inequality}),
which makes $w_1$ and $w_j$ comparable in $T$.
Again by Claim~\ref{claim:u2-in-component} we get $w_1\neq w_j$.
If $w_1<w_j$ then it also holds that $u_1<w_1\leq u_j<w_j<b_1^T$ in $T$ (as $w_j\leq w_2< b_1^T$).
If $w_j<w_1$ then it follows that $u_j<w_j\leq u_1<w_1<b_j^T$ in $T$ (as $w_1\leq w_i< b_j^T$).
Thus in both cases $(a_1, b_1), (a_j, b_j)$ is a special $2$-cycle and a chord in $C$, a contradiction.
This completes the proof.
\end{proof}

Using Claim~\ref{claim:special-2-cycle}, for each
 $\Sigma \in {\bf \Sigma}(\nu_8)$ let
$\psi_{8, \Sigma}\colon\MM(P, \nu_8, \Sigma) \to \{1,2\}$ be
a (proper) $2$-coloring of $S_{\Sigma}$.
The function $\q{sign:colors_special_graph}$ then simply records the color of a pair in this coloring:
\begin{center}
\medskip
\fbox{\begin{varwidth}{0.9\textwidth}
For each $\Sigma\in\mathbf{\Sigma}(\nu_8)$ and each pair $(a,b)\in\MM(P,\nu_8,\Sigma)$, we let
$$\q{sign:colors_special_graph}(a,b) := \psi_{8, \Sigma}(a,b).$$
\end{varwidth}}
\medskip
\end{center}

\subsection{Node \texorpdfstring{$\nu_9$ and its function $\alpha_9$}.}
\label{sec:node-9}

By the definition of $\q{sign:colors_special_graph}$ there is no special $2$-cycle in
$\MM(P, \nu_9, \Sigma)$, for every $\Sigma \in {\bf \Sigma}(\nu_9)$.
This will be used in the definition of the function $\q{sign:colors_K}$.

In order to define $\q{sign:colors_K}$ we first need to introduce an auxiliary directed graph.
For each $\Sigma \in {\bf \Sigma}(\nu_9)$, let
$K_{\Sigma}$ be the directed graph with vertex set $\MM(P, \nu_9, \Sigma)$
where for any two distinct pairs $(a_1,b_1), (a_2,b_2) \in \MM(P, \nu_9, \Sigma)$
there is an arc from $(a_1,b_1)$ to $(a_2,b_2)$ in $K_{\Sigma}$ if there exists a strict alternating cycle $\set{(a_i',b_i')}_{i=1}^k$
in $\MM(P, \nu_9, \Sigma)$ with root $(a_1', b_1')$ such that $(a_1', b_1')=(a_1,b_1)$ and $(a_2', b_2')=(a_2,b_2)$.
In that case we say that the arc $f$ {\em is induced by} the strict alternating cycle $\{(a_i',b_i')\}_{i=1}^k$.
(Note that there could possibly be different strict alternating cycles inducing the same arc $f$.)

\begin{claim}\label{claim:k-sigma-1}
 Let $\Sigma\in\mathbf{\Sigma}(\nu_9)$. Then for each arc $((a_1,b_1),(a_2,b_2))$ in $K_{\Sigma}$ we have
\begin{enumerate}
 \item\label{item:k-sigma-1} $x_{1}\leq y_{2}$, and
 \item\label{item:k-sigma-2} $y_{1}\leq z_{2}\leq b_1$
\end{enumerate}
in $P$, where $x_i :=x_{a_ib_i}$, $y_i :=y_{a_ib_i}$, and $z_i :=z_{a_ib_i}$
for $i=1,2$.
\end{claim}

\begin{proof}
Let $u_i :=u_{a_ib_i}$, $p_i :=p_{a_ib_i}$, and $w_i :=w_{a_ib_i}$ for $i=1,2$.
By the definition of an arc in $K_{\Sigma}$ and by Claim~\ref{claim:path-simple} it holds that $u_{1}<w_{1}\leq u_{2}<b_1^T$ in $T$.
Thus, the path from $a_1^T$ to $b_2^T$ in $T$ goes through $u_{1},w_{1},u_{2}$ and $w_{2}$.
Hence, the relation $a_1\leq b_2$ (which exists by the definition of an arc in $K_{\Sigma}$) hits
$B(u_{1})\cap B(w_{1})=\set{x_{1},y_{1}}$ and $B(u_{2})\cap B(w_{2})=\set{x_{2},y_{2}}$.
It cannot hit $y_{1}$ (otherwise $a_1\leq y_{1}\leq b_1$ in $P$) nor $x_{2}$ (otherwise $a_2\leq x_{2}\leq b_2$ in $P$).
It follows that $a_1\leq x_{1}\leq y_{2}\leq b_2$ in $P$ by Observation~\ref{obs:hitting-vertex-or-edge},
and~\ref{item:k-sigma-1} is proven.

For~\ref{item:k-sigma-2} observe that the relation $a_0 \leq b_1$ hits $\set{x_{1},y_{1}}=B(u_{1})\cap B(w_{1})$ and
$\set{y_{2},z_{2}}=B(p_{2})\cap B(u_{2})$. It cannot hit $x_{1}$ nor $y_{2}$ since otherwise $a_1\leq b_1$ in $P$ as we have $a_1\leq x_1\leq y_2$ in $P$.
Hence we obtain $a_0 \leq y_1 \leq z_2 \leq b_1$ in $P$ by Observation~\ref{obs:hitting-vertex-or-edge}.
\end{proof}
Let us remark that it is possible to strengthen the statement of Claim~\ref{claim:k-sigma-1}.
In fact, strict inequalities $x_1<y_2$ and $y_1<z_2$ hold in \ref{item:k-sigma-1} and \ref{item:k-sigma-2}, respectively.
Indeed, since $\phi$ is a proper coloring and $\q{sign:H-coloring}(a_1,b_1)=\q{sign:H-coloring}(a_2,b_2)$, we have $\phi(x_1)=\phi(x_2)\neq \phi(y_2)$ and $\phi(y_1)=\phi(y_2)\neq \phi(z_2)$, implying that $x_1\neq y_2$ and $y_1\neq z_2$.
For our purposes the non-strict inequalities of Claim~\ref{claim:k-sigma-1} (and also in forthcoming claims) suffice.
As we do not want give more arguments than needed, we will not always prove the strongest possible statement in the following.

\begin{claim}\label{claim:k-sigma-2}
Let $\Sigma\in\mathbf{\Sigma}(\nu_9)$. Suppose that $(a_1,b_1),(a_2,b_2),(a_3,b_3)\in\MM(P,\nu_9,\Sigma)$
are three distinct pairs such that $((a_1,b_1),(a_2,b_2))$ is an arc in $K_{\Sigma}$ and $u_{1}<u_{3}<w_{3}\leq u_{2}$ in $T$,
where $u_i :=u_{a_ib_i}$ and $w_i :=w_{a_ib_i}$ for each $i\in \set{1,2,3}$. Then
\begin{enumerate}
 \item $x_{1}\leq z_{3}\leq y_{2}$, and
 \item $y_{1}\leq y_{3}\leq z_{2}\leq b_1$
\end{enumerate}
in $P$, where  $x_i :=x_{a_ib_i}$, $y_i :=y_{a_ib_i}$, and $z_i :=z_{a_ib_i}$  for each $i\in \set{1,2,3}$.
\end{claim}
\begin{proof}
Let $p_i :=p_{a_ib_i}$ for each $i\in \set{1,2,3}$.
We have $x_{1}\leq y_{2}$ and $y_{1}\leq z_{2}\leq b_1$ in $P$ by Claim~\ref{claim:k-sigma-1}.
Since $u_{1}<u_{3}<w_{3}\leq u_{2}$ in $T$, the relations $x_{1}\leq y_{2}$ and $y_{1}\leq z_{2}$
hit $B(u_{3})\cap B(w_{3})=\set{x_{3},y_{3}}$. They cannot hit the same element since
otherwise $a_1\leq x_{1}\leq z_{2}\leq b_1$ in $P$.

If $x_{1}\leq y_{3}\leq y_{2}$ in $P$ then $y_{1}\leq x_{3}\leq z_{2}$,
and we conclude $a_1\leq x_{1}\leq y_{3}\leq b_3$ and $a_3\leq x_{3}\leq z_{2}\leq b_1$ in $P$.
Thus, $(a_1,b_1)$ and $(a_3,b_3)$ form an alternating cycle of length $2$.
Applying Claim~\ref{claim:path-simple} on the pairs $(a_1,b_1)$, $(a_2,b_2)$ we obtain in particular $u_{2}<b_1^T$ in $T$.
Together with our assumptions it follows that both $u_{3}$ and $w_{3}$ lie on the path from $u_{1}$ to $b_1^T$ in $T$.
Hence, $u_{1}<w_{1}\leq u_{3}<w_{3}<b_1^T$ in $T$, and therefore $(a_1,b_1), (a_3,b_3)$
is a special $2$-cycle, contradicting the fact that $\psi_{8, \Sigma}(a_1,b_1)=\q{sign:colors_special_graph}(a_1,b_1)=\q{sign:colors_special_graph}(a_3,b_3)=\psi_{8, \Sigma}(a_3,b_3)$.

We conclude that $x_{1}\leq x_{3}\leq y_{2}$ and $y_{1}\leq y_{3}\leq z_{2}$ in $P$.
It remains to show that $x_{1}\leq z_{3}\leq y_{2}$ in $P$.
For this, note that $x_{1}\leq x_{3}$ hits $B(p_{3})\cap B(u_{3})=\set{y_{3},z_{3}}$.
It cannot hit $y_{3}$ as otherwise $a_1\leq x_{1}\leq y_{3}\leq z_{2}\leq b_1$ in $P$.
This implies $x_{1}\leq z_{3}\leq x_{3}\leq y_{2}$ in $P$, as desired.
\end{proof}

We are now ready to prove our main claim about $K_{\Sigma}$ ($\Sigma\in\mathbf{\Sigma}(\nu_9)$), namely
that $K_{\Sigma}$ is bipartite.  (We consider  a directed graph to be bipartite if its underlying undirected graph is.)

\begin{claim}\label{claim:k-sigma-bipartite_nu_9}
 The graph $K_{\Sigma}$ is bipartite for each $\Sigma\in\mathbf{\Sigma}(\nu_9)$.
 \end{claim}
\begin{proof}
 Arguing by contradiction,
suppose that there is an odd undirected cycle $C=\set{(a_i,b_i)}_{i=1}^k$ in $K_{\Sigma}$ for some $\Sigma\in\mathbf{\Sigma}(\nu_9)$.
Let $u_i :=u_{a_ib_i}$, $w_i :=w_{a_ib_i}$, $x_i :=x_{a_ib_i}$, $y_i :=y_{a_ib_i}$,
and $z_i :=z_{a_ib_i}$ for each $i \in \{1, \dots, k\}$.
We may assume that $\alpha_1(a_i,b_i)=\text{left}$ for each $i \in \{1, \dots, k\}$.

Consider the cyclic sequence of nodes $(u_1,u_2,\ldots,u_k)$.
It might be the case that some of the nodes coincide.
In order to avoid this, we modify the sequence as follows:
If $u_i=u_j$ for some  $i, j \in \{1, \dots, k\}$ with $i < j$, then
consider the two cyclic sequences $(u_i,u_{i+1},\ldots,u_{j-1})$ and $(u_j,u_{j+1},\ldots, u_{i-1})$ (thus the second
one contains $u_k$, and also $u_1$ if $i > 1$).
Since $k$ is odd, exactly one of the two cyclic sequences has odd length.
We replace the original sequence by that one, and repeat this process as long as some
node appears at least twice in the current cyclic sequence.

We claim that at every stage of the above modification process the cyclic sequence
$S=(u_{i_1},u_{i_2},\ldots,u_{i_\ell})$ under consideration satisfies the following property:
For every $s\in \set{1, \dots, \ell}$ there is an index $j \in \{1, \dots, k\}$ such that
$u_{i_s}=u_j$ and $u_{i_{s+1}}=u_{j+1}$ (taking indices cyclically in each case, as expected).
This property obviously holds at the start, so let us show that it remains true during the rest of the procedure.
Thus suppose that the current cyclic sequence $S=(u_{i_1},u_{i_2},\ldots,u_{i_\ell})$ satisfies the property,
and that we modify it because of two indices $p, q \in \set{1, \dots, \ell}$ with $p < q$ such that $u_{i_p}=u_{i_q}$.
Without loss of generality we may assume that the resulting odd sequence is $S'=(u_{i_p},\ldots,u_{i_{q-1}})$.
We only need to show that there is an index $j \in \{1, \dots, k\}$ such that
$u_{i_{q-1}}=u_j$ and $u_{i_{p}}=u_{j+1}$, since $u_{i_{q-1}}$ and $u_{i_p}$  are the only two consecutive nodes
in $S'$ that were not consecutive in $S$.
Then it suffices to take $j \in \{1, \dots, k\}$ such that $u_{i_{q-1}}=u_j$ and $u_{i_{q}}=u_{j+1}$,
which exists since $S$ satisfies our property, and observe that $u_{i_{p}}=u_{i_{q}}=u_{j+1}$.
Therefore, the property holds at every step, as claimed.

Recalling that any two consecutive nodes in the original cyclic sequence $(u_1,u_2,\ldots,u_k)$
are distinct and comparable in $T$ (by Claim~\ref{claim:path-simple}),
it follows from the property considered above that this holds
at every step of the modification procedure, and thus in particular for the final sequence
$S=(u_{i_1},u_{i_2},\ldots,u_{i_\ell})$ resulting from the procedure.
In particular, $\ell \geq 3$, because the cycle has odd length.

Since $\ell$ is odd, there exists $m\in\set{1,\ldots,\ell}$
such that $u_{i_{m-1}}<u_{i_m}<u_{i_{m+1}}$ or $u_{i_{m-1}}>u_{i_m}>u_{i_{m+1}}$ in $T$.
Reversing the ordering of $C$ and the cyclic sequence $S$ if necessary,
we may assume without loss of generality $u_{i_{m-1}}<u_{i_m}<u_{i_{m+1}}$ in $T$.
Similarly, shifting the sequence $S$ cyclically if necessary, we may assume $m=1$.
Thus $u_{i_{\ell}}<u_{i_1}<u_{i_{2}}$ in $T$.

Let $w$ be the neighbor of $u_{i_1}$ on the  $u_{i_1}$--$u_{i_{2}}$ path in $T$.
Thus $w\leq u_{i_2}$ in $T$.
Now let $n \in \set{3, \dots, \ell}$ be minimal
such that $w\nleq u_{i_n}$ in $T$.
Since $u_{i_{\ell}}<w$ in $T$, this index exists.
As $u_{i_{n-1}}$ and $u_{i_n}$ are comparable in $T$, it follows that $u_{i_n}<w\leq u_{i_{n-1}}$ in $T$.
(This follows from the definition of $n$ if $n>3$, and from the fact that $w\leq u_{i_2}$ in $T$ if $n=3$.)
Furthermore we have $u_{i_n}\neq u_{i_1}$, because $n\neq 1$ and all nodes in $S$ are distinct. We conclude that
\begin{equation}
u_{i_n}<u_{i_1}<w\leq u_{i_{n-1}}\label{eq:u_i_n}
\end{equation}
in $T$.
Now, by the property that is maintained during the modification process of $S$, there exist indices $r,s,t\in\set{1,\ldots,k}$ such that
\begin{align}
 u_{i_{\ell}}&=u_{r}  &\text{and}& & u_{i_{1}}&=u_{r+1},\label{eq:u_r}\\
 u_{i_{1}}&=u_{s}  &\text{and}& & u_{i_{2}}&=u_{s+1},\label{eq:u_s}\\
 u_{i_{n-1}}&=u_{t}  &\text{and}& & u_{i_{n}}&=u_{t+1}.\label{eq:u_t}
\end{align}

\begin{figure}
 \centering
 \includegraphics[scale=1.0]{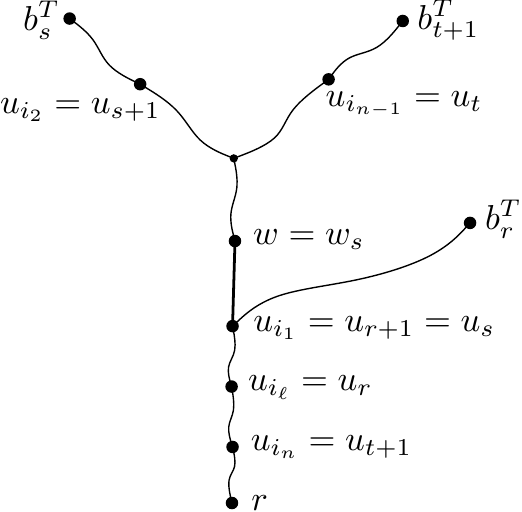}
 \caption{Possible situation in proof of Claim~\ref{claim:k-sigma-bipartite_nu_9}.
 Note that we could also have $u_r<u_{t+1}$ in $T$.}
 \label{fig:k-sigma-bipartite-nu9}
\end{figure}

It follows that $K_{\Sigma}$ contains the following arcs:
$((a_r,b_r),(a_{r+1},b_{r+1})),((a_s,b_s),(a_{s+1},b_{s+1}))$, and
$((a_{t+1},b_{t+1}),(a_t,b_t))$.
Applying Claim~\ref{claim:path-simple} on these arcs we obtain:
\begin{align}
 u_r<w_r\leq u_{r+1}<b_r^T,\label{eq:w_r}\\
 u_s<w_s\leq u_{s+1}<b_s^T,\label{eq:w_s}\\
 u_{t+1}<w_{t+1}\leq u_t<b_{t+1}^T\label{eq:w_t+1}
\end{align}
in $T$.
See Figure~\ref{fig:k-sigma-bipartite-nu9} for a possible configuration in $T$ and for upcoming arguments.
Since $u_{i_1}=u_s$ and $u_{i_{2}}=u_{s+1}$ we have $w=w_s$.
Hence with \eqref{eq:u_i_n} we get $u_{t+1}<u_s<w_s\leq u_t$ in $T$,
and by Claim~\ref{claim:k-sigma-2} it follows that
\[
 x_{t+1}\leq z_s\leq y_t \quad\text{and}\quad y_{t+1}\leq y_s\leq z_t\leq b_{t+1}
\]
in $P$.
Now applying Claim~\ref{claim:k-sigma-1} on the arc $((a_r,b_r),(a_{r+1},b_{r+1}))$ we get
\[
 x_r\leq y_{r+1}\quad\text{and}\quad y_r\leq z_{r+1}\leq b_r
\]
in $P$.
Since $\q{sign:H-coloring}(a_{r+1},b_{r+1})=\q{sign:H-coloring}(a_s,b_s)$ and $B(u_{r+1})=B(u_s)$ (because $u_{r+1}=u_s$),
we conclude that $y_{r+1}=y_s$ and $z_{r+1}=z_s$.
Using this and the derived inequalities we see that
\[
a_{t+1}\leq x_{t+1}\leq z_s=z_{r+1}\leq b_r \quad\text{and}\quad a_r\leq x_r\leq y_{r+1}=y_s\leq b_{t+1}
\]
in $P$.
Thus, $(a_r,b_r)$ and $(a_{t+1},b_{t+1})$ form an alternating cycle of length $2$.
In particular, this shows $r\neq t+1$ (otherwise we would have $a_r\leq b_{t+1}=b_r$ in $P$), and consequently (by Claim~\ref{claim:path-simple})
\begin{equation}
 u_r\neq u_{t+1}.\label{eq:u_r-neg}
\end{equation}
Observe that $u_r<u_{r+1}=u_{i_1}$ (by \eqref{eq:w_r} and \eqref{eq:u_r}) and $u_{t+1}=u_{i_n}<u_{i_1}$ (by \eqref{eq:u_t} and \eqref{eq:u_i_n}) in $T$, which makes $u_r$ and $u_{t+1}$ comparable in $T$.
Furthermore,
\[
 u_{i_1}\overset{\eqref{eq:u_r}}{=}u_{r+1}\overset{\eqref{eq:w_r}}{<}b_r^T \quad\text{ and }\quad u_{i_1}\overset{\eqref{eq:u_i_n}}{<}u_{i_{n-1}}\overset{\eqref{eq:u_t}}{=}u_t\overset{\eqref{eq:w_t+1}}{<}b_{t+1}^T
 \]
 in $T$, so all together we have
\[
 \set{u_r,u_{t+1}}<u_{i_1}<\set{b_r^T,b_{t+1}^T}
\]
in $T$.
But from this and \eqref{eq:u_r-neg} it follows that either $u_r<w_r\leq u_{t+1}<w_{t+1}\leq u_{i_1} <b_r^T$
or $u_{t+1}<w_{t+1}\leq u_r<w_r\leq u_{i_1}<b_{t+1}^T$ holds in $T$.
Both cases imply that $(a_r,b_r)$ and $(a_{t+1},b_{t+1})$ form a special $2$-cycle, which is our final contradiction.
\end{proof}

Using Claim~\ref{claim:k-sigma-bipartite_nu_9} we let
$\psi_{9,\Sigma}\colon \MM(P,\nu_9,\Sigma)\to\set{1,2}$ be a $2$-coloring of $K_{\Sigma}$,
for each $\Sigma\in\mathbf{\Sigma}(\nu_9)$.
The function $\q{sign:colors_K}$ then records the color of a pair in this coloring:
\begin{center}
\medskip
\fbox{\begin{varwidth}{0.9\textwidth}
For each $\Sigma\in\mathbf{\Sigma}(\nu_9)$ and each pair $(a,b)\in\MM(P,\nu_9,\Sigma)$, we let
\[
 \q{sign:colors_K}(a,b):=\psi_{9,\Sigma}(a,b).
\]
\end{varwidth}}
\medskip
\end{center}

\subsection{Third leaf of \texorpdfstring{$\sigtree$: $\nu_{10}$}.}
\label{sec:third_leaf}

Now suppose that there was a strict alternating cycle
$\set{(a_i,b_i)}_{i=1}^k$ with root $(a_1, b_1)$  in $\MM(P,\nu_{10},\Sigma)$, for some
$\Sigma\in\mathbf{\Sigma}(\nu_{10})$. Then, by the definition of $K_{\Sigma}$,
there is an arc from $(a_1, b_1)$ to $(a_2, b_2)$ in $K_{\Sigma}$, and hence
$\q{sign:colors_K}(a_1,b_1) \neq  \q{sign:colors_K}(a_2,b_2)$, a contradiction.
Therefore, there is no such cycle in $\MM(P,\nu_{10},\Sigma)$, and we have established the following claim:

\begin{claim}
 The set $\MM(P,\nu_{10},\Sigma)$ is reversible for each $\Sigma\in\mathbf{\Sigma}(\nu_{10})$.
\end{claim}

This concludes our study of the third leaf of $\Psi$.

\subsection{Node \texorpdfstring{$\nu_{11}$ and its functions $\alpha_{11}$}.}
\label{sec:node-11}

In this section we consider pairs $(a, b) \in \MM(P,\nu_{11})$.
Recall that in this case $(a,b)$ satisfies $\q{sign:a<x}(a,b) = \textrm{yes}$ and even stronger $\q{sign:a-goes-below-u}(a,b) = \textrm{yes}$, and hence there is $q\in B(u_{ab})\cap B(p_{ab})$ such that $a\leq q$ in $P$.

For every such pair $(a, b)$, denote the three elements in  $B(u_{ab})$
as $x_{ab}, y_{ab}, z_{ab}$ in such a way that
\begin{equation}
 B(u_{ab})\cap B(p_{ab})=\set{x_{ab},y_{ab}}\label{eq:xy-in-edge}
\end{equation}
and
\begin{align}
&a \leq x_{ab}\not\leq b;\label{eq:x-ineq} \\
&a\not\leq y_{ab} \leq b \label{eq:y-ineq}
\end{align}
in $P$. (Here we use that $\q{sign:a-goes-below-u}(a,b) = \textrm{yes}$ and that $a_0 \leq b$ hits $B(u_{ab})\cap B(p_{ab})$.)

The function $\q{sign:H-coloring-second}$ is defined similarly as $\q{sign:H-coloring}$ in Section~\ref{sec:second_leaf}:
\begin{center}
\medskip
\fbox{\begin{varwidth}{0.9\textwidth}
For each $\Sigma \in {\bf \Sigma}(\nu_{11})$ and pair $(a,b) \in \MM(P, \nu_{11}, \Sigma)$, we let
$$
\q{sign:H-coloring-second}(a,b) := (\phi(x_{ab}), \phi(y_{ab}), \phi(z_{ab})).
$$
\end{varwidth}}
\medskip
\end{center}
(Recall that $\phi$ is the $3$-coloring of $X$ defined earlier on which is such that $x,y\in X$
receive distinct colors whenever $V(T_x)\cap V(T_y)\neq\emptyset$;
in particular, the three colors $\phi(x_{ab}), \phi(y_{ab}), \phi(z_{ab})$ are distinct.)

\subsection{Node \texorpdfstring{$\nu_{12}$ and its functions $\alpha_{12}$}.}
\label{sec:node-12}

We start by defining function $\q{sign:comparability-status}$:
\begin{center}
\medskip
\fbox{\begin{varwidth}{0.9\textwidth}
For each $\Sigma \in {\bf \Sigma}(\nu_{12})$ and pair $(a,b) \in \MM(P, \nu_{12}, \Sigma)$, we let $\q{sign:comparability-status}(a,b)$ record the three yes/no answers to three independent questions about the pair $(a,b)$, namely
\begin{align}
&\textrm{``Is $a\leq z_{ab}$ in $P$?''};\tag{Q1}\label{eq:Q1} \\[0.8ex]
&\textrm{``Is $z_{ab} \leq b$ in $P$?''};\tag{Q2}\label{eq:Q2} \\[0.8ex]
&\textrm{``Is $a_0 \leq x_{ab}$ in $P$?''}.\tag{Q3}\label{eq:Q3}
\end{align}
\end{varwidth}}
\medskip
\end{center}
(Recall that $a_0<b$ in $P$ for all $b\in\max(P)$).
Formally speaking, $\q{sign:comparability-status}(a, b)$ is defined as the vector $(s_1, s_2, s_3) \in \{\yes, \no\}^3$ where
$s_i$ is the answer to the $i$-th question above, for $i=1,2,3$. Note that
we cannot have  $a\leq z_{ab}$ and $z_{ab} \leq b$ at the same time in $P$, and thus there are only
$6$ possible vectors of answers. (This is why the corresponding edge in the signature tree $\Psi$ is labeled
$6$ instead of $8$.)

\subsection{Node \texorpdfstring{$\nu_{13}$ and its functions $\alpha_{13}$: Dealing with strict alternating cycles of length $2$}.}\label{sec:2-cycles}
\label{sec:node-13}

Let us first summarize the properties of pairs in $\MM(P,\nu_{13})$.
For each $\Sigma\in \mathbf{\Sigma}(\nu_{13})$ and each $(a,b)\in\MM(P,\nu_{13},\Sigma)$ we have that
\begin{itemize}
 \item elements of $B(u_{ab})$ can be labeled with $x_{ab},y_{ab},z_{ab}$ such that \eqref{eq:xy-in-edge}-\eqref{eq:y-ineq} hold,
 \item $\phi(x_{ab}),\phi(y_{ab}),\phi(z_{ab})$ are pairwise distinct,
 \item $\phi(x_{ab})=\phi(x_{a'b'})$, $\phi(y_{ab})=\phi(y_{a'b'})$ and $\phi(z_{ab})=\phi(z_{a'b'})$ for every $(a',b')\in \MM(P,\nu_{13},\Sigma)$,
 \item all pairs of $\MM(P,\nu_{13},\Sigma)$ produce the same answers to \eqref{eq:Q1}-\eqref{eq:Q3}.
\end{itemize}

For each $\Sigma \in {\bf \Sigma}(\nu_{13})$ let
$J_{\Sigma}$ be the graph with vertex set $\MM(P, \nu_{13}, \Sigma)$ where
two distinct pairs $(a,b),(a',b') \in \MM(P, \nu_{13}, \Sigma)$ are adjacent if and only if
$(a,b),(a',b')$ is an alternating cycle.
Our goal in this section is to show that $J_{\Sigma}$ is $4$-colorable:

\begin{lemma}\label{claim:color-J}
For each $\Sigma \in {\bf \Sigma}(\nu_{13})$ there is a proper coloring of $J_{\Sigma}$ with $4$ colors.
\end{lemma}

To this aim, we show a number of properties of $2$-cycles in  $\MM(P,\nu_{13},\Sigma)$.

\begin{claim}\label{claim:equal-u}
Let $\Sigma \in {\bf \Sigma}(\nu_{13})$ and suppose that
$(a_1,b_1), (a_2,b_2)$ is a $2$-cycle in  $\MM(P,\nu_{13},\Sigma)$.
Let $u_i := u_{a_ib_i}$ for $i=1, 2$. Then $u_1\neq u_2$.
\end{claim}
\begin{proof}
Let $x_i:=x_{a_ib_i}$, $y_i:=y_{a_ib_i}$, and $z_i:=z_{a_ib_i}$  for $i=1,2$.
We may assume $\q{sign:left-or-right}(a_i, b_i)=\textrm{left}$ for $i=1,2$.
Arguing by contradiction suppose that $u_1=u_2$.  Exchanging $(a_1,b_1)$ and $(a_2,b_2)$ if necessary we may assume that
$b_1^T$ is left of $b_2^T$ in $T$.

Since in $T$ the node $a_1^T$ is left of $b_1^T$, which itself is left of $b_2^T$,
the path connecting $a_1^T$ to $b_2^T$ in $T$ goes through  $u_1$.
Thus, the relation $a_1\leq b_2$ hits $B(u_1)=\set{x_1,y_1,z_1}=B(u_2)=\set{x_2,y_2,z_2}$,
and hence $a_1\leq c\leq b_2$ in $P$ for some element $c\in B(u_1)$.
Recall that $a_i \leq x_i$ and   $a_i \nleq y_i$ in $P$ for $i=1,2$.  Thus $c \in \{x_1, z_1\}$.
Moreover, $(\phi(x_1), \phi(y_1), \phi(z_1)) =  (\phi(x_2), \phi(y_2), \phi(z_2))$
since $\q{sign:H-coloring-second}(a_1,b_1)=\q{sign:H-coloring-second}(a_2,b_2)$,
which implies $x_1 = x_2$, $y_1=y_2$, and $z_1 = z_2$.

If $c=x_1$ then $a_2 \leq x_2 = x_1 \leq b_2$ in $P$, a contradiction.
If $c=z_1$ then, using that $a_2 \leq z_2$ in $P$
(since $\q{sign:comparability-status}(a_1,b_1)=\q{sign:comparability-status}(a_2,b_2)$), we obtain
$a_2 \leq z_2 = z_1 \leq b_2$ in $P$, again a contradiction.
\end{proof}

By Claim~\ref{claim:equal-u} if $(a_1,b_1), (a_2,b_2)$ is a $2$-cycle in  $\MM(P,\nu_{13},\Sigma)$ for some $\Sigma \in {\bf \Sigma}(\nu_{13})$
then $u_{a_1b_1} \neq  u_{a_2b_2}$.
Let us say that the $2$-cycle is of {\em type 1} if the latter two nodes are comparable in $T$
(that is, $u_{a_1b_1} < u_{a_2b_2}$ or $u_{a_1b_1} > u_{a_2b_2}$ in $T$),
and of {\em type 2} otherwise.
By extension, each edge of the graph $J_{\Sigma}$ is either of type 1 or of type 2.
Let  $J_{\Sigma, i}$ denote the spanning subgraph of $J_{\Sigma}$ defined by the edges of type $i$, for $i=1,2$.
Thus $J_{\Sigma, 1}$ and $J_{\Sigma, 2}$ are edge disjoint, and $J_{\Sigma} = J_{\Sigma, 1} \cup J_{\Sigma, 2}$.
In what follows we will first show that $J_{\Sigma, 1}$ is bipartite, and then
considering a $2$-coloring of $J_{\Sigma, 1}$, we will prove that the two subgraphs
of $J_{\Sigma, 2}$ induced by the two color classes are bipartite.
This clearly implies  our main lemma, Lemma~\ref{claim:color-J}, that $J_{\Sigma}$ is $4$-colorable.

\begin{claim}\label{claim:w1-below-u_2}
Let $\Sigma \in {\bf \Sigma}(\nu_{13})$ and suppose that
$(a_1,b_1), (a_2,b_2)$ is a $2$-cycle in  $\MM(P,\nu_{13},\Sigma)$ of type 1.
Let $u_i := u_{a_ib_i}$ and $w_i:= w_{a_ib_i}$ for $i=1, 2$, and suppose further
that  $u_1< u_2$ in $T$. Then  $u_1 < w_1 \leq u_2$ in $T$.
\end{claim}
\begin{proof}
Arguing by contradiction, suppose that $w_1\nleq u_2$ in $T$.
Let $v_i:= v_{a_ib_i}$, $p_i := \p(u_i)$,
$x_i:=x_{a_ib_i}$, $y_i:=y_{a_ib_i}$, and $z_i:=z_{a_ib_i}$  for $i=1,2$.

First suppose that $v_1 \leq u_2$ in $T$.
Then the path from $a_2^T$ to $b_1^T$ in $T$ goes through $u_2,p_2,v_1$ and $u_1$.
Thus the relation $a_2\leq b_1$ hits $B(u_2)\cap B(p_2)=\set{x_2,y_2}$ and $B(v_1)\cap B(u_1)$.
Note that the two edges $u_2p_2$ and $v_1u_1$ may coincide (if $u_2=v_1$).
Since  $a_2\leq b_1$ cannot hit $y_2$ because $y_2\leq b_2$ in $P$, it hits $x_2$, and by Observation~\ref{obs:hitting-vertex-or-edge}
we then have
\begin{equation}
a_2\leq x_2\leq c\leq b_1\label{eq:x2-c}
\end{equation}
in $P$ for some $c \in B(v_1)\cap B(u_1)$.
Let $d$ be the element in $(B(v_1)\cap B(u_1)) \setminus \{c\}$.

The $a_1^T$--$u_1$ path and the $r$--$b_2^T$ path in $T$ both go through the edge $v_1u_1$.
Thus, the relations $a_1\leq x_1$ and $a_0 \leq b_2$ both hit $\set{c,d}$. Neither can hit
$c$ since otherwise $a_1\leq c\leq b_1$ or $a_2\leq c\leq b_2$ in $P$.
Hence, $a_1\leq d\leq x_1$ and $a_0 \leq d\leq b_2$ in $P$, which implies $a_0 \leq x_1$.
We then have $a_0 \leq x_2$ in $P$ as well, since by
$\q{sign:comparability-status}(a_1,b_1)=\q{sign:comparability-status}(a_2,b_2)$ both pairs $(a_1,b_1),(a_2,b_2)$ give the same answer to question~\eqref{eq:Q3}.

Clearly,  $a_0 \leq x_2$ hits  $\set{c,d}$. This relation cannot hit $d$ since otherwise $a_1\leq d\leq x_2$ and $x_2\leq b_1$ (by~\eqref{eq:x2-c}) would imply $a_1\leq b_1$ in $P$.
Thus $a_0 \leq c\leq x_2$ in $P$.
Given that $x_2\leq c$ in $P$ by \eqref{eq:x2-c}, we conclude $x_2=c$.
Using that $x_1=c\in B(u_1)=\{x_1,y_1,z_1\}$ and
$(\phi(x_1), \phi(y_1), \phi(z_1)) =  (\phi(x_2), \phi(y_2), \phi(z_2))$
(since $\q{sign:H-coloring-second}(a_1,b_1)=\q{sign:H-coloring-second}(a_2,b_2)$),
we further deduce that $x_1\neq y_2$ (because $\phi(x_1)=\phi(x_2)\neq\phi(y_2)$) and $x_1\neq z_2$ (because $\phi(x_1)=\phi(x_2)\neq \phi(z_2)$) and therefore $x_2 = c = x_1$.
However,  this implies $a_1\leq x_1=x_2\leq b_1$ in $P$, a contradiction.

Next assume that $v_1\nleq u_2$ in $T$.
Let $v'_1$ be the neighbor of $u_1$ on the $u_1$--$u_2$ path in $T$.
Thus $v_1' \neq w_1$ and $v_1' \neq v_1$.
The path from $a_2^T$ to $b_1^T$ in $T$ goes through $u_2,p_2,v'_1$ and $u_1$.
Thus, the relation $a_2\leq b_1$ hits $B(p_2)\cap B(u_2)=\set{x_2,y_2}$ and
$B(v'_1)\cap B(u_1)$. It cannot hit $y_2$ since otherwise $a_2\leq y_2\leq b_2$ in $P$.
By Observation~\ref{obs:hitting-vertex-or-edge}, we then have
\begin{equation}
a_2\leq x_2\leq c'\leq b_1\label{eq:x2-c'}
\end{equation}
for some $c' \in B(v'_1)\cap B(u_1)$. Let $d'$ be the element in
$(B(v'_1)\cap B(u_1)) \setminus \{c'\}$.

The paths from $r$ to $b_2^T$ and from $a_1^T$ to $b_2^T$ in $T$ both go through $u_1$ and $v'_1$.
Thus the two relations $a_0 \leq b_2$ and $a_1\leq b_2$ hit the set $\set{c',d'}$. They cannot hit $c'$ since otherwise we would get $c'\leq b_2$, implying $a_2\leq c'\leq b_2$ in $P$ by~\eqref{eq:x2-c'}. Hence, $a_0 \leq d'\leq b_2$ and $a_1\leq d'\leq b_2$ in $P$.

Observe that $\set{c',d'}\subseteq \set{x_1,y_1,z_1}=B(u_1)$, and that we have $c'\neq x_1$
(otherwise $a_1\leq x_1=c'\leq b_1$ in $P$ with~\eqref{eq:x2-c'})
and $d'\neq y_1$ (otherwise  $a_1\leq d'=y_1\leq b_1$ in $P$).

We claim that $a_0 \leq x_1$ in $P$.
If $d'=x_1$ this is obvious, so suppose that $d'=z_1$, which implies $c'=y_1$.
The relation $a_0 \leq d'$ clearly hits $B(p_1)\cap B(u_1)=\set{x_1,y_1}$.
If it hits $x_1$, then $a_0 \leq x_1$ in $P$.
If, on the other hand, it hits $y_1$, then together with~\eqref{eq:x2-c'} we obtain $a_2\leq c'=y_1\leq d'\leq b_2$ in $P$, a contradiction.
Hence $a_0 \leq x_1$ in $P$, as claimed.
We then have $a_0 \leq x_2$ in $P$ as well, since by
$\q{sign:comparability-status}(a_1,b_1)=\q{sign:comparability-status}(a_2,b_2)$ both pairs $(a_1,b_1),(a_2,b_2)$ give the same answer to question~\eqref{eq:Q3}.

The relation $a_0 \leq x_2$ also hits $\set{c',d'}$.
It cannot hit $d'$, since otherwise $a_1\leq d'\leq x_2\leq b_1$ in $P$ by~\eqref{eq:x2-c'}.
Thus we have $a_0 \leq c'\leq x_2$ in $P$.
Note that this yields $c'=x_2$, since we had $x_2\leq c'$ in $P$.
Using that $c'\in B(u_1)$ and
$(\phi(x_1), \phi(y_1), \phi(z_1)) =  (\phi(x_2), \phi(y_2), \phi(z_2))$
(since $\q{sign:H-coloring-second}(a_1,b_1)=\q{sign:H-coloring-second}(a_2,b_2)$),
we deduce that $x_2 = c' = x_1$.
However,  this implies $a_1\leq x_1=x_2\leq b_1$ in $P$, a contradiction.
\end{proof}

\begin{claim}\label{claim:2-cycle-comp}
Let $\Sigma \in {\bf \Sigma}(\nu_{13})$ and suppose that
$(a_1,b_1), (a_2,b_2)$ is a $2$-cycle in  $\MM(P,\nu_{13},\Sigma)$ of type 1.
Let $u_i := u_{a_ib_i}$, $x_i:=x_{a_ib_i}$, and $y_i:= y_{a_ib_i}$ for $i=1, 2$, and suppose further
that  $u_1< u_2$ in $T$. Then
\begin{enumerate}
 \item\label{item:c1} $a_1\leq y_2$, and
 \item\label{item:c2} $x_2\leq b_1$
\end{enumerate}
in $P$.
\end{claim}
\begin{proof}
Let $p_i := \p(u_i)$, $v_i:=v_{a_ib_i}$ and $w_i:=w_{a_ib_i}$ for $i=1,2$.
The path from $a_1^T$ to $b_2^T$ in $T$ has to go through $u_1$ since $u_1=\p(w_1)$, $w_1\not< a_1^T$ and $w_1\leq u_2<b_2^T$ in $T$ (by Claim~\ref{claim:w1-below-u_2}).
As a consequence, this path goes through $p_2$ and $u_2$.
Hence the relation $a_1\leq b_2$ hits $B(p_2)\cap B(u_2)=\set{x_2,y_2}$.
It cannot hit $x_2$, since otherwise $a_2\leq x_2\leq b_2$ in $P$. Therefore, it hits $y_2$, showing~\ref{item:c1}.

To show~\ref{item:c2},
observe that in $T$ at least one of the $r$--$b_1^T$ path and the $a_2^T$--$b_1^T$ path goes through $p_2$ and $u_2$.
Thus, at least one of the two relations $a_0 \leq b_1$ and $a_2\leq b_1$ hits $\set{x_2,y_2}$. Neither can hit $y_2$
since otherwise $a_1\leq y_2\leq b_1$ in $P$. Therefore, $x_2\leq b_1$ in $P$.
\end{proof}

\begin{claim}\label{claim:J1-triangles}
The graph $J_{\Sigma, 1}$ is triangle-free for each  $\Sigma \in {\bf \Sigma}(\nu_{13})$.
\end{claim}
\begin{proof}
Let $\Sigma \in {\bf \Sigma}(\nu_{13})$. Arguing by contradiction, suppose that
there is a triangle $(a_1,b_1), (a_2,b_2), (a_3,b_3)$ in $J_{\Sigma, 1}$.

Let $u_i := u_{a_ib_i}$, $p_i := \p(u_i)$, $w_i:=w_{a_ib_i}$,   $x_i:=x_{a_ib_i}$, and $y_i:= y_{a_ib_i}$ for each $i\in \set{1,2,3}$.
Since the nodes $u_1,u_2,u_3$ are pairwise comparable in $T$
and are all distinct (by Claim~\ref{claim:equal-u}), we may assume without loss of generality $u_1<u_2<u_3$ in $T$.

First we show that $a_0 \leq x_i$ holds in $P$ for some $i \in \set{1,2,3}$. Suppose this is not the case.
Consider the path from $r$ to $b_3^T$ in $T$. This path goes through the nodes $p_1,u_1,p_2,u_2,p_3$ and $u_3$.
Hence, the relation $a_0 \leq b_3$ hits $\set{x_1,y_1}$, $\set{x_2,y_2}$ and $\set{x_3,y_3}$.
By our assumption it hits $y_i$ for each $i \in \set{1,2,3}$, and
we have $a_0 \leq y_1\leq y_2\leq y_3$ in $P$ by Observation~\ref{obs:hitting-vertex-or-edge}.

If $u_2\leq b_1^T$ in $T$ then $a_0 \leq b_1$ hits $\set{x_2,y_2}$, and thus hits $y_2$ by our assumption.
Hence $y_2\leq b_1$ in $P$, which  using Claim~\ref{claim:2-cycle-comp}~\ref{item:c1} implies  $a_1\leq y_2\leq b_1$ in $P$, a contradiction.
Therefore, $u_2 \inc b_1^T$ in $T$.

The fact that  $u_1<u_2<u_3$ in $T$ further implies $$u_1< w_1 \leq u_2< w_2 \leq u_3$$ by Claim~\ref{claim:w1-below-u_2}.
Observe that the $a_1^T$--$b_2^T$ path, the $a_2^T$--$b_3^T$ path, and the $a_3^T$--$b_1^T$ path in $T$ all include the edge
$u_2w_2$. This is clear for the first two paths, and
follows from $u_2 \inc b_1^T$ in $T$ for the third one.
Thus the three relations $a_1 \leq b_2$, $a_2 \leq b_3$, $a_3 \leq b_1$ all hit
$B(u_2) \cap B(w_2)=:\set{c,d}$. Without loss of generality we have $a_1 \leq c \leq b_2$ in $P$.
This implies $a_2 \leq d \leq b_3$ in $P$, which in turn implies $a_3 \leq c\leq b_1$. However, it follows that $a_1 \leq c \leq b_1$ in $P$, a contradiction.

This shows that $a_0 \leq x_i$ holds in $P$ for some $i \in \set{1,2,3}$, as claimed.
Now, since $\q{sign:comparability-status}(a_1,b_1)=\q{sign:comparability-status}(a_2,b_2)=\q{sign:comparability-status}(a_3,b_3)$,
it follows that  $a_0 \leq x_i$ in $P$ for each $i \in \set{1,2,3}$.

Consider the relation $a_0 \leq x_3$ in $P$.
The path from $r$ to $u_3$ in $T$ goes through $p_2$ and $u_2$.
Thus, $a_0 \leq x_3$ hits $\set{x_2,y_2}$.
It cannot hit $x_2$ because otherwise $a_2\leq x_2\leq x_3$ in $P$, which together with $x_3\leq b_2$
(by Claim~\ref{claim:2-cycle-comp}~\ref{item:c2}) implies $a_2\leq x_3\leq b_2$ in $P$.
Hence $a_0 \leq x_3$ hits $y_2$, and we have $y_2\leq x_3$ in $P$.
On the other hand, by Claim~\ref{claim:2-cycle-comp} we have $a_1\leq y_2$ and $x_3\leq b_1$ in $P$.
It follows that $a_1\leq y_2\leq x_3\leq b_1$ in $P$, a contradiction. This concludes the proof.
\end{proof}

\begin{claim}\label{claim:2-cycle-struct}
Let $\Sigma \in {\bf \Sigma}(\nu_{13})$. Suppose that
$(a_1,b_1), (a_2,b_2), (a_3,b_3) \in \MM(P,\nu_{13},\Sigma)$ are three distinct pairs
such that $(a_1,b_1),(a_2,b_2)$ form a $2$-cycle of type $1$ and $(a_1,b_1), (a_3,b_3)$ do not form a $2$-cycle (so neither of type $1$ nor of type $2$).

Let $u_i := u_{a_ib_i}$, $x_i:=x_{a_ib_i}$, and $y_i:= y_{a_ib_i}$ for each $i\in \set{1,2,3}$.
Assume further that  $u_1<u_3\leq u_2$ in $T$. Then
\begin{enumerate}
 \item $a_1\leq x_3$, and
 \item $y_1\leq y_3\leq b_1$
\end{enumerate}
in $P$.
\end{claim}
\begin{proof}
Let $p_i := \p(u_i)$  and $w_i:=w_{a_ib_i}$ for each $i\in \set{1,2,3}$.
Since $u_1<u_3\leq u_2$ in $T$ and also $u_1<w_1 \leq u_2$ in $T$ by Claim~\ref{claim:w1-below-u_2}, it follows that
$$u_1<w_1 \leq u_3\leq u_2$$ in $T$.

First suppose that $u_3\leq b_1^T$ in $T$.
Then the relation $y_1\leq b_1$ hits $B(p_3)\cap B(u_3)=\set{x_3,y_3}$.
By Claim~\ref{claim:2-cycle-comp}~\ref{item:c1} we have $a_1\leq y_2$ in $P$.
Furthermore, $a_1\leq y_2$ also hits $\set{x_3,y_3}$.
Clearly,  $y_1 \leq b_1$ and $a_1\leq y_2$ cannot hit the same element of $\set{x_3,y_3}$.
If $y_1\leq x_3\leq b_1$ then  $a_1\leq y_3 \leq y_2$ in $P$, which implies
$a_3\leq x_3\leq b_1$ and $a_1\leq y_3 \leq b_3$, that is, that $(a_1,b_1), (a_3,b_3)$ is a $2$-cycle, a contradiction.
Hence we have $y_1\leq y_3\leq b_1$ and $a_1\leq x_3 \leq y_2$ in $P$, as claimed.

Next assume that $u_3\nleq b_1^T$ in $T$.
Then the path from $a_2^T$ to $b_1^T$ in $T$ includes the edge $u_3p_3$,
and hence $a_2\leq b_1$ hits $B(p_3)\cap B(u_3)=\set{x_3,y_3}$.
The relation $a_1\leq b_2$ also hits $\set{x_3,y_3}$.
Clearly, $a_2\leq b_1$ and $a_1\leq b_2$ cannot hit the same element of $\set{x_3,y_3}$.

If $a_2\leq x_3 \leq b_1$ in $P$ then $a_1 \leq y_3 \leq b_2$.
However, it then follows $a_3 \leq x_3\leq b_1$ and $a_1\leq y_3 \leq b_3$ in $P$, implying that $(a_1,b_1), (a_3,b_3)$ is a $2$-cycle,
a contradiction.

Therefore, $a_2\leq y_3 \leq b_1$ and $a_1 \leq x_3 \leq b_2$ in $P$.
In order to conclude the proof, it only remains to show that $y_1\leq y_3$ in $P$.
For this, observe that the path from $r$ to $u_3$ in $T$ includes the edge $p_1u_1$.
Hence, the relation $a_0 \leq y_3$ hits $\set{x_1,y_1}$. It cannot hit $x_1$ since otherwise $a_1\leq x_1\leq y_3\leq b_1$.
Thus, $a_0 \leq y_1\leq y_3$ in $P$, as desired.
\end{proof}

An illustration for the next two claims is given on Figure~\ref{fig:cross-edges}.

\begin{figure}[t]
 \centering
 \includegraphics[scale=0.8]{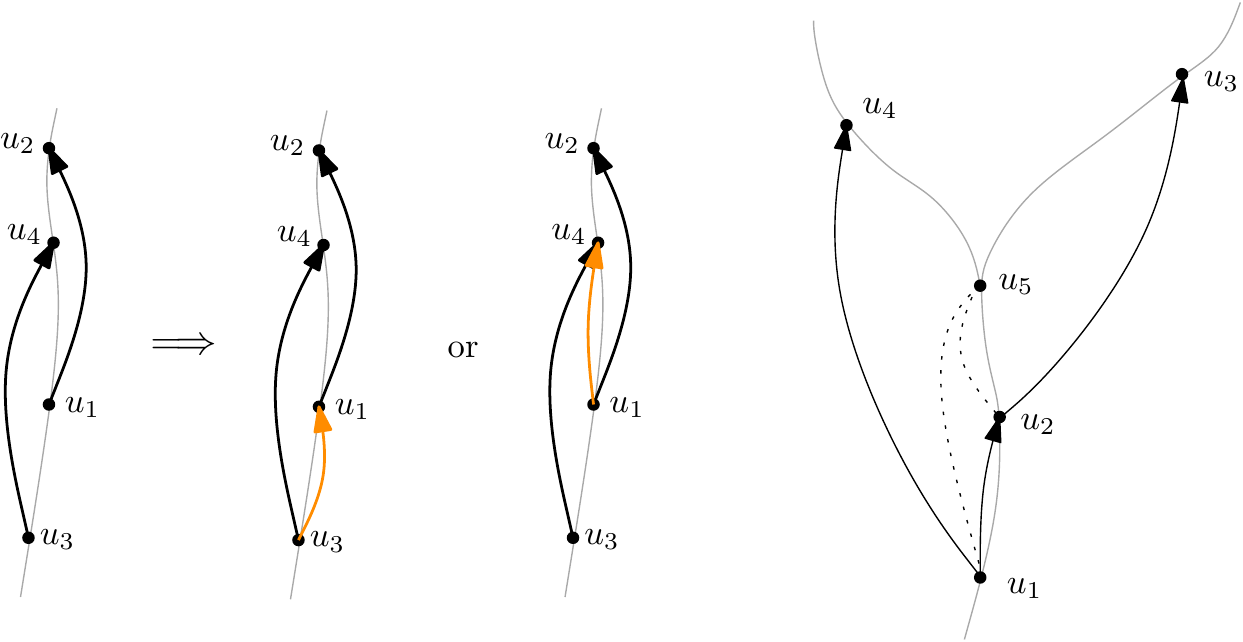}
 \caption{Possible situations in Claim~\ref{claim:implied-edge} and~\ref{claim:forbidden-object}. An edge indicates that its endpoints correspond to the meeting points of a $2$-cycle.}
 \label{fig:cross-edges}
\end{figure}

\begin{claim}\label{claim:implied-edge}
Let $\Sigma \in {\bf \Sigma}(\nu_{13})$. Suppose that
$(a_1,b_1), (a_2,b_2), (a_3,b_3), (a_4,b_4) \in \MM(P,\nu_{13},\Sigma)$ are four distinct pairs
such that $u_1<u_4<u_2$ and $u_3<u_4$ in $T$, where $u_i := u_{a_ib_i}$ for each $i\in \set{1,2,3,4}$.
Assume further that $(a_1,b_1),(a_2,b_2)$ and  $(a_3,b_3),(a_4,b_4)$ are $2$-cycles (which are thus of type 1).
Then at least one of $(a_1,b_1),(a_3,b_3)$  and $(a_1,b_1),(a_4,b_4)$ is a $2$-cycle of type 1.
\end{claim}
\begin{proof}
Let $x_i:=x_{a_ib_i}$ and $y_i:= y_{a_ib_i}$ for each $i\in \set{1,2,3,4}$.
Suppose that $(a_1,b_1),(a_4,b_4)$ is not a $2$-cycle, since otherwise we are done.
Applying Claim~\ref{claim:2-cycle-struct} on the three pairs $(a_1,b_1),(a_2,b_2),(a_4,b_4)$
we obtain that  $a_1\leq x_4$ and $y_1\leq y_4\leq b_1$ in $P$.

Using Claim~\ref{claim:2-cycle-comp} on the $2$-cycle $(a_3,b_3),(a_4,b_4)$ we also obtain
$a_3\leq y_4$ and $x_4\leq b_3$ in $P$. It follows that $a_3\leq y_4\leq b_1$ and $a_1\leq x_4\leq b_3$ in $P$.
Hence $(a_1,b_1), (a_3,b_3)$ is a $2$-cycle.
Furthermore, it is of type 1 because $u_1 < u_4$ and $u_3 < u_4$ in $T$, implying that
$u_1$ and $u_3$ are comparable in $T$.
\end{proof}

\begin{claim}\label{claim:forbidden-object}
Let $\Sigma \in {\bf \Sigma}(\nu_{13})$. Suppose that
$(a_1,b_1), (a_2,b_2), (a_3,b_3), (a_4,b_4), (a_5,b_5) \in \MM(P,\nu_{13},\Sigma)$ are five distinct pairs
such that  $u_1<u_5<u_4$ and $u_2<u_5<u_3$ in $T$, where $u_i := u_{a_ib_i}$ for each $i\in \set{1,\dots, 5}$.
Assume further that $(a_1,b_1),(a_2,b_2)$ and  $(a_2,b_2),(a_3,b_3)$ and
$(a_1,b_1),(a_4,b_4)$ are $2$-cycles of type 1.
Then at least one of $(a_1,b_1),(a_5,b_5)$  and $(a_2,b_2),(a_5,b_5)$ is a $2$-cycle of type 1.
\end{claim}
\begin{proof}
Assume to the contrary that neither $(a_1,b_1),(a_5,b_5)$  nor $(a_2,b_2),(a_5,b_5)$ is a $2$-cycle.
(Note that if one is a $2$-cycle, then it is automatically of type 1 since $u_1<u_5$ and $u_2<u_5$ in $T$.)
We either have $u_1<u_2$ or $u_2<u_1$ in $T$. Exploiting symmetry
we may assume $u_1<u_2$ in $T$. (Indeed, if not then it suffices to exchange
$(a_1,b_1)$ and $(a_4,b_4)$  with respectively $(a_2,b_2)$ and $(a_3,b_3)$.)

Applying Claim~\ref{claim:2-cycle-struct} on  the three pairs $(a_1,b_1),(a_4,b_4),(a_5,b_5)$ and
on the three pairs $(a_2,b_2),(a_3,b_3),(a_5,b_5)$,  we obtain $y_5\leq b_1$ and $y_2\leq y_5$ in $P$.
Since $(a_1,b_1),(a_2,b_2)$ is a $2$-cycle and $u_1<u_2$ in $T$, we have $a_1\leq y_2$ in $P$ by Claim~\ref{claim:2-cycle-comp}.
But all together this implies $a_1\leq y_2\leq y_5\leq b_1$ in $P$, a contradiction.
\end{proof}

\begin{figure}[t]
 \centering
 \includegraphics[scale=1.0]{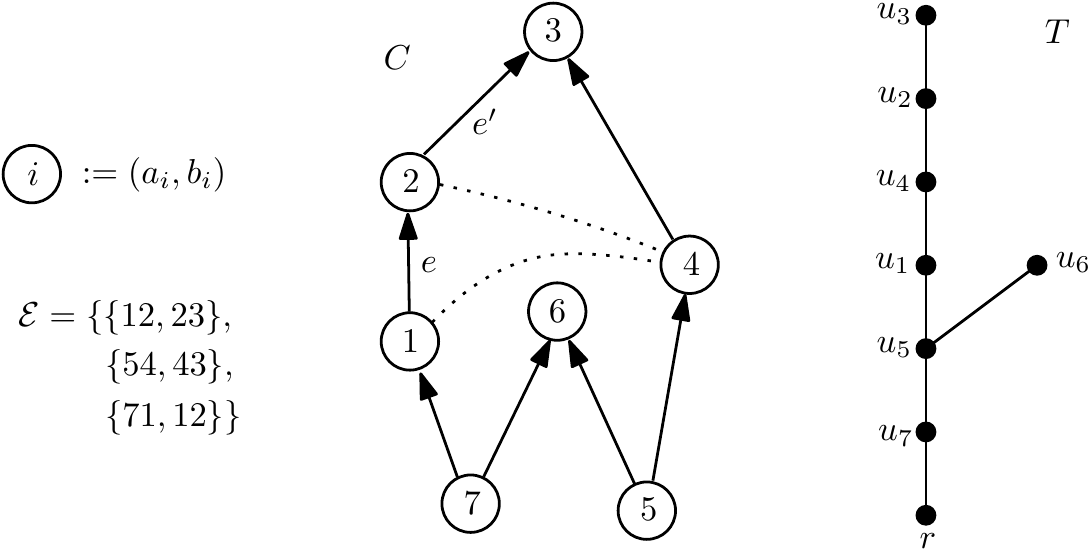}
 \caption{Example for Claim~\ref{claim:J1-is-bipartite}: Cycle $C$ on the pairs $(a_1,b_1),\ldots,(a_7,b_7)$ and the positions of $u_1,\ldots,u_7$ in $T$. Given this, it follows that $u_{\{12,23\}}=u_2$ is maximal in $T$ among $u_{\{12,23\}},u_{\{54,43\}},u_{\{71,12\}}$. Following the proof for Claim~\ref{claim:J1-is-bipartite} we would get that one of the dotted edges in the figure must be a real edge in $J_{\Sigma,1}$, implying that $C$ cannot be induced. }
 \label{fig:J1-cycle}
\end{figure}

\begin{claim}\label{claim:J1-is-bipartite}
The graph $J_{\Sigma, 1}$ is bipartite for every $\Sigma \in {\bf \Sigma}(\nu_{13})$.
\end{claim}
\begin{proof}
Let $\Sigma \in {\bf \Sigma}(\nu_{13})$.
Arguing by contradiction, suppose that $J_{\Sigma, 1}$ is not bipartite.
Let $C$ be a shortest odd cycle in $J_{\Sigma, 1}$. Thus $C$ is induced, that is, $C$ has no chord.
By Claim~\ref{claim:J1-triangles} we know that $C$ has length at least $5$.

We orient the edges of $J_{\Sigma, 1}$ in the following natural way: For each edge $\{(a,b),(a',b')\}$ in $J_{\Sigma, 1}$,
we orient the edge towards $(a',b')$ if $u_{ab}<u_{a'b'}$ in $T$, and towards $(a,b)$ otherwise (that is, if
$u_{ab}>u_{a'b'}$ in $T$).
We encourage the reader to take a look at Figure~\ref{fig:J1-cycle} for upcoming new notations and arguments.

Let $\mathcal{E}$ be the set of all pairs of consecutive edges in $C$ such that the source of one coincides with the target of the other.
Since $C$ has an odd length, we have $|\mathcal{E}|\geq 1$.
For every $\set{e,e'}\in \mathcal{E}$ consider the pair $(a,b) \in \MM(P,\nu_{13},\Sigma)$ which is the common endpoint
of $e$ and $e'$ in $J_{\Sigma, 1}$, and let $u_{\set{e,e'}}:=u_{ab}$.

Choose $\set{e,e'}\in\mathcal{E}$ such that $u_{\set{e,e'}}$ is maximal in $T$, that is,
$u_{\set{e,e'}}\not<u_{\set{f,f'}}$ in $T$ for all $\set{f,f'}\in \mathcal{E}$.
Exchanging $e$ and $e'$ if necessary we may assume that the target of $e$ coincides with the source of $e'$.
Enumerate the  vertices of the odd cycle $C$ as $(a_1,b_1), (a_2,b_2), \dots, (a_k,b_k)$ in such a way that
$e=\set{(a_1,b_1),(a_2,b_2)}$ and $e'=\set{(a_2,b_2),(a_3,b_3)}$.
Let $u_i:=u_{a_ib_i}$, $x_i:= x_{a_ib_i}$, $y_i:= y_{a_ib_i}$, and $z_i:= z_{a_ib_i}$ for each $i\in \set{1,2,\dots, k}$.
Thus $u_1<u_2<u_3$ in $T$ and $u_2=u_{\set{e,e'}}$.

Let $i$ be the largest index in $\set{3,\dots, k}$ such that $u_2<u_j$ for all $j\in\set{3,\ldots,i}$ in $T$.
If there is an index $j\in \set{3,\ldots,i}$ such that $u_{j-1}<u_j<u_{j+1}$ or $u_{j-1}>u_j>u_{j+1}$ in $T$
(taking indices cyclically), then
$u_{\set{e,e'}}=u_2<u_j$ in $T$, which contradicts our choice of $\set{e,e'}$ in $\mathcal{E}$
(since $\set{\set{(a_{j-1},b_{j-1}),(a_{j},b_{j})}, \set{(a_{j},b_{j}),(a_{j+1},b_{j+1})}}$ was a better choice).
Thus no such index $j$ exists. Given that $u_2<u_3$ in $T$, it follows that $$\set{u_{j-1}, u_{j+1}} < u_j$$ in $T$
for each {\em odd} index $j \in \set{3,\ldots,i}$, and that $i$ is odd (because $u_{i+1}<u_i$ in $T$ by the choice of $i$).

Since $u_2 < u_i$ and  $u_{i+1} < u_i$ in $T$, the two nodes $u_2$ and $u_{i+1}$ are comparable in $T$.
By our choice of $i$ we have $u_2\not< u_{i+1}$ in $T$.
(This is clear if $i < k$, and if $i=k$ this follows from the fact that $u_{k+1}=u_1 < u_2$ in $T$.)
It follows that $u_{i+1}\leq u_2$ in $T$.
We claim that $$u_{i+1}< u_2$$ in $T$.
So suppose $u_{i+1}= u_2$.
Observe that $i\neq k$ (as otherwise $u_{i+1}=u_1<u_2$ in $T$) and $i\neq k-1$ (since $i$ and $k$ are odd) in this case.
Thus, $3 \leq i \leq k-2$ and since the odd cycle $C$ is induced, it follows that
the two pairs $(a_1,b_1),(a_{i+1},b_{i+1})$ do not form a $2$-cycle. (For if they did,
it would be a $2$-cycle of type 1 since $u_1 < u_2 = u_{i+1}$ in $T$, which would
give a chord of $C$ in $J_{\Sigma, 1}$.)
Applying Claim~\ref{claim:2-cycle-struct} on the pairs $(a_1,b_1),(a_2,b_2)$ and $(a_{i+1},b_{i+1})$
we obtain $a_1\leq x_{i+1}$ in $P$.
Using Claim~\ref{claim:2-cycle-comp} on $(a_1,b_1)$ and $(a_2,b_2)$ gives us $x_2\leq b_1$ in $P$.
However, since $u_2=u_{i+1}$ and $(\phi(x_2), \phi(y_2), \phi(z_2)) =  (\phi(x_{i+1}), \phi(y_{i+1}), \phi(z_{i+1}))$
(given that $\q{sign:H-coloring-second}(a_2,b_2)=\q{sign:H-coloring-second}(a_{i+1},b_{i+1})$), we deduce $x_2=x_{i+1}$,
which implies $a_1\leq x_{i+1}=x_2\leq b_1$ in $P$, a contradiction.
Therefore, $u_{i+1}< u_2$ in $T$, as claimed.

In order to finish the proof, we consider separately the case $i < k$ and $i = k$.
First suppose that $i < k$, and thus $i \leq k-2$.
Since $u_{i+1}<u_2<u_i$ and $u_1<u_2$ in $T$, using Claim~\ref{claim:implied-edge}
on the four pairs $(a_{i+1},b_{i+1}), (a_{i},b_{i}), (a_{1},b_{1}), (a_{2},b_{2})$
we obtain that $\set{(a_{i+1},b_{i+1}),(a_1,b_1)}$ or $\set{(a_{i+1},b_{i+1}),(a_2,b_2)}$
is an edge in $J_{\Sigma, 1}$, showing that $C$ has a chord, a contradiction.

Next assume that $i=k$.
Recall that $\set{u_{j-1}, u_{j+1}} < u_j$ in $T$ for each odd index $j \in \set{3,  \dots, k}$.
It follows that $u_{j-1}$ and $u_{j+1}$ are comparable in $T$ for each such index $j$.
Using Observation~\ref{obs:comp-nodes} and $k\geq 5$ we deduce in particular that there exists an even index $\ell \in  \set{4,  \dots, k-1}$
such that $u_{\ell} \leq u_{\ell'}$ for every even index $\ell' \in  \set{4,  \dots, k-1}$.

By the choice of $\ell$ we have $u_{\ell}<u_3$ in $T$ (since $u_{\ell} \leq u_4<u_3$)
and $u_{\ell} <u_k$ in $T$ (since $u_{\ell} \leq u_{k-1}<u_k$).
Note also that $u_1 < u_2 < u_{\ell}$ in $T$, since $i=k$.
Applying Claim~\ref{claim:forbidden-object} on the five pairs
$(a_1,b_1),(a_2,b_2),(a_3,b_3),(a_k,b_k)$ and $(a_{\ell},b_{\ell})$,
we then obtain that $\set{(a_1,b_1),(a_{\ell},b_{\ell})}$ or $\set{(a_2,b_2),(a_{\ell},b_{\ell})}$
is an edge in $J_{\Sigma, 1}$, showing that $C$ has a chord, a contradiction.
\end{proof}

Now that the bipartiteness of $J_{\Sigma, 1}$ is established for each $\Sigma \in {\bf \Sigma}(\nu_{13})$,
to finish our proof of Lemma~\ref{claim:color-J}
(asserting that $J_{\Sigma}$ is $4$-colorable), it remains to show that the two subgraphs
of $J_{\Sigma, 2}$ induced by the two color classes in a $2$-coloring of $J_{\Sigma, 1}$ are bipartite.
Clearly, it is enough to show that every subgraph of $J_{\Sigma, 2}$ induced by an
independent set of $J_{\Sigma, 1}$ is bipartite, which is exactly what we will do.
(Recall that an {\em independent set}, also known as {\em stable set}, is a set of pairwise non-adjacent
vertices.)

To this aim we introduce a new definition: Given a pair $(a,b) \in \MM(P)$ and
a set $\set{c,d}$ of two elements of $P$, we say that $(a, b)$ {\em is connected to}
$\set{c,d}$ if $a\leq c$ and $d\leq b$, or $a\leq d$ and $c\leq b$ in $P$.
Note that if $(a, b)$ is connected to $\set{c,d}$, then it is in exactly one of two possible ways (that is,
either  $a\leq c$ and $d\leq b$, or $a\leq d$ and $c\leq b$ in $P$).
Thus  two pairs $(a,b), (a',b') \in \MM(P)$ connected to $\set{c,d}$ are either connected
the same way, or in opposite ways. More generally, we can consider how a collection of pairs
are connected to a certain set $\set{c,d}$, which we will need to do in what follows.

\begin{obs}\label{obs:unidirect1}
Let $\Sigma \in {\bf \Sigma}(\nu_{13})$.
Let $I$ be an independent set in $J_{\Sigma, 1}$ and let $c,d$ be two distinct elements of $P$.
Suppose that $(a,b),(a',b')\in I$ are two pairs that are connected to $\set{c,d}$ in opposite ways.
By definition $a\leq c\leq b'$ and $a'\leq d\leq b$, or $a\leq d\leq b'$ and $a'\leq c\leq b$ in $P$.
Thus in both cases $(a,b),(a',b')$ is a $2$-cycle, which must be of type $2$ as $(a,b)$ and $(a',b')$ are non-adjacent in $J_{\Sigma,1}$.
In particular, $(a,b)$ and $(a',b')$ are adjacent in $J_{\Sigma,2}$.
\end{obs}

\begin{obs}\label{obs:unidirect2}
Let $\Sigma \in {\bf \Sigma}(\nu_{13})$.
Let $I$ be an independent set in $J_{\Sigma, 1}$ and let $c,d$ be two distinct elements of $P$.
Suppose that $(a,b),(a',b') \in I$ are adjacent in $J_{\Sigma, 2}$ and that the two relations
$a\leq b'$ and $a'\leq b$ both hit $\set{c,d}$. Then $(a,b),(a',b')$ are connected to
$\set{c,d}$ in opposite ways.
\end{obs}

\begin{claim}\label{claim:4-pairs}
Let $\Sigma \in {\bf \Sigma}(\nu_{13})$.
Let $I$ be an independent set in $J_{\Sigma, 1}$ and let $c,d$ be two distinct elements of $P$.
Suppose that $C$ is an induced odd cycle in the subgraph of $J_{\Sigma, 2}$ induced by $I$.
If four of the pairs composing $C$ are connected to $\set{c,d}$ then they all are
connected to $\set{c,d}$ the same way.
\end{claim}
\begin{proof}
Suppose that $(a_1,b_1), \dots, (a_4,b_4)$ are four pairs from $C$ that are connected to $\set{c,d}$.
(Note that these pairs are not necessarily consecutive in $C$.)
If three of these pairs are connected to $\set{c,d}$ the same way and the fourth the other way,
then by Observation \ref{obs:unidirect1} the fourth pair is adjacent to the first three in $J_{\Sigma, 2}$,
which is not possible since the odd cycle $C$ is induced.

It follows that if $(a_1,b_1), \dots, (a_4,b_4)$ are not  connected to $\set{c,d}$ the same way, then
without loss of generality  $(a_1,b_1), (a_2,b_2)$ are connected  to $\set{c,d}$ one way
and $(a_3,b_3), (a_4,b_4)$ the other.
We then deduce from Observation~\ref{obs:unidirect1} that $(a_1,b_1), (a_3,b_3),(a_2,b_2), (a_4,b_4)$ is a
cycle of length $4$ in $J_{\Sigma, 2}$, a contradiction to the properties of $C$.
Therefore, all four pairs must be connected to $\set{c,d}$ the same way.
\end{proof}

\begin{figure}
 \centering
 \includegraphics[scale=1.0]{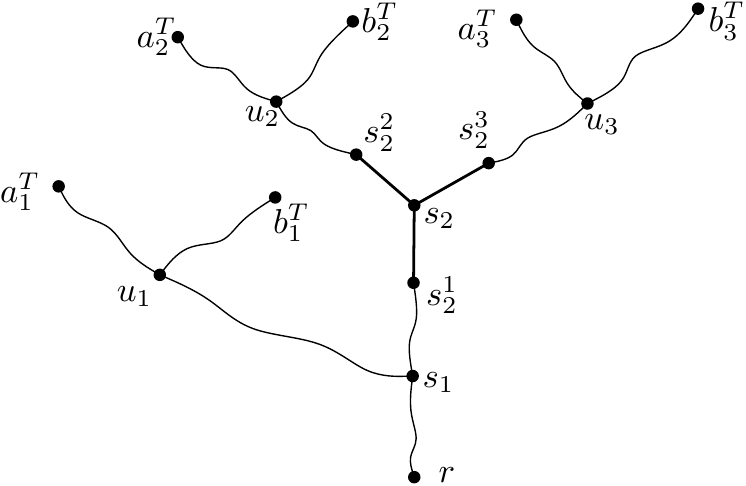}
 \caption{Illustration for Claim~\ref{claim:bipartite-J} with $j=2$}
 \label{fig:2-cycles-inc-u}
\end{figure}

\begin{claim}\label{claim:bipartite-J}
Let $\Sigma \in {\bf \Sigma}(\nu_{13})$ and let $I$ be an independent set in $J_{\Sigma, 1}$.
Then the subgraph of $J_{\Sigma, 2}$ induced by $I$ is bipartite.
\end{claim}
\begin{proof}
Arguing by contradiction, suppose there is an odd cycle in the subgraph of $J_{\Sigma, 2}$ induced by $I$, and
let $C$ be a shortest one.
Enumerate the  vertices of $C$ as $(a_1,b_1), (a_2,b_2), \dots, (a_k,b_k)$ in order.
Let $u_i := u_{a_ib_i}$  and
$s_i :=u_i\wedge u_{i+1}$ for each $i \in \set{1, \dots, k}$ (cyclically).
Recall that by the definition of $J_{\Sigma,2}$ the pairs $(a_i,b_i)$ and $(a_{i+1},b_{i+1})$ form a $2$-cycle of type $2$ for each $i$, and thus $u_i\inc u_{i+1}$ in $T$, implying that $s_i < \set{u_i, u_{i+1}}$ in $T$.

Let us start by pointing out the following consequence of Observation~\ref{obs:unidirect2}:
If $i \in \set{1, \dots, k}$  and $\set{c,d}$ are
such that the $u_i$--$u_{i+1}$ path in $T$ includes an edge $e$ of $T$ for which the
intersection of the two bags of its endpoints is $\set{c, d}$, then
$(a_i, b_i)$ and $(a_{i+1}, b_{i+1})$ are connected to $\set{c,d}$ in opposite ways.
This will be used a number of times in the proof.

Let $j \in \set{1, \dots, k}$ be such that $s_j$ is maximal in $T$ among $s_1, \dots, s_k$, that is,
such that $s_j\not< s_i$ in $T$ for  each $i \in \set{1, \dots, k}$.
Furthermore, if $s_j\neq u_i$ then we let $s_j^i$ be the neighbor of $s_j$
on the $s_j$--$u_i$ path in $T$, for each $i \in \set{1, \dots, k}$.
(Let us remark that we only make use of the notion $s_j^i$ when $s_j\neq u_i$).
Thus in particular $s_j < \set{s_j^{j}, s_j^{j+1}}$ in $T$.
See Figure~\ref{fig:2-cycles-inc-u} for an illustration of this definition.

We claim that
$$B(s_j) \cap B(s_j^s) \neq B(s_j) \cap B(s_j^{t})$$
for any $s, t \in \set{1, \dots, k}$ with $s<t$ such that  $s_j < \set{s_j^{s}, s_j^{t}}$ in $T$.
Suppose to the contrary that $B(s_j) \cap B(s_j^s) = B(s_j) \cap B(s_j^{t})=:\set{c,d}$.
It follows from our choice of index $j$ that both the $u_{s-1}$--$u_s$ path
and the $u_{s}$--$u_{s+1}$ path in $T$ include the edge
$s_js_j^s$ (otherwise $s_{j} < s_{s}$ in $T$),  and similarly that
the $u_{t}$--$u_{t+1}$ path in $T$ includes the edge $s_js_j^{t}$ (otherwise $s_{j} < s_{t}$ in $T$).
As a consequence we have that edge $s_js_j^s$ is included in the $a_{s-1}^T$--$b_s^T$ path (as it passes through $u_{s-1}$ and $u_s$), the $a_s^T$--$b_{s-1}^T$ path, the $a_s^T$--$b_{s+1}^T$ path and the $a_{s+1}^T$--$b_s^T$ path, and similarly that edge $s_js_j^t$ is included in the $a_t^T$--$b_{t+1}^T$ path and the $a_{t+1}^T$--$b_t^T$ path.
Therefore, the pairs $(a_{s-1},b_{s-1})$, $(a_{s},b_{s})$, $(a_{s+1},b_{s+1})$,
$(a_{t},b_{t})$, and $(a_{t+1},b_{t+1})$ are all connected
to $\set{c,d}$.  Furthermore, $(a_{s-1},b_{s-1})$ and  $(a_{s},b_{s})$
are connected in opposite ways, and the same holds for
$(a_{s},b_{s})$ and $(a_{s+1},b_{s+1})$, as well as for $(a_{t},b_{t})$ and $(a_{t+1},b_{t+1})$.
There cannot be four distinct pairs among these five, because otherwise this would contradict Claim~\ref{claim:4-pairs}.
Hence the only possibility is that $k=3$ and $s=t-1$, and thus $s-1$ and $t+1$ are the same indices cyclically.
But then recall that the $u_1$--$u_2$ path, the $u_2$--$u_3$ path and the $u_3$--$u_1$ path  all have to use edge $s_js_j^s$ or $s_js_j^t$ in $T$.
As a consequence, the three relations $a_{3}\leq b_1$, $a_1\leq b_{2}$ and $a_{2}\leq b_{3}$
all hit $\set{c,d}$. Hence two of them hit the same element, which implies $a_i\leq b_i$ for some $i\in \set{1,2,3}$.
With this contradiction, we have proved $B(s_j) \cap B(s_j^s) \neq B(s_j) \cap B(s_j^{t})$.

Since $B(s_j)\cap B(s_j^i)$ is a $2$-element subset of $B(s_j)$ for each $i$ such that $s_j\neq u_i$, it directly follows
that there are at most three indices $i$ such that $s_j < s_j^{i}$ in $T$ (recall that this inequality holds for $i=j$ and $i=j+1$).

This allows us to quickly dispense with the $k \geq 5$ case now:
If $u_{j-1}>s_j$ in $T$ (and hence $s_j^{j-1}>s_j$ in $T$), then let $(s,t):=(j-1,j+1)$.
Else, if $u_{j+2}>s_j$ in $T$, then let $(s,t):=(j,j+2)$.
If neither of the two cases is true, we let $(s,t):=(j,j+1)$.
We claim that in all three cases we obtain that
\begin{itemize}
 \item $s_j < \set{s_j^{s}, s_j^{t}}$ in $T$,
 \item the two indices $s-1$ and $t+1$ are not the same (cyclically), and
 \item $s_j^{s-1} = s_j^{t+1} = \p(s_j)$.
\end{itemize}
The first two items are obvious, and the third one follows from our last observation and because $s_j\neq u_{s-1}$ and $s_j\neq u_{t+1}$ (recall that $u_{s-1}\inc u_s$ and $u_t\inc u_{t+1}$ in $T$).

Then, the $u_{s-1}$--$u_s$ path and the $u_{t}$--$u_{t+1}$ path in $T$ both include the edge $\p(s_j)s_j$.
It follows that  $(a_{s-1},b_{s-1})$ and  $(a_{s},b_{s})$ are connected to $B(s_j) \cap B(\p(s_j))$
in opposite ways, and that the same holds for $(a_{t},b_{t})$ and $(a_{t+1},b_{t+1})$. Since these four pairs
are distinct, this contradicts Claim~\ref{claim:4-pairs} and concludes the case $k\geq 5$.

It remains to consider the $k=3$ case. Reordering the pairs of $C$ if necessary we may assume $j=1$ and
$B(s_1) = \set{c,d,e}$ with $B(s_1) \cap B(s_1^1)=\set{c,d}$ and $B(s_1) \cap B(s_1^2)=\set{d,e}$.
The two relations $a_1\leq b_{2}$ and $a_{2}\leq b_1$ both hit $\set{c,d}$ and $\set{d,e}$, and clearly
they cannot hit the same element. Thus, one of the relations hits $d$, and the other $c$ and $e$.
Exploiting symmetry again, we may assume without loss of generality that $a_{2}\leq b_1$ hits $d$.
(Indeed, if not then this can be achieved by reversing the ordering of the pairs of $C$.)
Thus we have $$a_{2}\leq d\leq b_1$$ in $P$, which then implies $$a_1\leq c\leq e\leq b_{2}$$
in $P$ by Observation~\ref{obs:hitting-vertex-or-edge}.
Now, the two relations
$a_{2}\leq b_{3}$ and $a_{3}\leq b_1$ both hit $B(s_1) = \set{c,d,e}$.
(Here we use that $s_{1} \not< \set{s_2,s_3}$ in $T$.)
Neither hit $c$ or $e$ since this would contradict $a_1\leq c\leq e\leq b_{2}$ in $P$.
Hence both relations hit $d$, which implies $a_{3}\leq d \leq b_{3}$ in $P$,  a contradiction.
\end{proof}

This concludes the proof of Lemma~\ref{claim:color-J}, asserting that $J_{\Sigma}$ is $4$-colorable
for  each $\Sigma \in {\bf \Sigma}(\nu_{13})$. Now, for each $\Sigma \in {\bf \Sigma}(\nu_{13})$ let
$\psi_{13, \Sigma}$ be such a coloring. Then we define $\q{sign:coloring-2-cycles}$ as follows:
\begin{center}
\medskip
\fbox{\begin{varwidth}{0.9\textwidth}
For each $\Sigma\in\mathbf{\Sigma}(\nu_{13})$ and pair $(a,b)\in\MM(P,\nu_{13},\Sigma)$, we let
$$\q{sign:coloring-2-cycles}(a,b) :=  \psi_{13, \Sigma}(a,b).$$
\end{varwidth}}
\medskip
\end{center}

\subsection{Node \texorpdfstring{$\nu_{14}$ and its function $\alpha_{14}$: Dealing with strict alternating cycles of length at least $3$}.}
\label{sec:node-14}

Recall that compared to $\MM(P,\nu_{14},\Sigma)$, pairs in $\MM(P, \nu_{14}, \Sigma)$ have the additional property that they do not form a $2$-cycle with another pair of $\MM(P, \nu_{14}, \Sigma)$, thanks to function $\q{sign:coloring-2-cycles}$.
Therefore, strict alternating cycles in $\MM(P, \nu_{14}, \Sigma)$ ($\Sigma \in {\bf \Sigma}(\nu_{14})$) have length at least $3$.
We will now list (and prove) a number of properties satisfied by these alternating cycles.

First we prove a claim that bears some similarity with Claim~\ref{claim:u-ordering}.

\begin{claim}
\label{claim:u_minimal}
Let  $\Sigma \in {\bf \Sigma}(\nu_{14})$ and
suppose that $\set{(a_i,b_i)}_{i=1}^k$ is a strict alternating cycle in $\MM(P, \nu_{14}, \Sigma)$.
Let $u_i$ denote $u_{a_ib_i}$ for each $i \in \{1, 2, \dots, k\}$.
Then there is an index $j \in \{1,2, \dots, k\}$ such that $u_{j} \leq u_{i}$ in $T$ for each $i \in \{1,2, \dots, k\}$.
\end{claim}
\begin{proof}
We denote $w_{a_ib_i},p_{a_ib_i}, x_{a_ib_i}, y_{a_ib_i}, z_{a_ib_i}$ by $w_i, p_i, x_i, y_i, z_i$ respectively, for each $i\in \{1,2, \dots, k\}$.
We may assume $\q{sign:left-or-right}(a_i, b_i)=\textrm{left}$.

Consider the nodes $u_1,\ldots,u_k$
and let $j \in \{1,2, \dots, k\}$ be such that $u_j$ is minimal in $T$ among these.
We will show that $u_{j} \leq u_{i}$ in $T$ for each $i \in \{1,2, \dots, k\}$.
This can equivalently be rephrased as follows:  Every element $u_i$ which is minimal in $T$ among $u_1,\ldots,u_k$ satisfies
$u_i=u_j$ (note that we could possibly have $u_i = u_j$ for $i \neq j$).
Arguing by contradiction, let us assume that there is an element minimal in $T$ among $u_1,\ldots,u_k$ which is distinct from $u_j$.

We start by showing that under this assumption $u_1,\ldots,u_k$ are all pairwise incomparable in $T$ (and thus are in particular all distinct).
Once this is established, we will then be able to derive the desired contradiction.

Of course, to prove that $u_1,\ldots,u_k$ are pairwise incomparable in $T$ it is enough to show that $u_i \inc u_j$ in $T$ for each $i \in \{1,2, \dots, k\}$ with $i\neq j$, since
$u_j$ was chosen as an arbitrary minimal element in $T$ among $u_1,\ldots,u_k$.
Assume not, that is, that there is an index $i \in \{1,2, \dots, k\}$ with $i\neq j$ such that $u_j\leq u_{i}$ in $T$.
We may choose $i$ in such a way that we additionally have $u_{i-1}\inc u_j$ or $u_{i+1}\inc u_j$ in $T$.
As the arguments for the two cases are analogous we consider only the case $u_{i-1}\inc u_j$ in $T$.

We have $a_{i-1}^T \not \geq u_j$ since $u_{i-1}\inc u_j$ in $T$, and we also
have $u_j\leq u_{i} < b_{i}^T$ in $T$. It follows that the path from ${a_{i-1}^T}$ to $b_{i}^T$ in $T$ goes through the edge $p_ju_j$.
Thus, the relation $a_{i-1}\leq b_{i}$ hits $B(p_j)\cap B(u_j)=\set{x_j,y_j}$.
But then $a_{i-1} \leq x_j \leq b_{i}$ or $a_{i-1} \leq y_j \leq b_{i}$ in $P$, which implies $a_{j} \leq x_j \leq b_{i}$ or $a_{i-1} \leq y_j \leq b_j$.
Since we have $i\neq j+1$ (as $u_{i-1}\inc u_j$ in $T$) and $j\neq i$,
this contradicts the assumption that our alternating cycle $(a_1,b_1), (a_2,b_2), \dots, (a_k,b_k)$ is strict.

We conclude that $u_1,\ldots,u_k$ are all pairwise incomparable in $T$, as claimed.

Let $s_i :=u_i\wedge u_{i+1}$ for each $i \in \{1,2, \dots, k\}$ (indices are taken cyclically, as always).
Note that the path from $a_i^T$ to $b_{i +1}^T$ in $T$ has to go through $s_i$.
Choose $i\in\set{1,\ldots,k}$ such that $s_{i}$ is maximal among $s_1,\ldots,s_k$ in $T$.
The nodes $s_{i-1}$ and $s_{i}$ are comparable in $T$, since $s_{i-1}\leq u_{i}$ and $s_{i}\leq u_{i}$ in $T$.
Thus we have $s_{i-1}\leq s_{i}$ in $T$. Similarly, $s_{i+1}\leq s_{i}$ in $T$.

Let us first look at the case $s_{i-1}=s_{i}$.
This implies $s_{i} \leq \set{u_{i-1},u_{i},u_{i+1}}$ in $T$.
Now the $a_{i-1}^T$--$b_{i}^T$ path, the $a_{i}^T$--$b_{i+1}^T$ path, and the $a_{i+1}^T$--$b_{i+2}^T$ path in $T$
all go through $s_{i}$ in $T$.
This means that the relations $a_{i-1}\leq b_{i}$, $a_{i}\leq b_{i+1}$ and $a_{i+1}\leq b_{i+2}$ all hit $B(s_{i})$.
Clearly, no two of them can hit the same element (recall that $k\geq3$ and that our alternating cycle is strict),
and hence each element of $B(s_{i})$ is hit by exactly one of these three relations.
On the other hand, the three paths from $r$ to $b_{i-1}^T,b_{i}^T$ and $b_{i+1}^T$ in $T$ all go through $\p(s_{i})$ and $s_{i}$,
implying that the relations $a_0 \leq b_{i-1}$, $a_0 \leq b_{i}$, and $a_0 \leq b_{i+1}$ all hit $B(\p(s_{i})) \cap B(s_{i})$.
In particular, some element in $B(s_{i})$ is hit by {\em at least two} of these three relations.
But with the observations made before, it follows that some element in $\set{a_{i-1},a_{i},a_{i+1}}$ is below two elements of $\set{b_{i-1},b_{i},b_{i+1}}$ in $P$, which is not possible in a strict alternating cycle. Therefore, $s_{i-1} \neq s_{i}$.

Thus we have $s_{i-1}<s_{i}$ in $T$, and with a similar argument one also deduces that $s_{i+1}<s_{i}$ in $T$.

To conclude the proof, consider the $a_{i-1}^T$--$b_{i}^T$ path, the $a_{i+1}^T$--$b_{i+2}^T$ path, and the $r$--$b_{i+1}^T$ path in $T$.
They all go through the edge $\p(s_{i})s_{i}$ of $T$, and hence the corresponding relations in $P$ all hit $B(\p(s_{i}))\cap B(s_{i})$.
Therefore, two of these relations hit the same element in that set, which again contradicts the fact that our alternating cycle is strict.
\end{proof}

By Claim~\ref{claim:u_minimal} we are in a situation similar to that
first encountered in Section~\ref{sec:second_leaf}, namely
for each $\Sigma \in {\bf \Sigma}(\nu_{14})$
each alternating cycle in $\MM(P, \nu_{14}, \Sigma)$ can be written
as $\set{(a_i,b_i)}_{i=1}^k$ in such a way that $u_{a_1b_1} \leq u_{a_ib_i}$ in $T$ for each $i \in \{1, \dots, k\}$.
We may further assume that the pair $(a_1,b_1)$ is such that $b_1^T$ is to the right of $b_i^T$ in $T$
if $a_1^T$ is to the left of $b_1^T$ in $T$, and to the left of $b_i^T$ otherwise,
for each $i \in \{2, \dots, k\}$ such that $u_{a_1b_1} = u_{a_ib_i}$.
As before the pair $(a_1,b_1)$ is uniquely defined, and we call it the \emph{root} of the alternating cycle.

Our next claim mirrors Claim~\ref{claim:u2-in-component} from Section~\ref{sec:second_leaf}.

\begin{claim}\label{claim:orange-tree}
Let  $\Sigma \in {\bf \Sigma}(\nu_{14})$ and suppose that $\set{(a_i,b_i)}_{i=1}^k$
is a strict alternating cycle in $\MM(P, \nu_{14}, \Sigma)$ with root $(a_1, b_1)$.
Let $u_i, w_i$ denote $u_{a_ib_i}, w_{a_ib_i}$ respectively, for each $i \in \{1, 2, \dots, k\}$.
Then $u_{1} < w_{1} \leq u_{i}$ in $T$ for each $i \in \{2, \dots, k\}$.
\end{claim}
\begin{proof}
We denote $p_{a_ib_i}, x_{a_ib_i}, y_{a_ib_i}, z_{a_ib_i}$ by $p_i, x_i, y_i, z_i$ respectively, for each $i\in \{1,2, \dots, k\}$.
We may assume $\q{sign:left-or-right}(a_i, b_i)=\textrm{left}$ for each $i\in \{1,2, \dots, k\}$.

First we will show that $u_1< w_1 \leq u_k$ in $T$.
To do so suppose first that $u_1=u_k$.
Then $w_1\not\leq a_k^T$ in $T$, as otherwise we would have $b_k^T$ to the right of $b_1^T$ in $T$ (Observation~\ref{obs:left-right}),
which contradicts the choice of $(a_1, b_1)$ as the root of the strict alternating cycle.
In particular, the path from $a_k^T$ to $b_1^T$ in $T$ goes through $u_1$.
Hence the relation $a_k\leq b_1$ hits $B(u_1)=\set{x_1,y_1,z_1}$; let $q\in B(u_1)$ be such that $a_k\leq q \leq b_1$ in $P$.
Clearly, $q\in \set{y_1, z_1}$.
Given that $u_1=u_k$ and $(\phi(x_1), \phi(y_1), \phi(z_1))= (\phi(x_k), \phi(y_k), \phi(z_k))$
(since $\q{sign:H-coloring-second}(a_1,b_1)=\q{sign:H-coloring-second}(a_k,b_k)$),
we obviously have $x_1=x_k$, $y_1=y_k$, and $z_1 = z_k$.
If $q = y_1 =y_k$ then we directly obtain $q \leq b_k$ in $P$.
If $q = z_1 =z_k$ then we also deduce $q \leq b_k$ in $P$,
because $\q{sign:comparability-status}(a_1, b_1) = \q{sign:comparability-status}(a_k, b_k)$,
and thus in particular $z_k \leq b_k$ in $P$
since $z_1 \leq b_1$.  Hence in both cases $q \leq b_k$ in $P$.
This implies $a_k\leq q\leq b_k$ in $P$, a contradiction.
Therefore, $u_1 \neq u_k$, and $u_1<u_k$ in $T$.

Let $w'$ be the neighbor of $u_1$ on the $u_1$--$u_k$ path in $T$.
In order to show $u_1< w_1 \leq u_k$ in $T$, it remains to prove $w'=w_1$.
Suppose to the contrary that $w'\neq w_1$.
Then the $a_k^T$--$b_1^T$ path and the $r$--$b_k^T$ path in $T$ both go through $u_1$ and $w'$.
Hence the relations
$a_k\leq b_1$ and $a_0 \leq b_k$ both hit $B(u_1)\cap B(w')\subsetneq \set{x_1,y_1,z_1}$. Clearly,
they cannot hit the same element.
None of the two relations hit $x_1$, as otherwise $a_1\leq x_1\leq b_1$ or $a_1\leq x_1\leq b_k$ in $P$
(which is not possible since $k\geq3$ and the alternating cycle is strict).
We conclude that $B(u_1)\cap B(w')=\set{y_1,z_1}$.
Since the relation $a_0 \leq b_k$ also hits $B(u_1)\cap B(p_1)=\set{x_1,y_1}$, and thus hits $y_1$,
it follows that $a_0 \leq y_1\leq b_k$ and $a_k\leq z_1\leq b_1$ in $P$.
Now let $i \in \{1, \dots, k-1\}$ be maximal such that $w'\not\leq u_{i}$ in $T$.
Note that there is such an index since $w'\not\leq u_1$ in $T$.
If $w'\leq a_{i}^T$ in $T$, then the path from $a_{i}^T$ to $u_{i}$ in $T$
goes through the edge $u_1w'$.
If, on the other hand, $w'\inc a_{i}^T$ in $T$, then the path from $a_{i}^T$ to $b_{i+1}^T$ in $T$ goes through the edge $u_1w'$.
Thus at least one of $a_{i}\leq x_{i}$ and $a_{i}\leq b_{i+1}$ hits $B(u_1)\cap B(w')=\set{y_1,z_1}$.
Hence $a_{i}\leq y_1$ or $a_{i}\leq z_1$ in $P$. However, since  $\set{y_1,z_1}\leq b_1$ in $P$,
this implies $a_{i}\leq b_1$ in both cases. Given that $i < k$, this contradicts the fact that
the alternating cycle is strict. Therefore, we must have $w'=w_{1}$, and
$u_1<w_1\leq u_k$ in $T$, as claimed.

So far we know that $w_1\leq u_i$ in $T$ for $i=k$, and it remains to show
it for each $i\in \set{2, \dots, k-1}$. Arguing by contradiction,
assume that this does not hold, and let $i\in \set{2, \dots, k-1}$ be maximal such that $w_1\not\leq u_{i}$ in $T$. By our choice it holds that $w_1\leq u_{i+1}$ in $T$, even in the case $i=k-1$.

First suppose that $u_i=u_{1}$. Then $w_1\not\leq a_{i}^T$ in $T$,
because otherwise $b_{i}^T$ would be to the right of $b_1^T$ in $T$ (Observation~\ref{obs:left-right}),
contradicting the fact that $(a_{1}, b_{1})$ is the root of the strict alternating cycle.
But then, the path from $a_{i}^T$ to $b_{i+1}^T$ goes through the edge $u_1w_1$ (since $w_1\leq u_{i+1}$ in $T$).
Thus the relation $a_{i}\leq b_{i+1}$ hits in particular $B(u_1)$; let
$q\in B(u_1)$ be such that $a_{i}\leq q \leq b_{i+1}$ in $P$.
Given that $u_{1}=u_{i}$ we deduce
$x_1=x_i$, $y_1=y_i$, and $z_1 = z_i$ (using $\q{sign:H-coloring-second}(a_1,b_1)=\q{sign:H-coloring-second}(a_i,b_i)$),
and hence $a_1\leq  q$ in $P$ (using $\q{sign:comparability-status}(a_1, b_1) = \q{sign:comparability-status}(a_i, b_i)$),
exactly as in the beginning of the proof.
This implies $a_1\leq q\leq b_{i+1}$ in $P$, and as $i+1\geq3$ this contradicts once again the fact
that the alternating cycle is strict.
Therefore, $u_{1}\neq u_i$, and $u_{1} < u_i$ in $T$.

Let $w'$ be the neighbor of $u_{1}$ on the $u_1$--$u_{i}$ path in $T$. Note that $w'\neq w_1$.
The $a_{i}^T$--$b_{i+1}^T$ path and the $r$--$b_{i}^T$ path both
go through the edge $u_1w'$. Thus the relations $a_{i}\leq b_{i+1}$ and $a_0 \leq b_{i}$ both hit
$B(u_1)\cap B(w')\subsetneq \set{x_1,y_1,z_1}$. Clearly, they cannot hit the same element.
Since $a_{i} \leq x_1 \leq b_{i+1}$ in $P$ would imply $a_1 \leq b_{i+1}$ (which is not possible since $i+1\geq3$)
while $a_{i} \leq y_1 \leq b_{i+1}$ would imply $a_{i}\leq b_1$  (which cannot be since $i < k$),
we deduce
$$
a_{i}\leq z_1 \leq b_{i+1}
$$
in $P$, and
$$
a_0 \leq  q \leq b_{i}
$$
in $P$, where $q$ is the element in $\set{x_1,y_1}$ such that $B(u_1)\cap B(w')=\set{q, z_1}$.

We distinguish two cases, depending whether $q=x_{1}$ or $q=y_{1}$.
First suppose that $q=x_{1}$.
Since $a_0 \leq  q=x_{1} \leq b_{i}$ in $P$, this implies $a_{1} \leq x_{1} \leq b_{i}$ in $P$,
and hence $i=2$ (otherwise, the alternating cycle would not be strict).
Furthermore, given that $a_0 \leq  x_{1}$ in $P$ and
$\q{sign:comparability-status}(a_1, b_1) = \q{sign:comparability-status}(a_2, b_2)$,
we have $a_0 \leq  x_{2}$ in $P$ as well.
The $r$--$u_{2}$ path in $T$ includes the edge $u_{1}w'$ since $w' \leq u_{2}$ in $T$.
Using that $x_{2} \in B(u_{2})$, we deduce that the relation $a_0 \leq  x_{2}$
in $P$ hits $B(u_1)\cap B(w')=\set{x_{1}, z_1}$. In particular,
at least one of $x_{1} \leq x_{2}$ and $z_{1} \leq x_{2}$ holds in $P$.
Before considering each of these two possibilities,
let us observe that
the $a_{2}^T$--$u_1$ path in $T$ includes the edge $u_{2}p_{2}$. It follows that
the relation $a_{2}\leq z_1$ hits $B(u_{2})\cap B(p_{2})=\set{x_{2},y_{2}}$.
Clearly, it cannot hit $y_{2}$ (otherwise $a_{2}\leq y_{2}\leq b_{2}$), and
hence $x_{2}\leq z_1$ in $P$.

Now, if $z_{1} \leq x_{2}$ in $P$ then $x_{2}=z_1$.
However, we also know that $\phi(x_{2}) = \phi(x_{1}) \neq \phi(z_{1})$, since
$\q{sign:H-coloring-second}(a_1,b_1)=\q{sign:H-coloring-second}(a_2,b_2)$, which is a contradiction.

On the other hand, if $x_{1} \leq x_{2}$ in $P$ then
$a_1\leq x_1\leq x_{2}\leq z_1\leq b_{i+1} = b_{3}$ in $P$, which contradicts the fact the alternating cycle
is strict.
This concludes the case where $q=x_{1}$.

Next, assume $q=y_{1}$.
Let $j\in \{1, \dots, i-1\}$ be maximal such that $w'\not\leq u_j$ in $T$.
(Note that there is such an index $j$ since $w'\not\leq u_1$ in $T$.)
If $w'\leq a_j^T$ in $T$ then the path from $a_j^T$ to $u_j$ in $T$ goes through the edge $u_1w'$.
If, on the other hand, $w'\inc a_j^T$ in $T$,
then the path from $a_j^T$ to $b_{j+1}^T$ in $T$ goes through the edge $u_1w'$
since $w'\leq u_{j+1}< b_{j+1}^T$ in $T$.
Hence at least one of the two relations $a_j\leq x_j$ and $a_j\leq b_{j+1}$ hits $\set{q,z_1}=\set{y_{1}, z_{1}}$.
It follows that $a_j\leq y_1$ or $a_j\leq z_{1}$ in $P$.
The first inequality implies $a_j\leq y_{1}\leq b_{1}$ in $P$, a contradiction since $j \neq k$.
The second inequality implies $a_j\leq z_1 \leq b_{i+1}$ in $P$, which is not possible since $j \neq i$.
This concludes the proof.
\end{proof}

\begin{claim}
\label{claim:q-point}
Let  $\Sigma \in {\bf \Sigma}(\nu_{14})$ and suppose that $\set{(a_i,b_i)}_{i=1}^k$ is a strict alternating cycle in $\MM(P, \nu_{14}, \Sigma)$ with root $(a_1, b_1)$.  Let $u_{i}$ denote
$u_{a_ib_i}$ for each $i\in \{1,2, \dots, k\}$.
Then the $u_{1}$--$b_1^T$ path in $T$ avoids $u_{2}$.
\end{claim}
\begin{proof}
We denote $w_{a_ib_i},p_{a_ib_i}, x_{a_ib_i}, y_{a_ib_i},$ by $w_i, p_i, x_i, y_i$ respectively, for each $i\in \{1,2, \dots, k\}$.
We may assume $\q{sign:left-or-right}(a_i, b_i)=\textrm{left}$.

Arguing by contradiction, suppose that $u_1\leq u_2 < b_1^T$ in $T$.
By Claim~\ref{claim:orange-tree} we know $u_1 < w_1 \leq u_2 < b_2^T$ in $T$.
The  $a_1^T$--$b_2^T$ path in $T$ goes through the edge $p_2u_2$.
Hence the relation $a_1\leq b_2$ hits $B(p_2)\cap B(u_2)=\{x_2,y_2\}$.
Clearly, it cannot hit $x_2$ because otherwise $a_2\leq x_2\leq b_2$ in $P$.
Therefore, $a_1\leq y_2\leq b_2$ in $P$.

Now consider the path connecting $r$ to $b_1^T$ in $T$. This path also includes the edge $p_2u_2$.
Thus the relation $a_0 \leq b_1$ hits $\set{x_2,y_2}$.
If it hits $x_2$, then we obtain $a_2\leq x_2\leq b_1$ in $P$, which contradicts the fact that the
alternating cycle is strict (recall that $k \geq3$).
If it hits $y_2$, then we deduce $a_1\leq y_2\leq b_1$ in $P$, again a contradiction.
\end{proof}

Let  $\Sigma \in {\bf \Sigma}(\nu_{14})$ and suppose that $\set{(a_i,b_i)}_{i=1}^k$ is a strict alternating cycle
in $\MM(P, \nu_{14}, \Sigma)$ with root $(a_1, b_1)$.
In what follows we will need to consider the nodes $q_i:= u_i\wedge b_1^T$  of $T$ where $i\in \{1,2, \dots, k\}$.
Observe that $$u_1<w_1\leq q_i \leq u_i$$ in $T$ for each $i\in \{2,3, \dots, k\}$ by Claim~\ref{claim:orange-tree}, and
$$q_2<u_2$$ in $T$ by Claim~\ref{claim:q-point}.

\begin{claim}\label{claim:q2-q3}
Let  $\Sigma \in {\bf \Sigma}(\nu_{14})$ and suppose that $\set{(a_i,b_i)}_{i=1}^k$ is a strict alternating cycle
in $\MM(P, \nu_{14}, \Sigma)$ with root $(a_1, b_1)$.
Let $u_i$ denote $u_{a_ib_i}$ and let $q_i:= u_i\wedge b_1^T$, for each $i\in \{1,2, \dots, k\}$.
Then
\begin{enumerate}
 \item\label{item:uj-not-below-v2} $u_i \not\leq q_2$ in $T$ for each $i\in \{3,4, \dots, k\}$, and
 \item\label{item:v2-below-v3} $u_1<q_2\leq q_3<b_1^T$ in $T$.
\end{enumerate}
\end{claim}
\begin{proof}
We denote $w_{a_ib_i},p_{a_ib_i}, x_{a_ib_i}, y_{a_ib_i},$ by $w_i, p_i, x_i, y_i$ respectively, for each $i\in \{1,2, \dots, k\}$.
We may assume $\q{sign:left-or-right}(a_i, b_i)=\textrm{left}$.

To prove~\ref{item:uj-not-below-v2} we argue by contradiction: Suppose  $u_i\leq q_2$ in $T$ for some $i\in \{3,4, \dots, k\}$.
Since $q_2 < b_1^T$ in $T$, and  $u_1<w_1 \leq u_i$ by
Claim \ref{claim:orange-tree}, it follows that $u_1 <w_1 \leq u_i \leq q_2 < b_1^T$ in $T$.
In particular, the path connecting $a_1^T$ to $b_2^T$ in $T$ goes through the edge $p_iu_i$.
Hence the relation $a_1\leq b_2$ hits $B(p_i)\cap B(u_i)=\set{x_i,y_i}$.
If it hits $x_i$ then $a_i\leq x_i\leq b_2$ in $P$, while if it
hits $y_i$ then $a_1\leq y_i\leq b_i$ in $P$. In both cases it
contradicts the fact that the alternating cycle is strict.

Let us now prove~\ref{item:v2-below-v3}.
Using Claim~\ref{claim:orange-tree} we already deduce that
$u_1 < \set{q_2, q_3} < b_1^T$ in $T$.
Thus, it remains to show $q_2 \leq q_3$ in $T$.
Arguing by contradiction, suppose $q_3 < q_2$ in $T$ (note that $q_2$ and $q_3$ are comparable in $T$).
Let $i$ be the largest index in $\{3,4, \dots, k\}$ such that $q_i<q_2$ in $T$.
If $i < k$ then $q_i<q_2\leq q_{i+1}\leq u_{i+1}< b_{i+1}^T$ in $T$.
If $i=k$ then clearly $q_i <q_2 < b_{1}^T$ in $T$.
Thus in both cases
\begin{equation}
 q_i <q_2 < b_{i+1}^T\label{eq:qi-q2}
\end{equation}
 in $T$ (taking indices cyclically).

Observe also that
\begin{equation}
 q_2 \not\leq a_i^T \label{eq:q2-ai}
\end{equation}
in $T$.
Indeed, if $q_2 \leq a_i^T$ in $T$ then
$q_2 \leq b_i^T$ as well, since otherwise $u_i < q_2$ in $T$, contradicting~\ref{item:uj-not-below-v2}.
However, this implies $q_2 \leq u_i$ in $T$, and hence $q_2 \leq q_i$ since
$q_2 < b_1^T$, a contradiction.

Now consider the edge $\p(q_2)q_2$ in $T$ and let $B(\p(q_2))\cap B(q_2)=\set{c,d}$.
In the following, we aim to show that the relevant part of $T$ essentially looks like in Figure~\ref{fig:q2-q3}, and consequently that the relations $a_i\leq b_{i+1}$ and $a_2\leq b_3$ have to hit $\{c,d\}$.
From this observation we will obtain our final contradiction.

Using \eqref{eq:q2-ai} and that $q_2 \leq b_{i+1}^T$ in $T$ (see~\eqref{eq:qi-q2}),
we deduce that the path from $a_i^T$ to $b_{i+1}^T$ in $T$ goes through this edge.
Thus the relation $a_i\leq b_{i+1}$ hits $\set{c,d}$.
Without loss of generality
\begin{equation}
 a_i\leq c\leq b_{i+1}\label{eq:c-ineq}
\end{equation}
in $P$.
To see that $a_2\leq b_3$ also hits $\{c,d\}$ we first show that $$q_2 \not \leq b_3^T$$ in $T$.
For this suppose $q_2 \leq b_3^T$ in $T$. Then $q_2$ and $u_3$ are comparable in $T$,
and thus $q_2 < u_3$ in $T$ by~\ref{item:uj-not-below-v2}. Since $q_2 < b_1^T$ in $T$,
it follows that $q_2 \leq u_3\wedge b_1^T=q_3$ in $T$, contradicting our assumption that $q_3<q_2$ in $T$.

So we have $q_2 \not \leq b_3^T$, and since $q_3<q_2 \leq u_2 < a_2^T$ in $T$, we deduce that the path connecting $a_2^T$ to $b_3^T$ in $T$
also includes the edge $\p(q_2)q_2$. Thus the relation $a_2\leq b_3$  indeed hits $\set{c,d}$.
It cannot hit $c$ because otherwise $a_2\leq c\leq b_{i+1}$ (by~\eqref{eq:c-ineq}), which is not possible since $i \neq 2$.
Hence we have
\begin{equation}
 a_2\leq d\leq b_3\label{eq:a2-d-b3}
\end{equation}
in $P$.
Now, the relation $a_0\leq b_2$ clearly hits $\set{c,d}$ as well,
but this is not possible as this implies $a_i\leq c\leq b_2$ (using~\eqref{eq:c-ineq}) or  $a_2\leq d\leq b_2$ in $P$ (using~\eqref{eq:a2-d-b3}), a contradiction in both cases.
This concludes the proof of \ref{item:v2-below-v3}.
\end{proof}

\begin{figure}[t]
 \centering
 \includegraphics[scale=1.0]{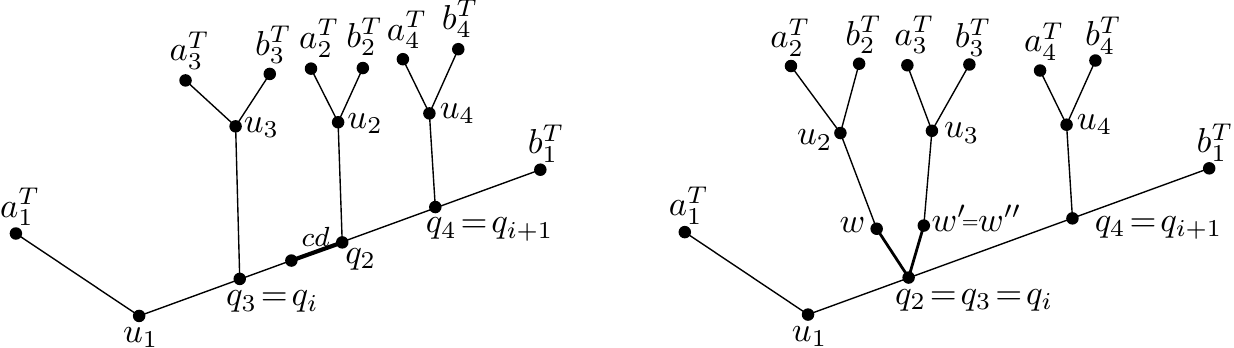}
 \caption{Left: Situation in Claim~\ref{claim:q2-q3} with $i=3$ (under the assumption that $q_3<q_2$ in $T$). Right: Possible situation in Claim~\ref{claim:q2-not-q3} with $i=3$ (under the assumptions that $q_2=q_3$ and $w\neq w'$).}
 \label{fig:q2-q3}
\end{figure}

The following claim is a strengthening of Claim~\ref{claim:q2-q3}~\ref{item:v2-below-v3}.
\begin{claim}\label{claim:q2-not-q3}
 Let  $\Sigma \in {\bf \Sigma}(\nu_{14})$ and suppose that $\set{(a_i,b_i)}_{i=1}^k$ is a strict alternating cycle
 in $\MM(P, \nu_{14}, \Sigma)$ with root $(a_1, b_1)$.
 Let $u_i$ denote $u_{a_ib_i}$ and let $q_i:= u_i\wedge b_1^T$, for each $i\in \{1,2, \dots, k\}$.
 Then
 $$u_1<q_2<q_3<b_1^T \text{ in } T.$$
\end{claim}
\begin{proof}
Using Claim~\ref{claim:q2-q3}~\ref{item:v2-below-v3} we only need to show that $q_2\neq q_3$.
Suppose to the contrary that we have $q_2=q_3$.
Recall  that $q_2< u_2$ in $T$, by Claim~\ref{claim:q-point}.
By Claim~\ref{claim:q2-q3}~\ref{item:uj-not-below-v2} we cannot have $u_3=q_3$ since $q_2 = q_3$.
Hence we also have $q_3< u_3$ in $T$.

Now, let $w$ be the neighbor of $q_2$ on the path from $q_2$ to $u_2$ in $T$, and let $w'$ be the
neighbor of $q_3$ on the path from $q_3=q_2$ to $u_3$ in $T$.
Using that $q_2<u_2$ and $q_3<u_3$ in $T$ we deduce that
\begin{equation}
 q_2<w\leq u_2 \quad\text{and}\quad q_3<w'\leq u_3\label{eq:w-w'}
\end{equation}
 in $T$.
Let us first suppose that $w=w'$.
Let $i$ be the largest index in $\{3,4, \dots, k\}$ such that $w\leq u_i$ in $T$.
We claim that $$w\not\leq b_{i+1}^T$$ in $T$ (taking indices cyclically, as always).
If $i < k$ this is because $w \leq b_{i+1}^T$ would
imply  $w\not\leq a_{i+1}^T$ (since $w\not \leq u_{i+1}$ in $T$), and thus
$u_{i+1}\leq q_2$ in $T$, contradicting~\ref{item:uj-not-below-v2} of Claim~\ref{claim:q2-q3}.
If $i = k$ this is because $w \leq b_{1}^T$
together with $w\leq u_2$ (by~\eqref{eq:w-w'}) would imply $w\leq u_2\wedge b_1^T=q_2$ in $T$, a contradiction.

Now, since $w\not\leq b_{i+1}^T$ in $T$
we deduce that the path from $a_i^T$ to $b_{i+1}^T$ includes the edge $q_2w$.
Let $B(w)\cap B(q_2)=\set{c,d}$.
Then the relations $a_i\leq b_{i+1}$ and $a_1\leq b_2$ both hit $\set{c,d}$, but not the same element
(as otherwise $a_i\leq b_2$ in $P$, which cannot be since $i\neq 1$).
Say we have
$$a_i\leq c\leq b_{i+1}\quad\text{and}\quad a_1\leq d\leq b_2$$
in $P$.
Observe that the path from $r$ to $b_i^T$ in $T$ goes through the edge $q_2w$ as well, and hence $a_0 \leq  b_i$ also hits $\set{c,d}$.
Thus $c\leq b_i$ or $d\leq b_i$ in $P$.
In the first case we obtain $a_i\leq b_i$ and in the second  $a_1\leq b_{i}$, a contradiction in each case.
This closes the case $w=w'$.

Finally, assume $w\neq w'$. Let $B(q_2)=\set{c, d, e}$.
Let $i$ be the largest index in $\{3,4, \dots, k\}$ such that $q_i=q_2$.
By Claim~\ref{claim:q2-q3}~\ref{item:uj-not-below-v2} we know that $q_i\neq u_i$ (in particular $u_i\not<b_1^T$ in $T$), and thus $q_2 = q_i < u_i$ in $T$.
Let $w''$ be the neighbor of $q_2$ on the path from $q_2$ to $u_i$ in $T$.
Note that we must have
\begin{equation}
 q_2<w''\leq u_i\quad\text{and}\quad w''\inc b_1^T\label{eq:q2-w''}
\end{equation}
in $T$.
Next we show that $$w''\not\leq b_{i+1}^T$$ in $T$.
Suppose that this is not true, and let us consider the case $i < k$ first.
Then $w'' \leq b_{i+1}^T$ would imply $w'' \leq a_{i+1}^T$ as otherwise $u_{i+1}\leq q_2$ in $T$, contradicting~\ref{item:uj-not-below-v2} of Claim~\ref{claim:q2-q3}.
Moreover, this yields  $w'' \leq u_{i+1}$ in $T$.
However, combined with the fact that $q_2 \leq b_{1}^T$ and $w'' \not \leq b_{1}^T$ in $T$ (by~\eqref{eq:q2-w''}) this implies $q_{i+1} = q_2$, contradicting the choice of $i$.
If $i = k$ then $w'' \leq b_{1}^T$ together with $w'' \leq u_i$ (see~\eqref{eq:q2-w''})
would imply $w'' \leq q_i$, again a contradiction.

So we indeed have $w''\not\leq b_{i+1}^T$ in $T$ (for an example illustrating this situation with $i=3$ see Figure~\ref{fig:q2-q3} on the right), and from this it follows that the $a_i^T$--$b_{i+1}^T$ path in $T$ includes the node $q_2$.
Observe that so does the $a_1^T$--$b_2^T$ path (because $u_1 < w_1 \leq q_2 <  b_2^T$ in $T$ by Claim~\ref{claim:orange-tree})
and the $a_2^T$--$b_3^T$ path (because $w \neq w'$).
Hence the three relations $a_1\leq b_2$, $a_2\leq b_3$ and $a_i\leq b_{i+1}$ all hit $B(q_2)=\set{c,d,e}$. Clearly, no element in $B(q_2)$ is
hit by two of these.
In other words, each element of $B(q_2)$ is greater or equal to $a_1$, $a_2$, or $a_i$ in $P$.

Furthermore, the paths from $r$ to $b_1^T$, $b_2^T$ and $b_3^T$ in $T$ all include the edge $\p(q_2)q_2$.
Hence two of the three relations $a_0 \leq b_1$, $a_0 \leq b_2$, $a_0 \leq b_3$ hit the same element in $B(\p(q_2)) \cap B(q_2) \subsetneq \set{c, d, e}$.
It follows that one element of the set $\set{a_1,a_2,a_i}$ is below two different elements of $\set{b_1,b_2,b_3}$ in $P$,
which contradicts the assumption that the alternating cycle is strict. This concludes the proof.
\end{proof}

Let  $\Sigma \in {\bf \Sigma}(\nu_{14})$ and suppose that $\set{(a_i,b_i)}_{i=1}^k$ is a strict alternating cycle
in $\MM(P, \nu_{14}, \Sigma)$ with root $(a_1, b_1)$.
Let $u_i$ denote $u_{a_ib_i}$ for each $i\in \{1,2, \dots, k\}$, and let $q_2:= u_2\wedge b_1^T$.
In the following claims we will need to consider three specific neighbors of the node $q_2$ in $T$,
namely, the neighbors of $q_2$ on the  $q_2$--$u_1$ path,  the $q_2$--$u_2$ path, and the $q_2$--$b_1^T$ path in $T$.
Let us denote these nodes by $\down{q_2}$, $\side{q_2}$ and $\up{q_2}$,  respectively.
By Claims~\ref{claim:orange-tree} and \ref{claim:q-point}, $\down{q_2}$, $\side{q_2}$ and $\up{q_2}$ are well defined and distinct.

The following claim is illustrated in Figure~\ref{fig:q-struct}.

\begin{figure}[t]
 \centering
 \includegraphics[scale=1.2]{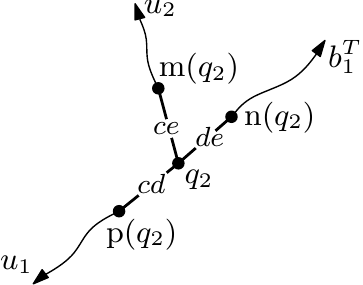}
 \caption{Illustration of Claim~\ref{claim:q2-struct}.}
 \label{fig:q-struct}
\end{figure}

\begin{claim}\label{claim:q2-struct}
Let  $\Sigma \in {\bf \Sigma}(\nu_{14})$ and suppose that $\set{(a_i,b_i)}_{i=1}^k$ is a strict alternating cycle
in $\MM(P, \nu_{14}, \Sigma)$ with root $(a_1, b_1)$.
Let $u_i$ denote $u_{a_ib_i}$ and let $q_i:= u_i\wedge b_1^T$, for each $i\in \{1,2, \dots, k\}$.
Then the elements of $B(q_2)$ can be written as $B(q_2)=\set{c,d,e}$ in such a way that
\begin{enumerate}
\item\label{item:bag-down} $B(q_2)\cap B(\down{q_2})=\set{c,d}$;
\item\label{item:bag-side} $B(q_2)\cap B(\side{q_2})=\set{c,e}$;
\item\label{item:bag-up} $B(q_2)\cap B(\up{q_2})=\set{d,e}$;
\end{enumerate}
and so that in $P$ we have
\begin{enumerate}
\setcounter{enumi}{3}
\begin{minipage}{.2\textwidth}
\item\label{item:paths-around-v2} \enskip
\begin{minipage}{.2\textwidth}
\begin{align*}
&a_1\leq c\leq b_2; \\
&a_0 \leq d\leq b_1; \\
&a_2\leq e\leq b_3;
\end{align*}
\end{minipage}
\end{minipage}
\hspace{.15\textwidth}
\begin{minipage}{.2\textwidth}
\item\label{item:cde-comparabilities} \enskip
\begin{minipage}{.13\textwidth}
\begin{align*}
&c\nleq d; \\
&e\nleq d; \\
&c\inc e;
\end{align*}
\end{minipage}
\end{minipage}
\hspace{.1\textwidth}
\begin{minipage}{.2\textwidth}
\item\label{item:ab-cde-comparabilities} \enskip
\begin{minipage}{.15\textwidth}
\begin{align*}
&a_2 \nleq c; \\[.4ex]
&a_2\nleq d.
\end{align*}
\end{minipage}
\end{minipage}
\end{enumerate}
\end{claim}

\begin{proof}
By Claims~\ref{claim:orange-tree}--\ref{claim:q2-not-q3}
we know that $u_1<w_1 \leq q_2<u_2<\set{a_2^T,b_2^T}$ and $q_2<q_3<\set{b_1^T,a_3^T,b_3^T}$ in $T$.
Thus the $a_1^T$--$b_2^T$ path in $T$ goes through the nodes $\down{q_2}$, $q_2$, and $\side{q_2}$;
the $r$--$b_1^T$ path goes through $\down{q_2}$, $q_2$, and $\up{q_2}$, and
the $a_2^T$--$b_3^T$ path goes through $\side{q_2}$, $q_2$, and $\up{q_2}$.
It follows that the corresponding three relations $a_1\leq b_2$, $a_0 \leq b_1$ and $a_2\leq b_3$ in $P$
hit respectively the two sets $B(q_2)\cap B(\down{q_2})$ and $B(q_2)\cap B(\side{q_2})$;
the two sets $B(q_2)\cap B(\down{q_2})$ and $B(q_2)\cap B(\up{q_2})$,  and
the two sets $B(q_2)\cap B(\side{q_2})$ and $B(q_2)\cap B(\up{q_2})$.
Clearly, no element of $B(q_2)$ is hit by two of these three relations.
It follows that the elements of $B(q_2)$ can be written as $B(q_2)=\set{c,d,e}$ in such a way that
properties~\ref{item:bag-down}-\ref{item:paths-around-v2} hold.
The remaining two properties~\ref{item:cde-comparabilities} and~\ref{item:ab-cde-comparabilities} are immediate consequences of these.
\end{proof}

For each  $\Sigma \in {\bf \Sigma}(\nu_{14})$ we define a corresponding directed graph $\hat K_{\Sigma}$ on the set
$\MM(P, \nu_{14}, \Sigma)$ similarly as in Section~\ref{sec:third_leaf}:
Given two distinct pairs $(a_1,b_1), (a_2,b_2) \in \MM(P, \nu_{14}, \Sigma)$, there is an arc from
$(a_1,b_1)$ to $(a_2,b_2)$ in $\hat K_{\Sigma}$ if and only if there is a strict alternating cycle $\{(a_i',b_i')\}_{i=1}^k$ in  $\MM(P, \nu_{14}, \Sigma)$
with root $(a_1',b_1')$ which is such that $(a_1',b_1')=(a_1,b_1)$ and $(a_2',b_2')=(a_2,b_2)$.
In the latter case, we say that the arc $f$ {\em is induced by} the strict alternating cycle $\{(a_i',b_i')\}_{i=1}^k$.

Note that there could possibly be different strict alternating cycles inducing the same arc in $\hat K_{\Sigma}$.
Observe also that if $\{(a_i,b_i)\}_{i=1}^k$ is a strict alternating cycle inducing an arc in $\hat K_{\Sigma}$
then $(a_1, b_1)$ is always the root of the cycle (by the definition of `inducing').

For each arc $f= ((a_1,b_1), (a_2,b_2))$ of $\hat K_{\Sigma}$, define the corresponding three nodes of $T$:
\begin{align*}
&u^-(f):=u_{a_1b_1}; \\[.3ex]
&u^+(f):=u_{a_2b_2}; \\[.3ex]
&q(f) :=u_{a_2b_2}\wedge b_1^T.
\end{align*}
Observe that $$u^-(f)<w_{a_1b_1} \leq q(f)<u^+(f)$$ in $T$ by Claims~\ref{claim:orange-tree} and~\ref{claim:q-point}.
This will be used repeatedly in what follows.

\begin{claim}\label{claim:same-source}
For each  $\Sigma \in {\bf \Sigma}(\nu_{14})$, any two arcs $f, g$ in $\hat K_{\Sigma}$ sharing the same source
satisfy $q(f)=q(g)$.
\end{claim}
\begin{proof}
Assume to the contrary that $q(f)\neq q(g)$. Let $(a_1, b_1) \in \MM(P, \nu_{14}, \Sigma)$ denote the source of the two arcs $f$ and $g$.
By definition $q(f) < b_1^T$ and $q(g) < b_1^T$ in $T$. Thus in particular $q(f)$ and $q(g)$ are comparable in $T$,
say without loss of generality $q(f)<q(g)$.
Hence we have  $u_{a_1b_1} <w_{a_1b_1} \leq q(f) < q(g) < b_1^T$ in $T$.

Let $(a_2, b_2), (a_2', b_2') \in \MM(P, \nu_{14}, \Sigma)$ denote the targets of arcs $f$ and $g$, respectively.
Let $(a_3, b_3), \dots, (a_k, b_k) \in \MM(P, \nu_{14}, \Sigma)$ be such that $\{(a_i,b_i)\}_{i=1}^k$ is  a strict alternating cycle inducing $f$.
Write the elements of $B(q(f))$ as $B(q(f))=\{c,d,e\}$ as in Claim~\ref{claim:q2-struct} when applied to the latter cycle.
Then the paths from $a_1^T$ to $b_2'^T$ and from $r$ to $b_1^T$ in $T$ both include the three nodes $\down{q(f)}$, $q(f)$, and $\up{q(f)}$.
Hence, the two relations $a_1\leq b_2'$ and $a_0 \leq b_1$ in $P$ both hit the two sets
$B(q(f))\cap B(\down{q(f)})=\set{c,d}$  and $B(q(f))\cap B(\up{q(f)})=\set{d,e}$. On the other hand, each of $c, d, e$ is clearly hit
by at most one of these two relations. It follows that one relation hits $d$ and the other hits both $c$ and $e$, implying
$c\leq e$ in $P$ by Observation~\ref{obs:hitting-vertex-or-edge}.
However, this contradicts $c\inc e$ in $P$ (cf.\  property~\ref{item:cde-comparabilities} of Claim~\ref{claim:q2-struct}).
\end{proof}

The following claim is similar to the previous one.

\begin{claim}\label{claim:same-sink}
For each  $\Sigma \in {\bf \Sigma}(\nu_{14})$, any two arcs $f, g$ in $\hat K_{\Sigma}$ sharing the same target
satisfy $q(f)=q(g)$.
\begin{proof}
Assume to the contrary that $q(f)\neq q(g)$. Let $(a_2, b_2) \in \MM(P, \nu_{14}, \Sigma)$ denote the common target of the two arcs $f$ and $g$.
We have $q(f) \leq u_{a_2b_2}$ and $q(g) \leq u_{a_2b_2}$ in $T$. Thus $q(f)$ and $q(g)$ are comparable in $T$, say  $q(f)<q(g)$
in $T$.

Applying Claim~\ref{claim:q2-struct} on a strict alternating cycle inducing $f$, we see that there exists
an element $e \in B(q(f))$ such that $a_2 \leq e$ in $P$.
Since $q(f)<q(g) \leq u_{a_2b_2} < a_2^T$ in $T$, the path from $q(f)$ to $a_2^T$ in $T$ goes through $\down{q(g)}$ and $q(g)$.
Thus the relation  $a_2 \leq e$ hits $B(q(g))\cap B(\down{q(g)})$, and hence
$a_2 \leq s$ in $P$ for some $s\in B(q(g))\cap B(\down{q(g)})$.
However, applying Claim~\ref{claim:q2-struct} on a strict alternating cycle inducing $g$ this time, we deduce that
$a_2 \nleq t$ in $P$ for each $t\in B(q(g))\cap B(\down{q(g)})$ (cf. property~\ref{item:ab-cde-comparabilities}), and therefore in particular $a_2 \nleq s$ in $P$, a contradiction.
\end{proof}
\end{claim}

\begin{claim}\label{claim:u-q-order1}
For each  $\Sigma \in {\bf \Sigma}(\nu_{14})$ and any two arcs $f, g$ in $\hat K_{\Sigma}$, we {\em neither}
have
\[
q(f)<q(g)<u^+(f) \leq u^+(g)
\]
{\em nor}
\[
q(f)<q(g)<u^+(g)\leq u^+(f)
\]
in $T$.
\end{claim}
\begin{proof}
Let $(a_2,b_2)$ and $(a_2',b_2')$ denote the targets of $f$ and $g$, respectively. Arguing by contradiction,
assume that at least one of the two inequalities holds.
Then we have
\begin{equation}
q(f)\leq \down{q(g)}<q(g)<\side{q(g)} \leq \set{u_{a_2b_2},u_{a_2'b_2'}}\label{eq:qf-qg}
\end{equation}
in $T$.

Now consider a strict alternating cycle inducing $g$
and write the elements of $B(q(g))$ as $B(q(g))=\{c,d,e\}$ as in Claim~\ref{claim:q2-struct} when applied
to the latter cycle. See Figure~\ref{fig:q-struct2} for an illustration.
\begin{figure}[t]
 \centering
 \includegraphics{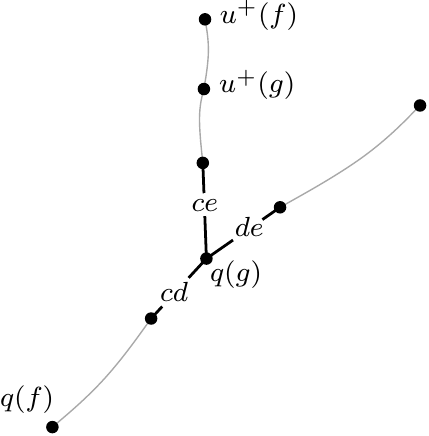}
 \caption{Illustration of the proof of Claim~\ref{claim:u-q-order1}.}
 \label{fig:q-struct2}
\end{figure}
By this claim we have
\begin{equation}
 c\leq b_2'\quad\text{and}\quad a_2'\leq e\quad\text{and}\quad e\nleq d\label{eq:ced}
\end{equation}
in $P$. Applying
Claim~\ref{claim:q2-struct} on a strict alternating cycle inducing $f$, we also deduce that
there exists $s\in B(q(f))$ such that $a_2\leq s$ in $P$.

Given that by~\eqref{eq:qf-qg} the path from $a_2^T$ to $q(f)$ goes first through $u_{a_2b_2}$ and then through the three nodes $\side{q(g)}$, $q(g)$ and $\down{q(g)}$ in $T$,
it follows that the relation $a_2\leq s$ hits both $B(q(g))\cap B(\side{q(g)}) = \set{c,e}$ and $B(q(g))\cap B(\down{q(g)})=\set{c,d}$.
If it {\em did not} hit $c$, then it would hit both $d$ and $e$, and we would have $e \leq d$ in $P$
by Observation~\ref{obs:hitting-vertex-or-edge}, which is not possible by~\eqref{eq:ced}.
Thus  $a_2\leq s$ hits $c$, that is, $a_2\leq c \leq s$ in $P$.
Together with~\eqref{eq:ced} this implies $$a_2\leq c \leq b_2'$$ in $P$, and we also deduce $(a_2,b_2)\neq (a_2',b_2')$.

Now, the path from $r$ to $b_2^T$ in $T$ goes through $q(g)$ and $\side{q(g)}$, and thus $a_0 \leq b_2$ hits $\set{c,e}$.
It cannot hit $c$, as otherwise $a_2\leq c\leq b_2$ in $P$.
Hence  $a_0 \leq b_2$ hits $e$, and we have $a_0 \leq e\leq b_2$ in $P$, which by~\eqref{eq:ced} implies $$a_2'\leq e \leq b_2.$$
It follows that $(a_2,b_2), (a_2',b_2')$ is an alternating cycle of length $2$, which is a contradiction since
there is no such cycle in $\MM(P, \nu_{14}, \Sigma)$.
\end{proof}

For the next claim let us recall that given a pair $(a,b)\in\MM(P,\nu_{14},\Sigma)$, the elements of $B(u_{ab})$ are labeled with $x_{ab},y_{ab},z_{ab}$, and it holds that $B(u_{ab})\cap B(\p(u_{ab}))=\{x_{ab},y_{ab}\}$,
$$a\leq x_{ab}\not\leq b,\quad\text{and}\quad a\not\leq y_{ab}\leq b$$
in $P$.

\begin{claim}\label{claim:u-q-order2}
For each  $\Sigma \in {\bf \Sigma}(\nu_{14})$, {\em no two} arcs $f, g$ in $\hat K_{\Sigma}$ satisfy
\[
u^-(f)\leq q(g)<u^+(g)\leq q(f).
\]
in $T$.
\end{claim}
\begin{proof}
Assume to the contrary that the inequality holds.
Let $\{(a_i,b_i)\}_{i=1}^k$ be a strict alternating cycle inducing $f$ and let $\{(a'_i,b'_i)\}_{i=1}^\ell$ be one inducing $g$.
Let $u_i:= u_{a_ib_i}$, $q_i:= u_{a_ib_i} \wedge b_1^T$,
$x_i:=x_{a_ib_i}$, $y_i:=y_{a_ib_i}$, and $z_i:=z_{a_ib_i}$
for each $i\in \set{1, \dots, k}$, and
let $u_i':= u_{a_i'b_i'}$, $q_i':= u_{a_i'b_i'} \wedge b_1'^T$,
$x_i':=x_{a_i'b_i'}$, $y_i':=y_{a_i'b_i'}$, and $z_i':=z_{a_i'b_i'}$
for each $i\in \set{1, \dots, \ell}$.
Thus $u^-(f)=u_1$, $q(f)=q_2$, $u^+(g)=u_{2}'$, $q(g)=q_2'$, and
\begin{equation}
u_1\leq q'_2<u'_2\leq q_2 < \set{b_1^T, b_2^T}\label{eq:u1-q2'}
\end{equation}
in $T$ by our assumption.
See Figure~\ref{fig:v-u-path} for an illustration of the situation.
\begin{figure}[t]
 \centering
 \includegraphics{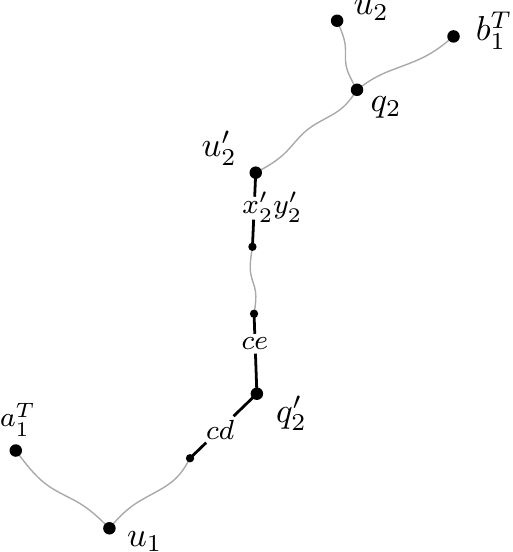}
 \caption{Illustration of the proof of Claim~\ref{claim:u-q-order2}.}
 \label{fig:v-u-path}
\end{figure}

The $a_1^T$--$b_2^T$ path and the $r$--$b_1^T$ path both go through $\p(u'_2)$ and $u'_2$ in $T$.
Thus the two relations $a_1\leq b_2$ and $a_0 \leq b_1$ both hit $B(\p(u'_2)) \cap B(u_2') = \set{x'_2,y'_2}$, and clearly neither
of $x'_2,y'_2$ is hit by both relations.
We cannot have $a_1\leq y'_2\leq b_2$ and $a_0 \leq x'_2\leq b_1$ in $P$,
because otherwise we would have $a_1\leq y'_2\leq b'_2$ and $a'_2\leq x'_2\leq b_1$ in $P$,
implying that $(a_1,b_1), (a_2',b_2')$ is an alternating cycle of length $2$ in $\MM(P, \nu_{14}, \Sigma)$, a contradiction.
Hence we have
\begin{equation}
a_1\leq x'_2\leq b_2 \quad \text{ and }\quad a_0 \leq y'_2\leq b_1\label{eq:x2'-y2'}
\end{equation}
in $P$.

Let us denote the elements in $B(q'_2)$ as $B(q'_2)=\set{c,d,e}$ as in Claim~\ref{claim:q2-struct}
when applied to the strict alternating cycle $\{(a'_i,b'_i)\}_{i=1}^\ell$.
Then we have $a_0 \leq d\leq b'_1$ in $P$, as well as
$a'_2\leq e\leq b'_3$ and  $a'_1\leq c\leq b'_2$.
Given that the relation $a'_2\leq e$ hits $B(\p(u'_2)) \cap B(u_2') = \set{x'_2,y'_2}$, and that it clearly cannot hit $y'_2$ because $y'_2\leq b'_2$ in $P$,
we have $a'_2\leq x'_2\leq e$ in $P$. Similarly, $c\leq b'_2$ hits $\set{x'_2,y'_2}$ as well and cannot hit $x'_2$ because $a'_2\leq x'_2$ in $P$,
hence $c\leq y'_2\leq b'_2$ in $P$.
Summarizing, we have
\begin{align}
 a'_2\leq x'_2\leq e\leq b'_3,\label{eq:a'2}\\
 a'_1\leq c\leq y'_2\leq b'_2,\label{eq:a'1}\\
 a_0 \leq d\leq b'_1\label{eq:a0-b'1}
\end{align}
in $P$.

Recall that $u_1\leq q_{2}'$ in $T$ by~\eqref{eq:u1-q2'}.
We split the rest of the argument into two cases and start with the case that $u_1< q_{2}'$ in $T$.
Then the path from $a_1^T$ to $u'_2$ goes through $\down{q'_2}$ and $q'_2$.
It follows that the relation $a_1\leq x'_2$ hits $B(\down{q'_2})\cap B(q'_2) = \set{c,d}$.
If it hits $c$ then $c \leq x'_2$ in $P$, which implies $a'_1\leq c\leq x'_2\leq e\leq b'_3$ by~\eqref{eq:a'1} and \eqref{eq:a'2}, a contradiction to the assumption that we deal with a strict alternating cycle of length at least $3$.
Hence, we have $a_1\leq d\leq x'_2$.
However, using~\eqref{eq:a0-b'1} we obtain $a_1\leq d\leq b'_1$ in $P$, and since $a'_1\leq c\leq y'_2\leq b_1$ in $P$ (by combining~\eqref{eq:a'1} and \eqref{eq:x2'-y2'}) this implies that $(a_1,b_1)\neq(a_1',b_1')$
and therefore that $(a_1,b_1), (a_1',b_1')$ is an alternating cycle of length $2$ in $\MM(P, \nu_{14}, \Sigma)$, a contradiction.

It remains to consider the case that $u_1=q'_2$ in $T$.
Then the path from $a_1^T$ to $u'_2$ in $T$ goes through $q'_2$ and $\side{q'_2}$.
Thus $a_1\leq x'_2$ hits $B(q'_2) \cap B(\side{q'_2}) = \set{c,e}$.
The relation $a_1\leq x'_2$ cannot hit $c$, for the same reason as in the previous paragraph.
Hence $a_1\leq e\leq x'_2$ in $P$, implying that $x'_2=e$ by~\eqref{eq:a'2}.

Now, observe that $e\in B(u_1)$ since $u_1=q'_2$.
Given that $e\not\in B(q'_2) \cap B(\down{q'_2}) = \set{c,d} = B(u_1)\cap B(\p(u_1))=\set{x_1,y_1}$, we conclude $x'_2=e=z_1$.
However, in the coloring $\phi$ we have $\phi(x'_2) = \phi(x_1) \neq \phi(z_1)$ since
$\q{sign:H-coloring-second}(a_1,b_1)=\q{sign:H-coloring-second}(a_2',b_2')$, contradicting $x'_2=z_1$.
This concludes the proof.
\end{proof}

\begin{claim}\label{claim:equal-q}
Let  $\Sigma \in {\bf \Sigma}(\nu_{14})$ and suppose that $f_1,f_2,g_1,g_2$ are arcs of $\hat K_{\Sigma}$ satisfying
\begin{itemize}
 \item $u^+(f_1)=u^-(f_2)$
 \item $u^+(g_1)=u^-(g_2)$
 \item $q(f_2)=q(g_2)$.
\end{itemize}
Then it also holds that $q(f_1)=q(g_1)$.
\end{claim}
\begin{proof}
Recall that $u^-(f) < q(f) < u^+(f)$ for every arc $f$ of $\hat K_{\Sigma}$ (by Claims~\ref{claim:orange-tree} and~\ref{claim:q-point}).
It follows from the assumptions that
\begin{equation}
q(g_1) < u^+(g_1)=u^-(g_2)<q(g_2)\label{eq:qg1}
\end{equation}
and
\begin{equation}
q(f_1)<u^+(f_1)=u^-(f_2)<q(f_2)\label{eq:qf1}
\end{equation}
in $T$.
Thus $q(g_1)$ and $q(f_1)$ are comparable in $T$.
Arguing by contradiction, suppose that $q(f_1) \neq q(g_1)$.
Using symmetry, we may assume without loss of generality $q(f_1) < q(g_1)$ in $T$.

Since $q(g_1)<q(g_2)=q(f_2)$ by~\eqref{eq:qg1}  and $u^+(f_1)= u^-(f_2) < q(f_2)$ in $T$ by~\eqref{eq:qf1},  the two nodes $q(g_1)$ and  $u^+(f_1)$ are also comparable in $T$.

First suppose that $q(g_1)<u^+(f_1)$ in $T$.
Then observe that the two nodes $u^+(g_1)$ and $u^+(f_1)$ are comparable in $T$ since
$u^+(g_1)= u^-(g_2) < q(g_2) = q(f_2)$ and $u^+(f_1) < q(f_2)$ in $T$ (by~\eqref{eq:qg1} and \eqref{eq:qf1}).
Hence we have $q(f_1) < q(g_1)<u^+(f_1)\leq u^+(g_1)$ or $q(f_1) < q(g_1)<u^+(g_1)\leq u^+(f_1)$ in $T$, neither of which is
possible by Claim~\ref{claim:u-q-order1}, a contradiction.

Next, assume that  $q(g_1)\geq u^+(f_1)$ in $T$.  We immediately obtain
$u^-(f_2)=u^+(f_1)\leq q(g_1)<u^+(g_1) < q(f_2)$ in $T$, which is forbidden by Claim~\ref{claim:u-q-order2},
again a contradiction.
\end{proof}

\begin{claim}\label{claim:k-sigma-bipartite_nu_14}
 The graph $\hat K_{\Sigma}$ is bipartite for each $\Sigma\in \mathbf{\Sigma}(\nu_{14})$.
\end{claim}
\begin{proof}
Suppose that there is an odd cycle  $C=\set{(a_i,b_i)}_{i=1}^{k}$ in the undirected graph underlying $\hat K_{\Sigma}$.
(Thus $C$ is not necessarily a directed cycle.)
For each $i \in \{1, \dots, k\}$, let $f_i$ be an arc between $(a_i,b_i)$ and $(a_{i+1},b_{i+1})$ in $\hat K_{\Sigma}$,  where indices
are taken cyclically as always.
If $f_i=((a_i,b_i),(a_{i+1},b_{i+1}))$, that is, $(a_i,b_i)$ is the source of $f_i$, we say that $f_i$ \emph{goes forward},
while if $f_i=((a_{i+1},b_{i+1}),(a_{i},b_{i}))$ we say that $f_i$ \emph{goes backward}.

We define a cyclically ordered sequence $S$ of arcs in $\{f_1, f_2, \dots, f_k\}$ as follows.
We start with $S=(f_1,\ldots,f_k)$.
It will be convenient to say that we go along $S$ in \emph{clockwise} order whenver we use the forward direction in $S$ (to not mix it up with forward and backward edges).
Now, we repeat the following modification until it is no longer possible:
If $S$ has size at least $5$ and there are two (cyclically) consecutive arcs $f,f'$ in clockwise order in $S$
with $f$ going forward and $f'$ going backward then remove both $f$ and $f'$ from $S$.

By construction, the resulting sequence $S$ has the following property: Either $S$ contains at least five arcs and
all arcs go in the same direction, or $S$ contains exactly three arcs.

We claim that during the above iterative process the cyclic sequence $S$ fulfills the following invariants at all times:
For any two consecutive arcs $f$ and $f'$ in clockwise order in $S$,
\begin{enumerate}
\item\label{inv:1} if $f$ and $f'$ both go forward then $q(f) < q(f')$ in $T$,
while if $f$ and $f'$ both go backward then $q(f) >q(f')$ in $T$;
\item\label{inv:2} if $f$ and $f'$ go in the same direction then there exist arcs $g,g'$ in $\hat K_{\Sigma}$ such that
\begin{itemize}
 \item $q(g)=q(f)$;
 \item $q(g')=q(f')$, and
 \item  $u^+(g)=u^-(g')$ if $f$ and $f'$ go forward, $u^-(g)=u^+(g')$ otherwise, and
\end{itemize}
\item\label{inv:3} if $f$ and $f'$ go in opposite directions then $q(f)=q(f')$.
\end{enumerate}

Note that~\ref{inv:2} implies~\ref{inv:1}. Indeed, suppose $f$ and $f'$ go forward (for the backward direction the argument is analogous).
Then take arcs $g$ and $g'$ witnessing~\ref{inv:2}. We have $q(f) = q(g) < u^+(g) = u^-(g') < q(g') = q(f')$.
(Recall that $u^-(g) < q(g) < u^+(g)$ for every arc $g$ of $\hat K_{\Sigma}$ by Claims~\ref{claim:orange-tree} and~\ref{claim:q-point}.)

First, we prove that the invariants hold at the beginning of the process, so for the sequence $(f_1,\ldots,f_k)$.
In order to prove~\ref{inv:2}, for each $i\in\{1,\dots, k\}$ take $g:=f_i$ and $g':=f_{i+1}$. Then clearly~\ref{inv:2} holds,
and property~\ref{inv:3} follows from Claims~\ref{claim:same-source} and~\ref{claim:same-sink}.

Next we show that the invariants hold after each modification step.  Consider thus the sequence $S$ just before a modification step,
and suppose that $S$ satisfied the required properties.
Let $f^0,f^1,f^2,f^3$ be the four consecutive arcs in $S$ in clockwise order which are such that
$f^1$ goes forward and $f^2$ goes backward.
After removing $f^1$ and $f^2$, the arcs $f^0$ and $f^3$ will become consecutive in $S$ (in clockwise order).
We only need to establish the invariants for the consecutive pair $f^0,f^3$, since all other consecutive pairs already satisfy them
by assumption.

Let us start with the case that $f^0$ and $f^3$ both go forward.
Since $f^1,f^2$ and $f^3$ alternate in directions, we get $q(f^1)=q(f^2)=q(f^3)$ by~\ref{inv:3}.
By~\ref{inv:2} and the fact that $f^0$ and $f^1$ go forward, there are arcs $g^0,g^1$ in $\hat K_{\Sigma}$ such that $q(g^0)=q(f^0)$, $q(g^1)=q(f^1)$ and $u^+(g^0)=u^-(g^1)$.
Now, since $q(g^1)=q(f^1)=q(f^3)$, the arcs $g^0,g^1$ also fulfill the conditions of~\ref{inv:2} for $f^0$ and $f^3$.

The case that both $f^0$ and $f^3$ go backward is symmetric to the previous one and is thus omitted.

Next, suppose that $f^0$ goes forward and $f^3$ goes backward.
Here we have to show that~\ref{inv:3} holds for $f^0$ and $f^3$.
Since $f^0$ and $f^1$ both go forward and $f^2$ and $f^3$ both go backward, by \ref{inv:2} there are arcs $g^0,g^1$ and $g^2, g^3$ in $\hat K_{\Sigma}$ such that $q(g^j)=q(f^j)$ for each $j \in \{0,1,2,3\}$, $u^+(g^0)=u^-(g^1)$, and $u^-(g^2)=u^+(g^3)$.
Using~\ref{inv:3} we deduce that $q(g^1)=q(f^1)=q(f^2)=q(g^2)$. Applying Claim~\ref{claim:equal-q} on the arcs
$g^0,g^1,g^3,g^2$ (in this order), we conclude $q(f^0)=q(g^0)=q(g^3)=q(f^3)$, as desired.

Finally, assume that $f^0$ goes backward and $f^3$ goes forward.
Again we have show that~\ref{inv:3} holds for $f^0, f^3$.
But in this case the four directions of $f^0,f^1,f^2,f^3$ alternate.
It follows that $q(f^0)=q(f^1)=q(f^2)=q(f^3)$ by~\ref{inv:3}.

Now that the above invariants of $S$ have been established, let us go back to the
final sequence $S$ resulting from the modification process.
We claim that there are always two consecutive arcs going in opposite directions in $S$.
Indeed, if not then they either all go forward or all go backward. In the first case $q(f)< q(f')$ in $T$
for any two consecutive arcs $f, f'$ in clockwise order in $S$ by~\ref{inv:1},
while in the second case $q(f) > q(f')$ in $T$ for any two such arcs $f, f'$. However, neither of these two situations
can occur in a circular sequence.

This shows in particular that the modification process results in a sequence $S$ of size $3$, say $S=(f^1,f^2,f^3)$.
We may suppose without loss of generality that $f^1$ and $f^2$ go in the same direction and $f^3$ in the other
(since the sequence $S$ can always be shifted cyclically to ensure this property).
This implies $q(f^1) \neq q(f^2)$ by (i), $q(f^2)=q(f^3)$ by~\ref{inv:3}, and $q(f^3)=q(f^1)$ by~\ref{inv:3}.
This is a contradiction, which concludes the proof.
\end{proof}

Using Claim~\ref{claim:k-sigma-bipartite_nu_14} we let
$\psi_{14,\Sigma}\colon \MM(P,\nu_{14},\Sigma)\to\set{1,2}$ be a $2$-coloring of $\hat K_{\Sigma}$,
for each $\Sigma\in\mathbf{\Sigma}(\nu_{14})$.
The function $\q{sign:colors_K_second}$ then records the color of a pair in this coloring:
\begin{center}
\medskip
\fbox{\begin{varwidth}{0.9\textwidth}
For each $\Sigma\in\mathbf{\Sigma}(\nu_{14})$ and each pair $(a,b)\in\MM(P,\nu_{14},\Sigma)$, we let
\[
 \q{sign:colors_K_second}(a,b):=\psi_{14,\Sigma}(a,b).
\]
\end{varwidth}}
\medskip
\end{center}

\subsection{Fourth leaf of \texorpdfstring{$\sigtree$: Node $\nu_{15}$}.}
\label{sec:a-goes-down}

It remains to verify that for each $\Sigma\in\mathbf{\Sigma}(\nu_{15})$ the set $\MM(P,\nu_{15},\Sigma)$ is reversible.
Recall that $\q{sign:coloring-2-cycles}$ ensures that there are no $2$-cycles in $\MM(P,\nu_{15},\Sigma)$ and that $\q{sign:colors_K_second}$ ensures that there are no strict alternating cycles of length at least $3$ in $\MM(P,\nu_{15},\Sigma)$.
It follows that $\MM(P,\nu_{15},\Sigma)$ is reversible.

This concludes the proof of Theorem~\ref{thm:technical}.

\section*{Acknowledgments}

This research was initiated during the workshop {\em Order and Geometry}
held at the Technische Universit\"at Berlin in August 2013.
We are grateful to the organizers and the other participants for providing
a very stimulative research environment.
We also thank Grzegorz Gutowski and Tomasz Krawczyk for many fruitful discussions at the early stage of this project.
Finally, we are much grateful to the referees for their many helpful remarks and suggestions,
which greatly improved the readability of the paper.

\bibliographystyle{plain}
\bibliography{treewidth-2}

\end{document}